\documentclass{article}
\usepackage{amssymb}
\usepackage{amsmath}
\usepackage{enumitem}
\usepackage{bbm}
\usepackage{bookmark}
\usepackage{tikz-cd}
\usepackage{hyperref}

\usepackage[titletoc]{appendix}
\usepackage{amsthm}
\usepackage{parskip}
\usepackage{amsfonts}
\usepackage{caption}
\usepackage{mathrsfs}
\usepackage{afterpage}
\usepackage{float}
\usepackage{multirow}
\usepackage{graphicx}
\usepackage{color}
\usepackage[margin=1.0in]{geometry}
\newcommand*{\SHom}{\mathscr{H}\kern -.5pt om}
\newcommand*{\SExt}{\mathscr{E}\kern -.5pt xt}
\newcommand{\blackqed}{\hfill\ensuremath{\blacksquare}}

\begin{document}
\title{$p$-curvature in non-commutative Hodge theory and the Kontsevich-Soibelman operad}
\date{}
\author{Zihong Chen}
\maketitle
\theoremstyle{definition}
\newtheorem{mydef}{Definition}[section]
\numberwithin{mydef}{section}
\newtheorem{rmk}[mydef]{Remark}
\newtheorem{conj}[mydef]{Conjecture}
\theoremstyle{plain}
\newtheorem{question}[mydef]{Question}
\theoremstyle{plain}
\newtheorem{cor}[mydef]{Corollary}
\newtheorem{ex}[mydef]{Example}
\newtheorem{lemma}[mydef]{Lemma}
\newtheorem{thm}[mydef]{Theorem}
\newtheorem{prop}[mydef]{Proposition}
\begin{abstract}
Let $\mathcal{C}$ be a differential $\mathbb{Z}/2$-graded category over $\mathbb{C}$. Its periodic cyclic homology $HH^{per}_*(\mathcal{C})$, when viewed as a vector bundle over the formal punctured disk, is equipped with a canonical connection $\nabla^{\mathcal{C}}_{\partial_t}$ called the Getzler-Gauss-Manin connection in the $t$-direction (or the categorical $t$-connection). Our main result is that when $\mathcal{C}$ is smooth and proper, this connection has a regular singularity at $t=0$ and quasi-unipotent monodromy, affirming a conjecture of Katzarkov-Kontsevich-Pantev \cite{KKP}. Our proof follows a reduction mod $p$ argument using a spreading out technique of To\"{e}n \cite{To} and a regularity criterion of Katz \cite{Ka1}. The main novelty is the proof of a multiplicative property of the $p$-curvature of $\nabla^{\mathcal{C}}_{\partial_t}$ through an interpretation in terms of the two-colored Kontsevich-Soibelman operad. \par\indent
We then explore two applications of the main result. First, we give an explicit description (under additional assumptions) of the non-commutative Hodge filtration on the periodic cyclic homology of a smooth proper d$(\mathbb{Z}/2)$g category, following a construction of Shklyarov \cite{Shk}. The second application, which is special to our particular method or proof, is an upper bound on the sizes of Jordan blocks of the monodromy of $\nabla^{\mathcal{C}}_{\partial_t}$, which simultaneously generalizes Scherk's local monodromy theorem for isolated hypersurface singularities \cite{Sche} and (partially) a recent result of Pomerleano-Seidel on the quantum connection of a closed monotone symplectic manifold \cite{PS2}. As a specialization, we show that the sizes of these Jordan blocks are bounded above by the diagonal dimension of $\mathcal{C}$ plus one. 

\end{abstract}

\def\acts{\curvearrowright}
\tableofcontents

\numberwithin{equation}{section}
\renewcommand{\theHequation}{\thesection.\arabic{equation}}
\section{Introduction}
\subsection{Main result and motivation}
This paper studies the interplay of topological operads and the arithmetic theory of differential equations, with an eye towards applications in non-commutative Hodge theory.\par\indent
To serve as motivation, we consider the following question. Fix $\mathcal{C}$ a differential $\mathbb{Z}/2$-graded (d($\mathbb{Z}/2$)g) category over $\mathbb{C}$. Its periodic cyclic homology $HH^{per}_*(\mathcal{C})$, when viewed as a vector bundle over the formal punctured disk, is equipped with a canonical connection called the \emph{Getzler-Gauss-Manin connection in the $t$-direction} (or the \emph{categorical $t$-connection}). Explicitly, let $(CC_*(\mathcal{C}),b)$ denote the cyclic bar complex and $B$ the Connes operator, then $HH^{per}_*(\mathcal{C})$ is computed by the chain complex $(CC_*(\mathcal{C})((t)),b+tB)$, where $t$ is a formal even variable (which we think of as the generator of $H^2(BS^1;\mathbb{C})$). At the chain level, the $t$-connection can be written as 
\begin{equation}\label{eq:categorical t connection intro}
\nabla^{\mathcal{C}}_{\partial_t}:=\frac{\partial}{\partial t}+\frac{\Gamma}{2t}+\frac{\iota_{d}}{2t^2},    
\end{equation}
where $\Gamma$ denotes the length operator on cyclic chains, $d$ denotes the differential of $\mathcal{C}$ viewed as a Hochschild cochain, and $\iota$ denotes the cyclic contraction operator in non-commutative calculus (cf. Section 2 for a more detailed introduction). \par\indent
When the $\mathbb{Z}/2$-grading on $\mathcal{C}$ can be lifted to a $\mathbb{Z}$-grading, the datum of the categorical $t$-connection is extraneous and fully subsumed into the standard $t$-filtration on $HH^{per}_*(\mathcal{C})$, cf. \cite[Section 3.4, 3.5]{Shk}. However, for a general d($\mathbb{Z}/2)$g category (also known as an \emph{nc-space}), the $t$-connection is a genuinely important ingredient in the study of its Hodge theoretic aspects \cite{KKP}\cite{Shk}. In fact, it should be viewed as part of the \emph{non-commutative Hodge filtration}, cf. Section 6.1 for more explanation. \par\indent
One of the complications concerning the categorical $t$-connection is its apparent irregularity at $t=0$ (for a formal classification of irregular connections, we refer the readers to e.g. \cite{Lev}\cite{Mal1}). For instance, unlike in the regular case, formal solutions to irregular connections may not admit any local meromorphic lifts. The following conjecture predicts, nonetheless, that in favorable situations this apparent second order pole is only a disguise.\\
\begin{conj}(Katzarkov-Kontsevich-Pantev)\label{thm:KKP conjecture}
For a smooth and proper d$(\mathbb{Z}/2)$g category $\mathcal{C}$, its categorical $t$-connection \eqref{eq:categorical t connection intro} has a regular singularity at $t=0$.    
\end{conj}
Explicitly, this means that there exists a formal gauge transformation (with poles in $t$ allowed) such that the gauge-transformed connection matrix of \eqref{eq:categorical t connection intro} has a simple pole at $t=0$. Conjecture \ref{thm:KKP conjecture} is partially motivated by homological mirror symmetry and appears as a small component of \cite[Conjecture 2.24 and Question 2.2.7]{KKP}. Indeed, the evidence so far towards this conjecture mainly comes from the following two geometric situations appearing in mirror symmetry\footnote{The apparent discrepancy between the `regular singular' property in Conjecture \ref{thm:KKP conjecture} and `unramified exponential type' in (i) and (ii) comes from the fact that $\mathrm{MF}(W)$ or $\mathrm{Fuk}(X)$ is most naturally considered as a curved d$(\mathbb{Z}/2)$g (or $A_{\infty}$) category, while the $\mathcal{C}$ appearing in Conjecture \ref{thm:KKP conjecture} is uncurved. For instance, in the context of (ii), to match the setting of Conjecture \ref{thm:KKP conjecture} we may take the uncurved category $\mathrm{Fuk}(X)_{\lambda}$ (for each eigenvalue $\lambda$ of $c_1(TX)\star$ separately) consisting of monotone Lagrangian branes with disk potential $\lambda$, cf. \cite{Sh}.}: 
\begin{enumerate}[label=\roman*)]
    \item when $\mathcal{C}=\mathrm{MF}(W)$ is the d$(\mathbb{Z}/2)$g category of matrix factorizations of a smooth function $W:\mathbb{C}^n\rightarrow \mathbb{C}$ with isolated singularities, its categorical $t$-connection can be identified with the (inverted) Fourier-Laplace dual of the Gauss-Manin system of $W$, which has regular singularities by the Monodromy Theorem \cite{Lan}; cf. Section 6.2 for more details. By general D-module theory this implies that the $t$-connection itself has \emph{unramified exponential type} \cite[Section 2.2]{PS1}, i.e. admits a decomposition into a finite direct sum
    \begin{equation}\label{eq:exponential type decomp}
    \bigoplus_{\lambda}(\mathbb{C}((t)),\frac{\partial}{\partial t}-\frac{\lambda}{t^2})\otimes \nabla^{reg,\lambda}_{\frac{\partial}{\partial t}},
    \end{equation}
    where each $\nabla^{reg,\lambda}_{\frac{\partial}{\partial t}}$ is regular singular, and $\lambda$ ranges over the critical values of $W$. 
    \item More recently, works of \cite{PS1}\cite{Che2} prove that the quantum connection of a closed monotone symplectic manifold $X$ (which roughly corresponds to setting $\mathcal{C}$ to be the monotone Fukaya category of $X$) is also of unramified exponential type, where the decomposition \eqref{eq:exponential type decomp} sums over the eigenvalues of the quantum multiplication by $c_1(TX)$. 
\end{enumerate}
The main result of this paper is the following, which in particular affirms Conjecture \ref{thm:KKP conjecture}.\\
\begin{thm}\label{thm:main theorem}
For a smooth and proper d$(\mathbb{Z}/2)$g category $\mathcal{C}$, its categorical $t$-connection \eqref{eq:categorical t connection intro} has a regular singularity at $t=0$. Moreover, the eigenvalues of its monodromy are roots of unity. 
\end{thm}
By the following lemma, it suffices to prove Theorem \ref{thm:main theorem} for d$(\mathbb{Z}/2)$g \emph{algebras}, which we will work exclusively with from now on. \\
\begin{lemma}\label{thm:reduction from category to algebra}
For a smooth proper d$(\mathbb{Z}/2)$g category $\mathcal{C}$, there is a smooth proper d$(\mathbb{Z}/2)$g algebra $\mathcal{A}$ together with an isomorphism of connections
\begin{equation}
(HH^{per}_*(\mathcal{C}),\nabla^{\mathcal{C}}_{\partial_t})\cong (HH^{per}_*(\mathcal{A}),\nabla^{\mathcal{A}}_{\partial_t}). 
\end{equation}
\end{lemma}
\emph{Proof}. By \cite[Lemma 2.5]{Efi}, the derived category of $\mathcal{C}$ has a compact generator $G$. Let $\mathcal{A}$ be the endomorphism d$(\mathbb{Z}/2)$g algebra of $G$, which is smooth and proper. Then there are Morita equivalences 
\begin{equation}
 \mathcal{C}\rightarrow \mathrm{Perf}(\mathcal{C}) \leftarrow \mathcal{A}.  
\end{equation}
The statement follows by the Morita invariance of periodic cyclic homology and the functoriality of the categorical $t$-connection \cite[Lemma 3.26]{Hug}. \qed 

\subsection{Methods of proof}
The proof of Theorem \ref{thm:main theorem} follows a reduction mod $p$ argument using a theorem of Katz, which reduces the regularity of the categorical $t$-connection over $\mathbb{C}$ to a question about a fundamental invariant of its reduction modulo primes (what this means will be made precise in Section 5) called the \emph{$p$-curvature}. The mod $p$ problem can then be solved by considerations involving the structure of certain topological operads. Below we give a brief introduction to various concepts that arise in this process. 
\subsubsection{$p$-curvature}
$p$-curvature is a key invariant of an algebraic differential equation defined over a field of positive characteristic $p$. Its definition rests on the following simple observation: if $X$ be a scheme over a field $\mathbf{k}$ of characteristic $p>0$, and $D: \mathcal{O}_X\rightarrow \mathcal{O}_X$ be a $\mathbf{k}$-linear derivation on $X$, then its $p$-th iteration 
\begin{equation}
D^p:=\overbrace{D\circ D\circ\cdots\circ D}^{p\;times}: \mathcal{O}_X\rightarrow \mathcal{O}_X  
\end{equation}
remains a derivation (an easy application of the Leibniz rule), sometimes called the \emph{$p$-th power of $D$}. \par\indent
In the above setting, further fix $(M,\nabla)$ an $\mathcal{O}_X$-module with connection. The \emph{$p$-curvature} of $\nabla$ along $D$ is defined to be
\begin{equation}
F^{\nabla}_D:=\nabla^p_{D}-\nabla_{D^p}: M\rightarrow M.    
\end{equation}
We refer the readers to \cite[Section 5]{Ka1} for its key properties and a more in-depth introduction.  \par\indent
One of the most important features of $p$-curvature is that it gives a criterion (known as \emph{Cartier descent}, cf. \cite[Theorem 5.1]{Ka1}) for when a differential equation in characteristic $p$ has a full set of algebraic solutions. Namely, this is the case if and only if its $p$-curvature vanishes. This starkly contrasts the case of characteristic $0$, where no such simple criterion for the existence of algebraic solutions to a general differential equation (with singularity) is known. A well-known (and still wildly open) conjecture of Grothendieck and Katz predicts that the existence of algebraic solutions to a differential equation in characteristic $0$ is encoded in the $p$-curvature of its mod $p$ reduction across different primes $p$; we refer the readers to \cite[Lecture 2]{Esn} for a nice introduction to this topic and \cite{Ka3} for a more detailed overview. \par\indent
The following theorem of Katz plays a central role in our argument.\\
\begin{thm}\cite[Theorem 13.0]{Ka1}\label{thm:Katz original theorem}
Let $R$ be a finite type integral domain over $\mathbb{Z}$ whose fraction field has characteristic $0$; $X\rightarrow\mathrm{Spec}\,R$ a smooth $R$-scheme of relative dimension one; and $(M,\nabla)$ a vector bundle $M$ on $X$ equipped with an integrable connection. Suppose for almost all maximal ideals $\mathfrak{m}\subset R$, the $p$-curvature ($p$ equals the residue characteristic of $\mathfrak{m}$) of $(M,\nabla)$ basechanged to $R/\mathfrak{m}$ is nilpotent, then $(M,\nabla)$ has regular singularities and quasi-unipotent monodromies after basechanged to $\mathrm{Frac}(R)$. \qed  
\end{thm}
More precisely, what is relevant to us is a formal variant of Katz's theorem which is an easy adaptation of the original one, cf. Theorem \ref{thm:Katz theorem}. Theorem \ref{thm:Katz original theorem} has been successfully used to prove regular singularity of connections in a variety of situations, for instance:
\begin{itemize}
    \item In Katz's original paper \cite{Ka1}, Theorem \ref{thm:Katz original theorem} is used to prove the regular singularity of the Gauss-Manin connection on the algebraic de Rham cohomology of a smooth and proper family $X\rightarrow S$. The proof relies on the computation that the $p$-curvature of Gauss-Manin connection vanishes on the graded pieces of the \emph{conjugate filtration}, which stabilizes in finite steps\footnote{Katz additionally derives in \cite{Ka2} a much deeper relation between the $(-1)$-shifted associated graded of the $p$-curvature of the Gauss-Manin connection and the Kodaira-Spencer class, which leads to the proof of the Grothendieck-Katz conjecture in this setting.}, and hence giving the nilpotence of $p$-curvature required to apply Theorem \ref{thm:Katz original theorem}.
    \item The work of Petrov-Vaintrob-Vologodsky \cite{PVV} generalizes Katz's approach above to prove regular singularity for the Getzler-Gauss-Manin connection of a smooth and proper family of $\mathbb{Z}$-graded dg categories. On the other hand, even though the statement of our Theorem \ref{thm:main theorem} bears some similarity with the main results in \cite{PVV}, the methods of loc.cit. cannot be directly adapted (to the author's knowledge) to the $\mathbb{Z}/2$-graded setting. Indeed, their proof relies on Kaledin's \emph{non-commutative conjugate filtration} \cite{Kal1}, which is in general only defined on the co-periodic cyclic homology. The transfer of this structure to the periodic cyclic homology only works when $\mathcal{A}$ is a smooth and homologically bounded dg algebra, and in particular breaks down in the $\mathbb{Z}/2$-graded setting. 
    \item In a slightly different direction, a recent work of the author \cite{Che2} uses Theorem \ref{thm:Katz original theorem} to prove that the quantum connection of a closed monotone symplectic manifold is of unramified exponential type \eqref{eq:exponential type decomp}. In loc.cit., the proof of nilpotence does not use a filtration argument and instead relies on an interpretation, using ideas from an earlier work of \cite{Lee}, of the $p$-curvature of the quantum connection in terms of the \emph{quantum Steenrod operations}, which are certain operations on (equivariant) quantum cohomology constructed out of mod $p$ counts of (equivariant) Gromov-Witten invariants \cite{SW}\cite{Wilk}. In some sense, the current paper is an attempt to adapt the high-level ideas of \cite{Che2} to the context of non-commutative Hodge theory. 
\end{itemize}

\subsubsection{The Kontsevich-Soibelman operad}
Deligne's Hochschild cohomology conjecture, now a theorem with various proofs \cite{KS2}\cite{MS}\cite{Tam}\cite{Vo2}, states that the Hochschild cochain complex admits an action by (the singular chains on) the operad of little disks. The Kontsevich-Soibelman operad \cite{KS1}, a two-colored extension of the little disks operad, gives a generalization of Deligne's conjecture as it parametrizes `universal operations' on the pair $(CC^*(\mathcal{A}),CC_*(\mathcal{A}))$ given by the Hochschild cochain and chain complexes of a dg or d($\mathbb{Z}/2$)g algebra $\mathcal{A}$. \par\indent
There are various different perspectives on the Kontsevich-Soibelman operad, of which we review three that are most relevant for this paper. \par\indent
\textbf{A two-colored version of the little disks operad}. As a most visual description, consider the two-colored topological operad $\{\mathrm{Cyl}(l,0),\mathrm{Cyl(k,1)}\}_{l\geq 1,k\geq 0}$, where $\{\mathrm{Cyl}(l,0)\}_{l\geq 1}$ is defined to be the little disks operad and $\mathrm{Cyl(k,1)}$ is defined as the configuration space of $k$ disks lying on a cylinder (with some extra data; cf. Section 3.2 for more details). The operadic composition rule is induced by insertion of configurations of disks or stacking of two cylinders. This topological model is particularly convenient for computations involving equivariant cohomology, cf. Section 4.2. On the other hand, from this description it is a priori not clear how to write down an explicit action on the Hochschild cochain/chain complex.  \par\indent
\textbf{Homotopy calculus}. Instead of directly defining the operad, one can start with describing the type of operations the given operad should parametrize. A \emph{Gerstenhaber algebra} is a graded module $V^{\bullet}$ together with a degree $0$ graded commutative product $\cdot$ and a degree $-1$ Lie bracket $[-,-]$ such that
\begin{equation}
 [a,bc]=[a,b]c+(-1)^{(|a|-1)|b|} b[a,c].   
\end{equation}
Classical examples include the multi-vector fields on a smooth manifold equipped with the wedge product and the Schouten bracket or the Hochschild cohomology of a dg or d$(\mathbb{Z}/2)$g algebra equipped with its cup product and Gerstenhaber bracket. The \emph{Gerstenhaber operad}, which is obtained by abstracting the cup product and Lie bracket (as well as their compatibility relations), agrees with the homology of the little disks operad (this was one of the initial motivations for Deligne's conjecture). \par\indent
A \emph{calculus} is a pair consisting of a Gerstenhaber algebra $V^{\bullet}$ and another graded module $\Omega^{\bullet}$, together with a degree $0$ graded commutative algebra action $i$ (which in later parts of the paper is denoted `$e$') of $V^{\bullet}$ on $\Omega^{\bullet}$, a degree $-1$ graded Lie algebra action $L$ of $V^{\bullet}$ on $\Omega^{\bullet}$, and a degree $-1$ differential $d$ on $\Omega^{\bullet}$ such that 
\begin{equation}
[i_a,L_b]=i_{[a,b]}\quad,\quad L_{ab}=L_ai_b+(-1)^{|a|}i_aL_b\quad,\quad [d,i_a]=L_a.    
\end{equation}
As its name suggest, a calculus is the abstraction of the various algebraic structures that appear in the study of multi-vector fields and differential forms on a smooth manifold; another important example is the pair given by the Hochschild cohomology and homology of a dg or d$(\mathbb{Z}/2)$g algebra. In parallel to the case of the Gerstenhaber operad, the calculus operad is equivalent to the homology of the two-colored little disks operad.   \par\indent
Roughly speaking, a homotopy calculus operad should encode chain level lifts of the cohomology operations (and their relations) parametrized by the calculus operad. For instance, one should expect that the categorical $t$-connection \eqref{eq:categorical t connection intro} admits some kind of description inside this framework. This can be abstractly achieved by taking a cofibrant resolution (e.g. the cobar-bar construction) of the calculus operad, cf. \cite{DTT}. However, in practice it is often desirable to use a smaller and more workable model. \par\indent
\textbf{The cyclic cacti operad}. The primary model we use for the Kontsevich-Soibelman operad is a variant of the \emph{cacti operad} first introduced by McClure-Smith \cite{MS} in their proof of Deligne's conjecture, cf. Section 3.1 for more details. Cacti are in bijection with the cells of a cellular operad up to homotopy which is homotopy equivalent to the little disks operad. The cacti operad both encodes (essentially tautologically) chain level operations on the Hochschild cochain complex of an associative/dg algebra and has convenient combinatorics (due to its simplicial structure) to work with. In \cite{Che3}, the author introduced a two-colored generalization of the cacti operad, called the \emph{cyclic cacti operad} (cf. Figure \ref{fig:cells_in_cyclic_cacti_and_their_actions} for some illustrations), and proved that it is equivariantly homotopy equivalent to the chains on the two-colored little disks operad, cf. Theorem \ref{thm:comparison of Cact with Cyl}.  \par\indent
Let 
\begin{equation}
\mathbf{KS}=\{\mathbf{KS}(l,0),\mathbf{KS}(k,1)\}_{l\geq 1,k
\geq 0}    
\end{equation}
denote the Kontsevich-Soibelman operad (for now we will be agnostic about which model we are using). By considering its induced action on the pair $(CC^*(\mathcal{A}),CC_*(\mathcal{A}))$ and taking appropriate (homotopy) fixed points/orbits, one obtains an action 
\begin{equation}\label{eq: equivariant action of cyclic cacti on HH_* and HH^* intro}
\Xi_{\mathcal{A}}: H^*((\mathbf{KS}(k,1)_{\Sigma_k})^{tS^1})\rightarrow \mathrm{Hom}_{H^2_{S^1}(pt;R)}(H^*((CC^*(\mathcal{A})^{\otimes_R k})^{h\Sigma_k})\otimes_R HH^{per}_*(\mathcal{A}), HH^{per}_*(\mathcal{A})).        
\end{equation}
Let $F^{\mathcal{A}}_{2t^2\partial_t}$ denote the $p$-curvature of the categorical $t$-connection along the vector field $2t^2\partial_t$. The key observation is the following. \\
\begin{prop}(cf. Corollary \ref{thm: p curvature as KS operation})\label{thm:p curvature as KS operation intro}
Let $\mathcal{A}$ be a d$(\mathbb{Z}/2)$g algebra over a field of odd characteristic $p$. Then 
\begin{equation}
F^{\mathcal{A}}_{2t^2\partial_t}=\Xi_{\mathcal{A}}([\alpha])([d^{\otimes p}],-)   
\end{equation}
for some $[\alpha]\in H^0((\mathbf{KS}(p,1)_{\Sigma_p})^{tS^1})$. 
\end{prop}
As a corollary of Proposition \ref{thm:p curvature as KS operation intro}, one obtains the following multiplicative property of $p$-curvature, from which the desired nilpotence property (which allows one to apply Katz's theorem \ref{thm:Katz original theorem}) will be derived.\\
\begin{cor}(cf. Theorem \ref{thm: p curvature equals multiplicative operation})\label{thm: p curvature equals multiplicative operation intro}
In the setting of Proposition \ref{thm:p curvature as KS operation intro}, there exists a multiplicative (with respect to the cup product on $HH^{even}(\mathcal{A})^{(1)}$) action
\begin{equation}
\bigcap\nolimits^{C_p}: HH^{even}(\mathcal{A})^{(1)}\otimes HH_*^{per}(\mathcal{A})\rightarrow  HH_*^{per}(\mathcal{A})   
\end{equation}
and a constant $c$ such that 
\begin{equation}
 F^{\mathcal{A}}_{2t^2\partial_t}=c\cdot\bigcap\nolimits^{C_p}([d],-).
\end{equation}
Here, $HH^{even}(\mathcal{A})$ denotes the even part of Hochschild cohomology and $(-)^{(1)}$ denotes the Frobenius twist.
\end{cor}\par\indent
\begin{rmk}
 The $\bigcap^{C_p}$ in Corollary \ref{thm: p curvature equals multiplicative operation intro} is closed related to the \emph{$C_p$-equivariant cap product} introduced in \cite{Che1}, where it was shown to recover under favorable situations the quantum Steenrod operation of a closed monotone symplectic manifold mentioned at the end of Section 1.2.1. It is in this sense that the approach of this paper gives a non-commutative generalization of \cite{Che2}. On the other hand, using ideas from stable homotopy theory, a recent work of Rezchikov \cite{Rez} proved a version of Corollary \ref{thm: p curvature equals multiplicative operation intro} for Getzler-Gauss-Manin connections assuming $\mathcal{A}$ is smooth proper $\mathbb{Z}$-graded and satisfies suitable homological bounds \cite[Theorem 9.1]{Rez} (conditions that guarantee a lift to the sphere spectrum). 
\end{rmk}

\subsection{A non-commutative local monodromy theorem}
Perhaps surprisingly, the operadic interpretation of $p$-curvature allows one to bound the sizes of Jordan blocks of the monodromy of the categorical $t$-connection using ideas from \cite{PS2}. \\
\begin{thm}(cf. Theorem \ref{thm:nc Scherk}) \label{thm:nc Scherk intro}
Let $\mathcal{C}$ be a smooth proper d$(\mathbb{Z}/2)$g category over $\mathbb{C}$. Suppose the Hochschild cohomology class of its differential $[d]\in HH^*(\mathcal{C})$ satisfies $[d]^{\cup r}=0$, then the size of the largest Jordan block of the monodromy of the $t$-connection on $HH^{per}_*(\mathcal{C})$ is $\leq r$.      
\end{thm}
Setting $\mathcal{C}$ to be the category of matrix factorization of an isolated hypersurface singularity germ $W$, Theorem \ref{thm:nc Scherk intro} recovers a classical result of Scherk \cite{Sche}, cf. Remark \ref{thm:remark about Scherk thm and FL transform}. Moreover, in the special case when $[d]=0$ (which is a non-commutative analogue of being \emph{quasi-homogeneous}), Theorem \ref{thm:nc Scherk intro} implies that the categorical $t$-connection has finite monodromy. \par\indent
Via Theorem \ref{thm:nc Scherk intro} we also obtain a curious relation between the categorical $t$-connection and the diagonal dimension of $\mathcal{C}$, a certain measure of the complexity of $\mathrm{Perf}(\mathcal{C})$. \\
\begin{thm} (cf. Theorem \ref{thm:bounding Jordan blocks by Ddim}) \label{thm:bounding Jordan blocks by Ddim intro}
Let $\mathcal{C}$ be a smooth proper d$(\mathbb{Z}/2)$g category over $\mathbb{C}$. Then the size of the largest Jordan block of the monodromy of $\nabla^{\mathcal{C}}_{\partial_t}$ is $\leq \mathrm{Ddim}(\mathcal{C})+1$.     
\end{thm}

\subsection*{Organization}
The organization of this paper is as follows. In Section 2, we review the definition of the Hochschild invariants of a d$(\mathbb{Z}/2)$g algebra and the categorical $t$-connection on its periodic cyclic homology. In Section 3, we review the various models for the two-colored Kontsevich-Soibelman operad and the notion of Kontsevich-Soibelman operation. In Section 4, we study the relation between the $p$-curvature and the Kontsevich-Soibelman operad, and prove the key results Proposition \ref{thm:p curvature as KS operation intro} and Corollary \ref{thm: p curvature equals multiplicative operation intro}. Section 5 contains a spreading out argument by applying a theorem of To\"{e}n and completes the proof of Theorem \ref{thm:main theorem}. In Section 6, we apply our main results to obtain an explicit formula of the non-commutative Hodge filtration and a non-commutative local monodromy theorem.

\subsection*{Acknowledgments}
We would like to thank Paul Seidel, Dmitry Kaledin and Sasha Petrov for helpful discussions or correspondences at various stages of this project. Part of this project was completed during the author's visit to ShanghaiTech University in April 2026, and we thank Junwu Tu for his hospitality and many enlightening conversations. The author is partially supported by a Title A Fellowship from Trinity College, Cambridge and a Postdoctoral Fellowship from the Herchel-Smith Fund.

\subsection*{AI disclosure}
We used Chatgpt 5.6 Pro to proofread, check for mathematical errors and suggest potential improvements for a draft of this paper (42 pages; available upon request). Notably, it found a refinement of Corollary \ref{thm: nilpotence of p curvature} (which suffices for proving the main theorem) using work of Katzarkov-Kerr \cite{KK}, cf. Lemma \ref{thm:F^1HH^* is C^e ghost} and Corollary \ref{thm:bounding nilpotence order of d by Hdim}; we collected these results and some implications into a separate Section 6.2.1. Apart from Section 6.2.1, all new ideas and their technical implementations in this paper are human generated.

\section{The $t$-connection on periodic cyclic homology}
Let $\mathcal{A}$ be a differential $\mathbb{Z}/2$-graded algebra over a ground ring $R$, i.e. an $R$-module $\mathcal{A}=\mathcal{A}_0\oplus\mathcal{A}_1$ equipped with an odd degree differential $d$ and an $R$-bilinear associative product $\cdot$ such that
\begin{equation}
d(x\cdot y)=dx\cdot y+(-1)^{|x|}x\cdot dy. \label{eq:leibniz}    
\end{equation}
We will assume that $\mathcal{A}$ contains a strict unit $1$. \par\indent
The goal of this section is to review the definition of Hochschild (co)homology and cyclic homology of $\mathcal{A}$, and then give an explicit chain level formula for the canonical $t$-connection.  
\subsection{Hochschild cohomology}
As a graded $R$-module, the \emph{Hochschild cochain complex} of $\mathcal{A}$ is
\begin{equation}
\mathcal{CC}^*(\mathcal{A}):=\prod_{k\geq 0}\mathrm{Hom}_{R}^*(\mathcal{A}[1]^{\otimes_R k},\mathcal{A}).\label{eq:CC^*}
\end{equation}
There is a chain level binary operation on \eqref{eq:CC^*}, called the \emph{circle product}, given by
\begin{equation}
(\varphi\circ\phi)(x_1,\cdots,x_n):=\sum_{j\leq k}(-1)^{\|\phi\|\cdot \sum_{i=1}^{j-1}\|x_i\|} \varphi(x_1,\cdots,x_{j-1},\phi(x_j,\cdots,x_{k-1}),x_k,\cdots,x_n),  
\end{equation}
where $\|x\|:=|x|-1$ denotes the reduced degree. This gives rise to a degree $-1$ chain level Lie bracket by
\begin{equation}
[\varphi,\phi]:=\varphi\circ \phi-(-1)^{\|\varphi\|\|\phi\|}\phi\circ\varphi.
\end{equation}
Define $\mu_{\mathcal{A}}\in \mathcal{CC}^{2}(\mathcal{A})$ to be the Hochschild cochain given by
\begin{equation}
\mu_{\mathcal{A}}^0=0,\;\mu^1_{\mathcal{A}}(x)=dx,\;\mu^2_{\mathcal{A}}(x,y)=(-1)^{|x|}xy,\;\mu^k_{\mathcal{A}}=0\;\;\textrm{for}\;\;k\geq 3.    
\end{equation}
Then associativity and the Leibniz rule \eqref{eq:leibniz} implies that $\mu_{\mathcal{A}}\circ \mu_{\mathcal{A}}=0$, and in particular, $[\mu_{\mathcal{A}},\mu_{\mathcal{A}}]=0$. One then defines the differential on $\mathcal{CC}^*(\mathcal{A})$ to be
\begin{equation}
[\mu_{\mathcal{A}},-]: \mathcal{CC}^*(\mathcal{A})\rightarrow \mathcal{CC}^{*+1}(\mathcal{A}),
\end{equation}
which squares to zero as a consequence of the Jacobi identity. We denote the cohomology of this complex by $HH^*(\mathcal{A})$. \par\indent
There is a product structure on the Hochschild cochain complex $\mathcal{CC}^*(\mathcal{A})$ called the \emph{cup product} (or \emph{Yoneda product}). On the chain level, it is given by
\begin{equation}
\phi\cup\psi(x_1,\cdots,x_k):=\sum_{i=1}^{k+1}(-1)^{|\psi|\cdot\sum_{j=1}^{i-1}\|x_j\|}\phi(x_1,\cdots,x_{i-1})\psi(x_{i},\cdots,x_k).
\end{equation}
The cup product descends to cohomology and defines an associative (and in fact graded commutative) algebra structure on $HH^*(\mathcal{A})$, for which $1\in\mathcal{A}$ (viewed as a Hochschild cocycle of length $0$) is the unit.\par\indent
It is often more convenient to use a slightly different chain model for $HH^*(\mathcal{A})$, namely the \emph{normalized Hochschild cochain complex} 
\begin{equation}\label{eq:normalized HH cochain}
CC^*(\mathcal{A})\subset \mathcal{CC}^*(\mathcal{A})    
\end{equation} 
consisting of multi-linear functions $\phi$ such that $\phi(\cdots,1,\cdots)=0$. It is a standard fact using the co-simplicial description of $HH^*$ that the inclusion \eqref{eq:normalized HH cochain} is a subcomplex and a quasi-isomorphism.

\subsection{Hochschild homology and cyclic homology}
The \emph{Hochschild chain complex} of $\mathcal{A}$ (over $R$) is
\begin{equation}
\mathcal{CC}_*(\mathcal{A}):=\bigoplus_{n\geq 0}\mathcal{A}\otimes_R\mathcal{A}[1]^{\otimes_R n},
\end{equation}
with differential given by
$$ b(x_0\otimes x_1\otimes\cdots\otimes x_n)=\sum_{i=0}^n (-1)^{\sum_{j=0}^{i-1}\|x_j\|} x_0\otimes x_1\otimes\cdots \otimes dx_i\otimes\cdots\otimes x_n$$
\begin{equation}
+\sum_{i=0}^{n-1} (-1)^{1+\sum_{j=0}^{i}\|x_j\|}x_0\otimes x_1\otimes \cdots\otimes x_{i}x_{i+1}\otimes\cdots\otimes x_n-(-1)^{\|x_n\|\cdot(|x_0|+\sum_{i=1}^{n-1}\|x_i\|)} x_n x_0\otimes x_1\otimes\cdots\otimes x_{n-1}.\label{eq:b}
\end{equation}
The cohomology of this complex is called the \emph{Hochschild homology} of $\mathcal{A}$, denoted $HH_*(\mathcal{A})$, and the above chain complex is also called the cyclic bar model.  \par\indent
A key feature of Hochschild homology is that it admits a chain level $S^1$-action. There are various ways to interpret what this means, which we further explore in later parts of the paper. For the moment, we recall one such interpretation via the \emph{Connes operator}.\par\indent
For the purpose of describing the Connes operator, it would be more convenient to use a quasi-isomorphic chain model for Hochschild homology, namely the \emph{normalized Hochschild complex}. It is defined as
\begin{equation}
CC_*(\mathcal{A}):=\mathcal{CC}_*(\mathcal{A})/N,     
\end{equation}
where $N$ denotes the subcomplex generated by elements of the form $\mathbf{x}\otimes x_1\otimes\cdots\otimes 1\otimes\cdots\otimes x_n$ ($1$ sits in the $i$-th entry, $i>0$). The differential on $CC_*(\mathcal{A})$ is inherited from $\mathcal{CC}_*(\mathcal{A})$ under the quotient, which we still denote by $b$. It is a standard fact using the simplicial description of the Hochschild chain complex that the quotient map $\mathcal{CC}_*(\mathcal{A})\rightarrow CC_*(\mathcal{A})$ is a quasi-isomorphism. \par\indent
The Connes operator is the degree $-1$ operator on $CC_*(\mathcal{A})$ given by
\begin{equation}
B(x_0\otimes x_1\otimes\cdots\otimes x_n):=\sum_{i=0}^n (-1)^{(\sum_{j=i}^n\|x_j\|)(\sum_{j=0}^{i-1}\|x_j\|)} 1\otimes x_i\otimes x_{i+1}\otimes\cdots\otimes x_n\otimes x_0\otimes\cdots\otimes x_{i-1}.\label{eq:B}
\end{equation}
$B$ satisfies 
\begin{equation}
b\circ B+B\circ b=0,\;\;\;B^2=0.    
\end{equation}
In particular, it defines an action of the dg algebra $R[\epsilon]/\epsilon^2\simeq C_*(S^1)$ (where $|\epsilon|=-1$) on $CC_*(\mathcal{A})$.\par\indent 
The \emph{negative cyclic chain complex} is the following explicit complex computing the homotopy fixed point of this $S^1$-action: 
\begin{equation}
CC^{-}_*(\mathcal{A}):=(CC_*(\mathcal{A})[[t]],b+tB)\cong \mathrm{RHom}_{R[\epsilon]/\epsilon^2}(R,CC_*(\mathcal{A})),\;|t|=2
\end{equation}
where we think of $t$ as the generator of $H^2_{S^1}(\mathrm{pt})$. The homology of this complex is called the \emph{negative cyclic homology} of $\mathcal{A}$ and denoted $HH^-_*(\mathcal{A})$. \par\indent
Finally, the \emph{periodic cyclic chain complex} is defined as 
\begin{equation}
CC^{per}_*(\mathcal{A}):=(CC_*(\mathcal{A})((t)),b+tB),
\end{equation}
whose homology is called the \emph{periodic cyclic homology} of $\mathcal{A}$ and denoted $HH^{per}_*(\mathcal{A})$.

\subsection{The Getzler-Gauss-Manin connection}
In this subsection, we work in the following relative setting: fix $\mathcal{R}$ a commutative ring over the base ring $R$ (the primary example of $\mathcal{R}$ would be the ring of functions on a smooth affine irreducible curve over $R$). Let $(\mathcal{A}, d,\cdot)$ be a $\mathbb{Z}/2$-graded dg-algebra over $\mathcal{R}$. \par\indent
Given a Hochschild cocycle $D\in CC^*(\mathcal{A}/\mathcal{R})$, Getzler defined a `contraction' operator
\begin{equation}e_D:CC_*(\mathcal{A}/\mathcal{R})\rightarrow CC_{*+|D|}(\mathcal{A}/\mathcal{R})\label{eq:e_D}
\end{equation}
 given by
\begin{equation}
e_D(x_0\otimes x_1\otimes\cdots\otimes x_n):= \sum (-1)^{(\sum_{j=k+1}^n\|x_j\|)\cdot(|x_0|+\sum_{i=1}^k \|x_i\|)} D(x_{k+1},\cdots,x_n)x_0\otimes x_1\otimes\cdots\otimes x_k;    
\end{equation}
a `Lie derivative' operator 
\begin{equation}\mathcal{L}_D: CC_*(\mathcal{A}/\mathcal{R})\rightarrow CC_{*+|D|-1}(\mathcal{A}/\mathcal{R})\label{eq:L_D}
\end{equation}
given by
$$\mathcal{L}_D(x_0\otimes x_1\otimes\cdots\otimes x_n):=\sum (-1)^{\|D\|\cdot(\sum_{i=0}^{k} \|x_i\|)}x_0\otimes\cdots\otimes D(x_{k+1},\cdots,x_l)\otimes x_{l+1}\otimes\cdots\otimes x_n$$
\begin{equation}
+\sum (-1)^{(\sum_{i=l+1}^n\|x_i\|)\cdot(\sum_{j=0}^l\|x_j\|)} D(x_{l+1},\cdots,x_0,\cdots,x_k)\otimes x_{k+1}\otimes\cdots\otimes x_l;    
\end{equation}

and an auxiliary operator
\begin{equation}
E_D: CC_*(\mathcal{A}/\mathcal{R})\rightarrow CC_{*+|D|-2}(\mathcal{A}/\mathcal{R})  \label{eq:E_D}  
\end{equation}
given by $E_D(x_0\otimes x_1\otimes\cdots\otimes x_n):=$
\begin{equation}
\sum (-1)^{|D|+\|D\|\cdot(\sum_{i=s+1}^k\|x_i\|)+(\sum_{i=s+1}^n\|x_i\|)\cdot(\sum_{i=0}^s\|x_i\|)} 1\otimes x_{s+1}\otimes \cdots\otimes D(x_{k+1},\cdots,x_l)\otimes\cdots\otimes x_0\otimes\cdots\otimes x_s.  
\end{equation}
These operations satisfy
\begin{equation}
[e_D,b]=0, \;\;[E_D,B]=0   \label{eq:e_D,E_D basic}
\end{equation}
and an analogue of the Cartan homotopy formula (up to homotopy): 
\begin{equation}
[e_D, B]=\mathcal{L}_D-[E_D,b].     \label{eq:cartan}
\end{equation}
We now recall Getzler's construction of the Gauss-Manin connection in \cite{Get}. First, after replacing $\mathcal{A}$ by a semi-free resolution, we fix an $\mathcal{R}$-basis of $\mathcal{A}=\bigoplus_i \mathcal{A}_i$ as a $\mathbb{Z}/2$-graded $\mathcal{R}$-module. Let 
\begin{equation}\nabla': \mathcal{A}\rightarrow \mathcal{A}\otimes_\mathcal{R}\Omega^1_\mathcal{R}\label{eq:nabla'}
\end{equation}
be the trivial connection on $\mathcal{A}$ with respect to this basis. Extend $\nabla'$ to a connection $\bigoplus_n \mathcal{A}^{\otimes n+1}\rightarrow \bigoplus_n \mathcal{A}^{\otimes n+1}\otimes_\mathcal{R}\Omega^1_\mathcal{R}$ by the Leibniz rule. Define
 \begin{equation}
  \kappa:=[\nabla',\mu]: \bigoplus_n\mathcal{A}^{\otimes n}\rightarrow \mathcal{A}\otimes_\mathcal{R}\Omega^1_\mathcal{R}.\label{eq:kappa}
 \end{equation}
 Equivalently, \eqref{eq:kappa} is obtained by expanding $\mu$ as a matrix with respect to the chosen $R$-basis, and differentiating entry-wise. In particular, \eqref{eq:kappa} gives rise to a Hochschild cocycle and its cohomology class $[\kappa]\in HH^2(\mathcal{A})\otimes_\mathcal{R}\Omega^1_\mathcal{R}$, which is independent of the choice of an $\mathcal{R}$-basis, is called the \emph{Kodaira-Spencer class} of $\mathcal{A}$. \par\indent
 $\nabla'$ induces a connection on the negative cyclic chain complex $CC^{-}_*(\mathcal{A})$ as a graded $\mathcal{R}$-module, which however does not commute with the differential $b+tB$. The failure to commute is measured by
 \begin{equation}
 [\nabla', b]=\mathcal{L}_{\kappa},\;\;[\nabla',B]=0.   \label{eq:nabla' property} 
 \end{equation}
\eqref{eq:e_D,E_D basic}, \eqref{eq:cartan} and \eqref{eq:nabla' property} imply that the chain level connection on $CC^{-}_*(\mathcal{A})$
 \begin{equation}
 \nabla^{GGM}:=\nabla'-\frac{1}{t}\iota_{\kappa},    \label{eq:GGM}
 \end{equation}
 where $\iota_D=e_D+tE_D$,
 descends to the negative cyclic homology. \\
 \begin{mydef}
The connection $HH^{-}_*(\mathcal{A}/\mathcal{R})\rightarrow HH^{-}_*(\mathcal{A}/\mathcal{R})\otimes_\mathcal{R}\Omega^1_\mathcal{R}$ induced by \eqref{eq:GGM} is called the \emph{Getzler-Gauss-Manin connection} of $\mathcal{A}/\mathcal{R}$. We also refer to the induced (by inverting $t$) connection on $HH^{per}_*(\mathcal{A}/\mathcal{R})$ as the Getzler-Gauss-Manin connection.  
 \end{mydef}
 
\subsection{From Getzler-Gauss-Manin connection to the $t$-connection}
The main result of this paper concerns another connection on the periodic cyclic homology in the `$t$-direction' introduced in \cite[Section 2.2.5]{KKP}; see also \cite{Shk}. We now recall its definition, and specifically, its relation with the Getzler-Gauss-Manin connection.\par\indent
Given a d($\mathbb{Z}/2$)g algebra $\mathcal{A}/R$, let $\mathcal{R}=R[s]$ and consider the one-parameter family of d($\mathbb{Z}/2$)g algebra $(\mathcal{A}_s,d_s,\cdot_s)/\mathcal{R}$ given by 
\begin{equation}\label{eq:t-family}
(\mathcal{A}_s:=\mathcal{A}[s],\; d_s:=sd,\;a\cdot_sb:=a\cdot b).    
\end{equation}
In particular, the Kodaira-Spencer class of \eqref{eq:t-family} in the direction $\partial_s$ is just the differential $d$, viewed as an element of $HH^2(\mathcal{A})$ of length one.\par\indent
Consider the chain level operator
\begin{equation}\label{eq:Euler grading}
Gr:=s\frac{\partial}{\partial s}+ 2t\frac{\partial}{\partial t}+\Gamma
\end{equation}
on $CC_*(\mathcal{A}_s)[[t]]$, where 
\begin{equation}\label{eq:length operator}
\Gamma(x_0\otimes x_1\otimes \cdots \otimes x_n):=-n\cdot x_0\otimes x_1\otimes \cdots \otimes x_n.  
\end{equation}
A straightforward verification shows that $Gr$ satisfies
\begin{equation}
[Gr, b+tB]= b+tB.
\end{equation}
As a result, the chain level formula
\begin{equation}\label{eq:t connection}
\nabla_{t\partial_t}^{\mathcal{A}}:= \frac{1}{2}Gr-\frac{1}{2}\nabla^{GGM}_{s\partial_s}=t\frac{\partial}{\partial t}+\frac{\Gamma}{2}+\frac{\iota_{sd}}{2t}
\end{equation}
descends\footnote{Even though it is not a chain map, and instead satisfies $[\nabla^{\mathcal{A}}_{t\partial_t},b+tB]=\frac{1}{2}(b+tB)$.} to a well defined connection with second order pole at $t=0$ on $HH^{per}_*(\mathcal{A}_s/\mathcal{R})$.\\
\begin{mydef}\label{thm:t connection}
The restriction of \eqref{eq:t connection} to $s=1$ (after $t$-completion) gives a connection on $HH^{per}_*(\mathcal{A}/R)$ with second order pole at $t=0$, which is called the \emph{$t$-connection on the periodic cyclic homology} (or just the \emph{categorical $t$-connection} for short).     
\end{mydef}
We view this as a connection on the vector bundle $HH^{per}_*(\mathcal{A}/R)$ over the formal punctured disk $\mathrm{Spf}\,R((t))$. It is an easy fact that when the $\mathbb{Z}/2$-grading of $\mathcal{A}$ can be lifted to a $\mathbb{Z}$-grading, then $\nabla_{\partial_t}^{\mathcal{A}}$ only has a first order pole at $t=0$, cf. \cite[Lemma 3.2]{CLT}.

\section{The Kontsevich-Soibelman operad}
The purpose of this section is to introduce the two-colored Kontsevich-Soibelman operad \begin{equation}
\mathbf{KS}=\{\mathbf{KS}(l,0),\mathbf{KS}(k,1)\}_{l\geq 1,k
\geq 0}    
\end{equation}valued in chain complexes \cite{KS1}, which acts on the pair $(CC^*(\mathcal{A}),CC_*(\mathcal{A}))$ of (normalized) Hochschild (co)chain complexes of a dg or d($\mathbb{Z}/2$)g algebra $\mathcal{A}$. \par\indent
We introduce two models for this operad. The first model, reviewed in Section 3.1, is based on the combinatorial theory of cyclic cacti developed in \cite{Che3}, building on earlier foundational works of \cite{MS} \cite{Vo1} \cite{Kau} \cite{Sal1}. The second model, reviewed in Section 3.2, is based on variants of configuration spaces in classical topology, see for instance \cite{KS1} \cite{DTT}. In Section 3.3, we recall the definition of \emph{Kontsevich-Soibelman operations} introduced in \cite{Che3}. In words, these are endomorphisms of the periodic cyclic homology induced from certain equivariant homology classes of $\mathbf{KS}$. These operations and their operadic properties will play a crucial role later in studying the $p$-curvature of the $t$-connection. 

\subsection{A combinatorial model}
\subsubsection{Cacti}
Recall that an \emph{$n$-fold semi-simplicial (resp. semi-cosimplicial) object in a category $C$} is a covariant (resp. contravariant) functor from $(\overrightarrow{\Delta}^{op})^n$ to $C$. The \emph{geometric realization} of an $n$-fold semi-simplicial set $X$ is defined to be
\begin{equation}\label{eq:geometric realization of n fold semisimplicial set}
|X|:=\coprod_{a_1,\cdots,a_n\geq 0} X_{a_1,\cdots,a_n}\times \Delta^{a_1}\times\cdots\Delta^{a_n}/\sim,  
\end{equation}
where $\sim$ is the equivalence relation given by 
\begin{equation}
(f^*(x),y)\sim (x,f_*(y)),\;\mathrm{where}\;f\in \overrightarrow{\Delta}([a_1,\cdots,a_n],[b_1,\cdots,b_n]), x\in X_{b_1,\cdots,b_n}, y\in\prod_{i=1}^n\Delta^{a_i}.
\end{equation}
\begin{mydef}(\cite[Section 4.2]{MS},\cite[Definition 4.2]{Sal1})\label{thm:cacti}
$\mathfrak{Cact}^k=\mathfrak{Cact}^k_{\bullet,\cdots,\bullet}$ is the $k$-fold semi-simplicial set defined as follows. For $[m_1,\cdots,m_k]\in (\overrightarrow{\Delta}^{op})^k$, $\mathfrak{Cact}^k_{m_1,\cdots,m_k}$ is the set of surjective maps $f: \{1,\cdots,m+k\}\rightarrow \{1,\cdots,k\}$ with $m=\sum_{i=1}^k m_i$ such that
\begin{enumerate}[label=\arabic*)]
    \item $f^{-1}(j)$ has cardinality $m_j+1$ for each $1\leq j\leq k$.
    \item $f(a)\neq f(a+1)$ for $1\leq a<m+k$.
    \item There are no values $1\leq a<b<c<d\leq m+k$ such that $f(a)=f(c)\neq f(b)=f(d).$
\end{enumerate}
We will equivalently represent $f$ by the sequence $f(1)f(2)\cdots f(m+k)$. Fix $1\leq j\leq k$, and denote $f^{-1}(j)=\{a_0<a_1<\cdots<a_{m_j}\}$. Then the $i$-th ($0\leq i\leq m_j$) face map of the $j$-th component is defined by
\begin{equation}
d^{(j)}_i(f):=f(1)\cdots\widehat{f(a_i)}\cdots f(m+k): \{1,\cdots,m+k-1\}\rightarrow \{1,\cdots,k\}.
\end{equation}
The geometric realization of this $k$-fold semi-simplicial set $\mathrm{Cact}^k:=|\mathfrak{Cact}^k|$ is called the \emph{space of cacti with $k$ lobes}. \par\indent
Note that $\mathrm{Cact}^k$ has a natural cellular structure induced from the (multi-)semi-simplicial structure of $\mathfrak{Cact}^k$. Then $C^{cell}_*(\mathrm{Cact}^k;R)$ agrees with the realization of $R\langle\mathfrak{Cact}^k\rangle$ in the category of $R$-chain complexes. 
\end{mydef}
\textbf{Geometric intuition}. The name `cactus' comes from the following interpretation. An element $(f,t^1,\cdots,t^k)\in \mathfrak{Cact}^k_{m_1,\cdots,m_k}\times \Delta^{m_1}\times\cdots\Delta^{m_k}$ may be viewed as a partition of $[0,k]$ into $m+k$ closed sub-intervals 
\begin{equation}[0,t_1], [t_1,t_1+t_2],\cdots, [\sum_{i=1}^{m+k-1}t_i,\sum_{i=1}^{m+k}t_i=k]\end{equation}
such that if we denote $f^{-1}(j)=\{a_0<a_1<\cdots<a_{m_j}\}$ then $(t_{a_0},\cdots,t_{a_{m_j}})=t^j$, for each $1\leq j\leq k$. The $i$-th interval $[\sum_{l=1}^{i-1}t_l,\sum_{l=1}^it_l]$ is said to have `color' $f(i)\in\{1,\cdots,k\}$. Note the sum of lengths of intervals of each color is $1$. By identifying the endpoints of $[0,k]$, we view this as a partition of $S^1$. We define an equivalence relation on $S^1=[0,k]/0\sim k$ where $z\sim z'$ if $z$ and $z'$ are the boundary points of the same connected component of $S^1\backslash\mathrm{int}(I_j)$ for some $1\leq j\leq k$, where $I_j$ is the union of all intervals of color $j$. Condition 2) of Definition \autoref{thm:cacti} implies that the quotient under the prior equivalence relation is a connected union of $k$ circles of length $1$. This quotient is called the \emph{cactus associated with $(f,t^1,\cdots,t^k)$} (we will also abuse terminology and refer to $(f,t^1,\cdots,t^k)$ itself as a cactus); the image of $I_j$ is called \emph{the $j$-th lobe} of the cactus; the image of $0$ is \emph{the basepoint of the cactus} and the image of an endpoint of an interval $[\sum_{l=1}^{i-1}t_l,\sum_{l=1}^it_l]$ is called a \emph{marked point on the cactus}. See Figure \ref{fig:cactus} for an illustration.  
\begin{figure}[H]
 \centering
 \includegraphics[width=0.9\textwidth]{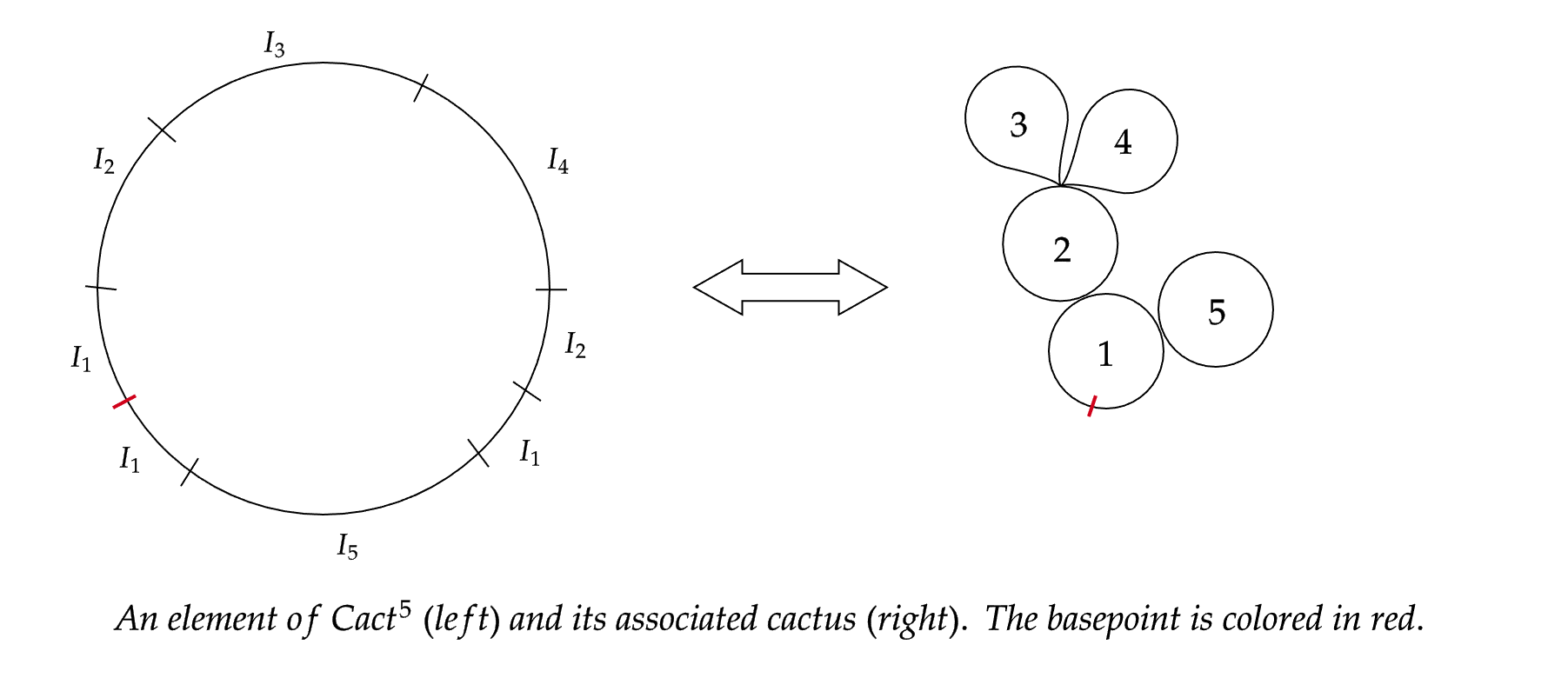}
 \caption{}
 \label{fig:cactus}
\end{figure}
\begin{mydef}\label{thm:cactus map}
Let $x=(f,t^1,\cdots,t^k)\in \mathrm{Cact}^k$ be a cactus. The \emph{cactus map} is a piecewise (linear) oriented isometry
\begin{equation}\label{eq:cactus map}
c_x=(c_x^1,\cdots,c_x^k): [0,k]\rightarrow [0,1]^k   
\end{equation}
defined by
\begin{itemize}
    \item $c_x(0)=(0,\cdots,0)$.
    \item If $t\in [\sum_{l=1}^{i-1}t_l,\sum_{l=1}^{i}t_l]$, then for $1\leq j\leq k$, $c^j_x(t)=\begin{cases} c^j_x(\sum_{l=1}^{i-1}t_l)\,,\,\mathrm{if}\;j\neq f(i)\\
    c_x^j(\sum_{l=1}^{i-1}t_l)+(t-\sum_{l=1}^{i-1}t_l)\,,\,\mathrm{if}\;j=f(i)     
    \end{cases}$.
\end{itemize}
\end{mydef}
Intuitively, imagine a point going around the cactus clockwise with unit speed starting from its basepoint. Then as the point moves on the $j$-th lobe, the $j$-th coordinate of $c_x$ increases with unit speed, while all other coordinates remain constant. It is clear from this description that a cactus $x$ is uniquely determined by its cactus map $c_x$. \par\indent
\textbf{The operadic structure on cacti}. For positive integers $k,n_1,\cdots,n_k$, there are operadic structure maps
\begin{equation}\label{eq:operadic structure on cacti}
\theta_{k,n_1,\cdots,n_k}: \mathrm{Cact}^{n_1}\times\cdots\times\mathrm{Cact}^{n_k}\times \mathrm{Cact}^k\rightarrow \mathrm{Cact}^{n},\;\mathrm{where}\;n=\sum_{i=1}^kn_i.    
\end{equation}
In words, the operadic composition of $(x_1,\cdots,x_k,x)$ is obtained from $x$ by inserting the cactus $x_j$ into its $j$-th lobe (via a piecewise linear identification) in a way such that the basepoint of $x_j$ coincides with the \emph{local basepoint} of the $j$-th lobe of $x$. The local basepoint of a lobe is given by the basepoint if the lobe contains the basepoint; else it is the intersection of that lobe with the connected component of the closure of its complement that contains the basepoint. The precise definition of $\theta$ can be formulated in terms of cactus maps.\par\indent
Namely, there exists a unique piecewise linear map $\alpha: [0,k]\rightarrow [0,n]$ and a piecewise oriented isometry $c: [0,n]\rightarrow\prod_{j=1}^k[0,n_j]$ making the following diagram commute
\begin{center}
\begin{tikzcd}[row sep=1.2cm, column sep=0.8cm]
[0,k]\arrow[r,"{c_x}"]\arrow[d,"{\alpha}"]& {[0,1]}^k\arrow[d,"{(t_1,\cdots,t_k)\mapsto(n_1t_1,\cdots,n_kt_k)}"] \\
{[0,n]} \arrow[r,"c"] &\prod_{j=1}^k[0,n_j]
\end{tikzcd}\label{eq:cacti operadic composition diagram}
\end{center}
\begin{mydef}\label{thm:operadic structure on cacti}
$\theta(x_1,\cdots,x_k;x)$ is the cactus whose cactus map is
\begin{equation}\label{eq:cactus map of operadic composition}
(\prod_{j=1}^kc_{x_j})\circ c.    
\end{equation}
\end{mydef}
It turns out that $\theta$ is only associative up to homotopy, due to a rescaling issue (since by definition, the `total length' of a cactus is normalized to $1$). However, we still have the following. \\  
\begin{thm}(\cite[Theorem 4.10, Corollary 4.11]{Sal1})\label{thm:associativity of cacti operad}
$\theta$ defines an operadic structure up to homotopy on $\{\mathrm{Cact}^k\}_{k\geq 1}$. Furthermore, it induces a (strictly associative) dg operadic structure on 
\begin{equation}\label{eq:cacti dg operad}
\mathrm{Cact}_R:=\{C^{cell}_{-*}(\mathrm{Cact}^k;R)\}_{k\geq 1},    
\end{equation}
where the cells are products of simplices coming from the construction of $\mathrm{Cact}^k$ as a multi-semi-simplcial realization (cf. Definition \ref{thm:cacti}).
\end{thm}
\begin{figure}[H]
 \centering
 \includegraphics[width=0.85\textwidth]{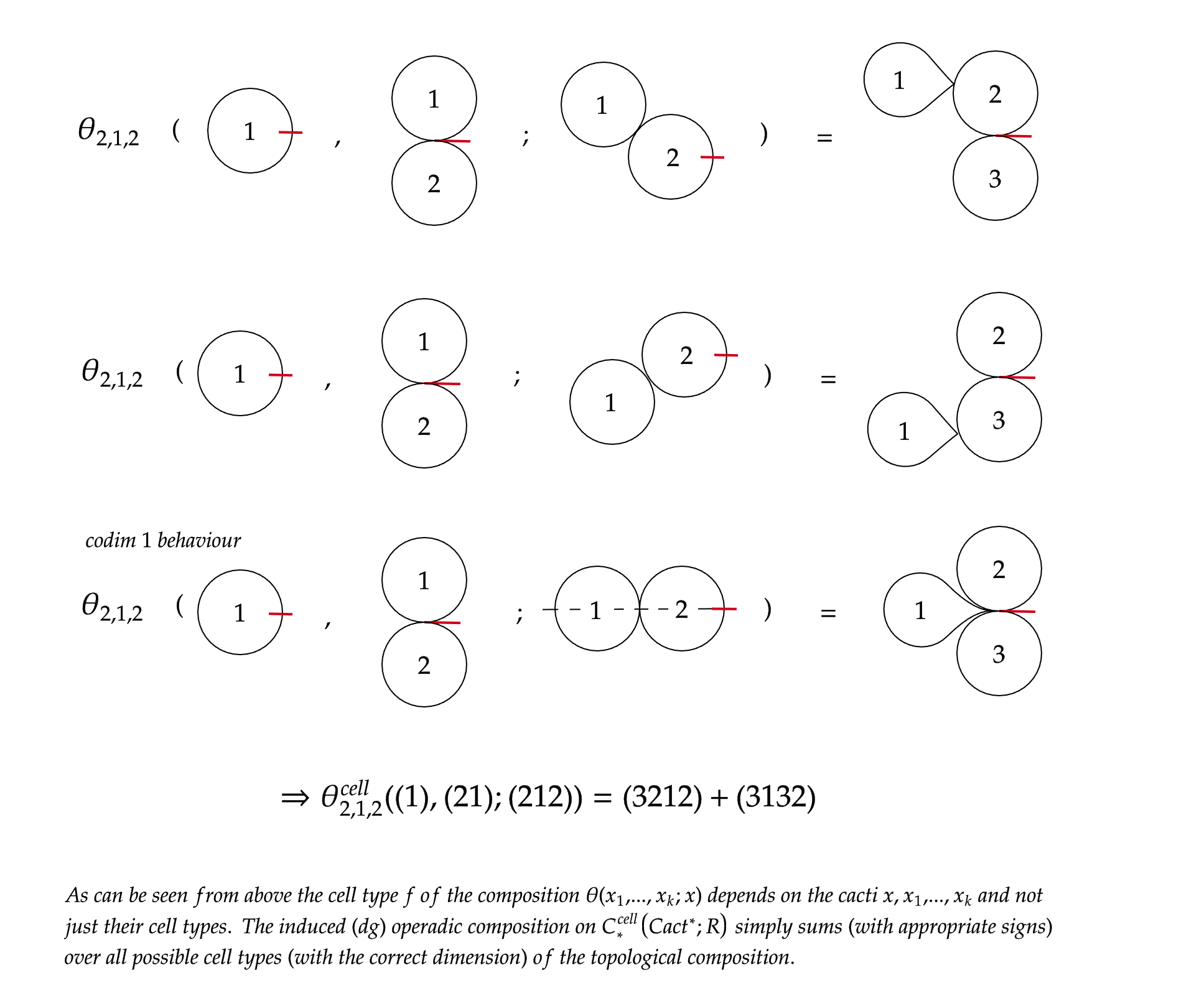}
 \caption{}
 \label{fig:cacti_composition}
\end{figure}
See Figure \ref{fig:cacti_composition} for an illustration of the operadic composition. \par\indent
\textbf{Action on Hochschild cochains}. A key property of the $R$-linear cacti operad is its action on normalized Hochschild cochains, which can be viewed as a version of Deligne's conjecture. \\
\begin{thm}(\cite{MS})\label{thm: action of cacti on HH^*}
For a dg or d($\mathbb{Z}/2$)g algebra $\mathcal{A}$ over $R$, there is an action of the $R$-linear cactus operad on $CC^*(\mathcal{A})$, i.e. there are $R$-linear $\Sigma_k$-equivariant chain maps (for $k\geq 1$)
\begin{equation}\label{eq: action of cacti on HH^*}
\mathrm{Act}_{\mathcal{A}}: \mathrm{Cact}^k_R\rightarrow \mathrm{Hom}_R(CC^*(\mathcal{A})^{\otimes_R k},CC^*(\mathcal{A}))    
\end{equation}
compatible with operadic compositions. 
\end{thm}

\subsubsection{Cyclic cacti}
For $k\geq 0$, consider the space of cacti with $k+1$ lobes $\mathrm{Cact}^{k+1}$, with the lobes labeled $0,1,\cdots,k$. Recall from Section 3.1.1 that each cactus is uniquely determined by a partition of $S^1\cong \mathbb{R}/(k+1)\mathbb{Z}$ into closed $1$-dimensional submanifolds $I_0,\cdots,I_k$ (recall $I_j$ is the union of intervals with color $j$).\\
\begin{mydef}\label{thm:space of cyclic cacti}
The \emph{space of cyclic cacti with $k$ lobes} $\mathrm{Cact}^k_{\circlearrowright}$ is the space $\mathrm{Cact}^{k+1}\times S^1$. For $k\geq 1$, we equip it with the $S^1\times S^1$-action given by the product of the $S^1$-action on $\mathrm{Cact}^{k+1}$ that clockwisely rotates the basepoint along the periphery of the cactus, i.e.  
\begin{equation}
t\cdot(I_0,\cdots,I_k)= (I_0-t,\cdots,I_k-t) \quad (t\in S^1\cong \mathbb{R}/(k+1)\mathbb{Z})
\end{equation}
with the regular action of $S^1$. 
Pictorially, we think of the second $S^1$ factor of $\mathrm{Cact}^{k}_{\circlearrowright}$ as an extra marked point moving along the $0$-th lobe; this requires a choice of a reference point on the $0$-th lobe, which we take to be the intersection of the $0$-th lobe with the connected component of its complement closure that contains the $1$-st lobe. We also call the $0$-th lobe the \emph{base lobe}. The basepoint of the underlying (non-cyclic) cactus $x$ will be called the \emph{input basepoint}; the extra marked point on the base lobe will be called the \emph{output basepoint}; the \emph{adjusted local basepoint} of the $j$-th lobe ($1\leq j\leq k$) is its intersection with the connected component of its complement closure that contains the base lobe, and if $j=0$ it is the output basepoint. See Figure \ref{fig:cyclic_cacti} for an example of a cyclic cactus.\par\indent
For $k=0$, we think of $S^1=\mathrm{Cact}^0_{\circlearrowright}$ as measuring the clockwise angle from the input basepoint to the output basepoint; it is equipped with the $S^1\times S^1$-action via the quotient $S^1\times S^1/\mathrm{diag}\cong S^1$.  
\begin{figure}[H]
 \centering
 \includegraphics[width=0.9\textwidth]{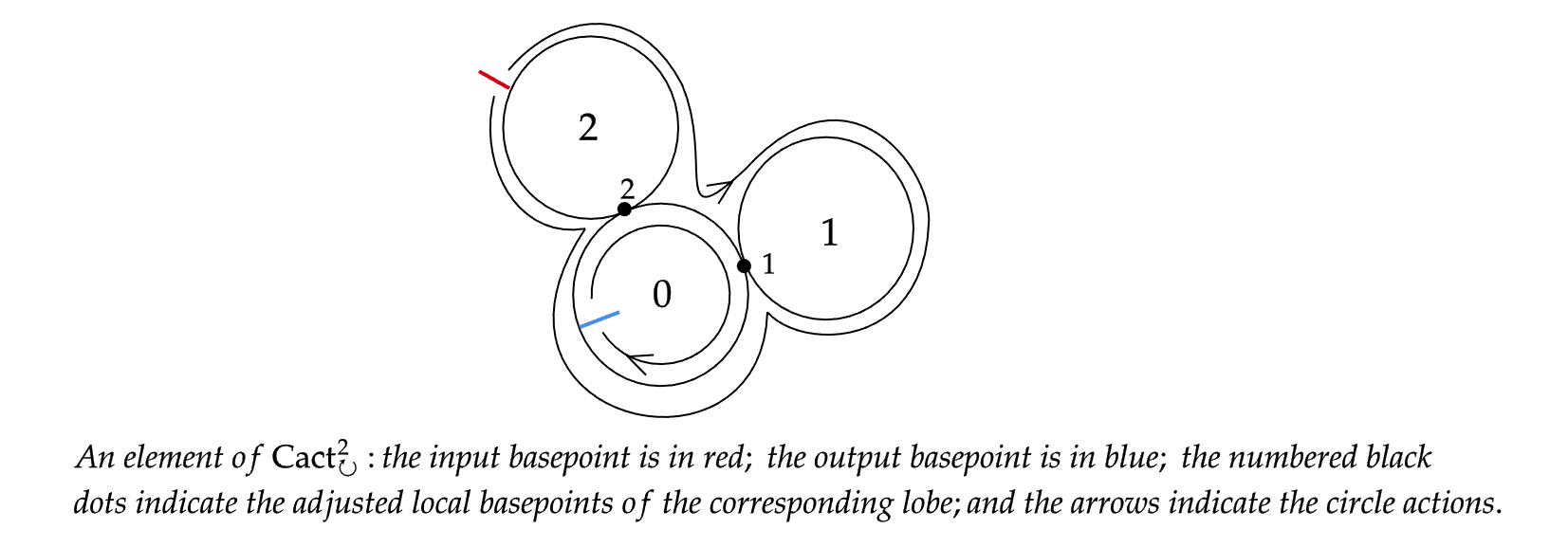}
 \caption{}
 \label{fig:cyclic_cacti}
\end{figure}
\end{mydef}
The `cactus map' of a cyclic cactus $(x,\theta)\in \mathrm{Cact}^k_{\circlearrowright}$ is defined to be the pair $(c_x,\theta)$, where $c_x$ is as in Definition \autoref{thm:cactus map}.\par\indent
\textbf{Cellular structure}. We equip the space of cyclic cacti with the following cellular structure. 
\begin{itemize}
    \item When $k=0$, the space $\mathrm{Cact}^0_{\circlearrowright}= S^1$ is equipped with the standard cellular structure on the circle with one $0$-cell and one $1$-cell.
    \item More generally, recall from Definition \ref{thm:cacti} that a cell $f$ of $\mathrm{Cact}^{k+1}$ corresponds to a product of simplices $\Delta^{i_0}\times \Delta^{i_1}\times\cdots\times\Delta^{i_k}$. Identify $I=[0,1]$ with the fundamental chain of $S^1$, and take the prismatic subdivision $I\times \Delta^{i_0}=\bigcup_{l=0}^{i_0}\Delta^{i_0+1}_l$ equipped with the induced simplicial complex structure, where each top dimensional simplex $\Delta^{i_0+1}_l$ is a copy of the standard simplex $\Delta^{i_0+1}$. We then equip $\mathrm{Cact}^k_{\circlearrowright}\cong S^1\times \mathrm{Cact}^{k+1}$ with the cellular structure where the cells are of the form
    \begin{equation}\label{eq:cells for cyclic cacti}
    K\times \Delta^{i_1}\times\cdots\times \Delta^{i_k},  
    \end{equation}
    where $K$ is a simplex of the simplicial complex $I\times \Delta^{i_0}$. We furthermore equip \eqref{eq:cells for cyclic cacti} with the product orientation where each simplex 
    \begin{equation}
     \Delta^{r}=\{(x_0,\cdots,x_r)\,|\,\sum_{i=0}^r{x_i}=1\;,\;x_i\geq 0\}   
    \end{equation}
    is given the standard orientation induced by $dx_1\wedge \cdots \wedge dx_r$.
\end{itemize}
With respect to this cellular structure, denote 
\begin{equation}\label{eq:dg cyclic cacti operad}
 \mathrm{Cact}_{\circlearrowright,R}^k:=C_{-*}^{cell}(\mathrm{Cact}^k_{\circlearrowright};R).
\end{equation}
From now on, we will represent cyclic cacti cells by drawings as in Figure \ref{fig:graphical_representation_of_cyclic_cacti_cell};
\begin{figure}[H]
 \centering
 \includegraphics[width=1.1\textwidth]{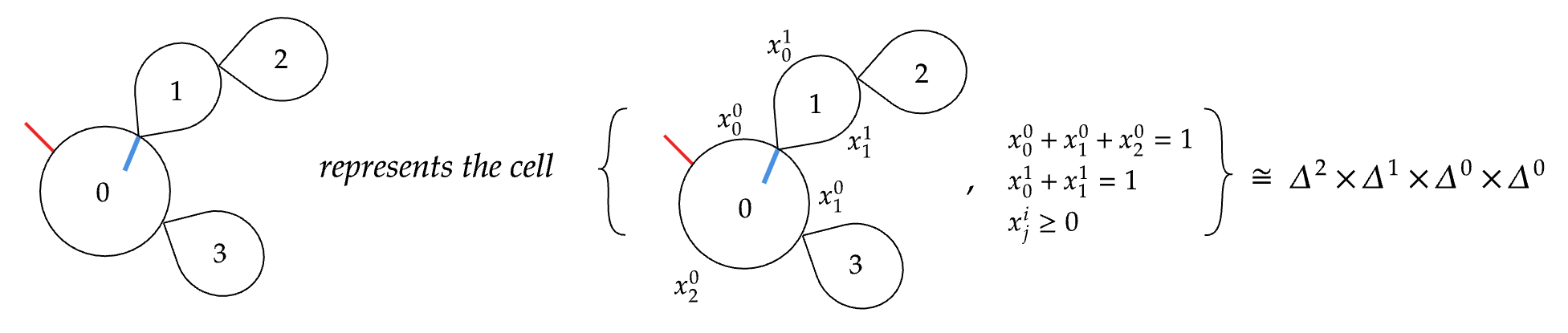}
 \caption{}
 \label{fig:graphical_representation_of_cyclic_cacti_cell}
\end{figure}
see Figure \ref{fig:cells_in_cyclic_cacti_and_their_actions} for more examples of cyclic cacti cells (where we omit the numbering of lobes) and their action on the Hochschild cochain/chain complex (explained in more details in Theorem \ref{thm: action of cyclic cacti on HH_* and HH^*}). 
 \begin{figure}[H]
 \centering
 \includegraphics[width=0.9\textwidth]{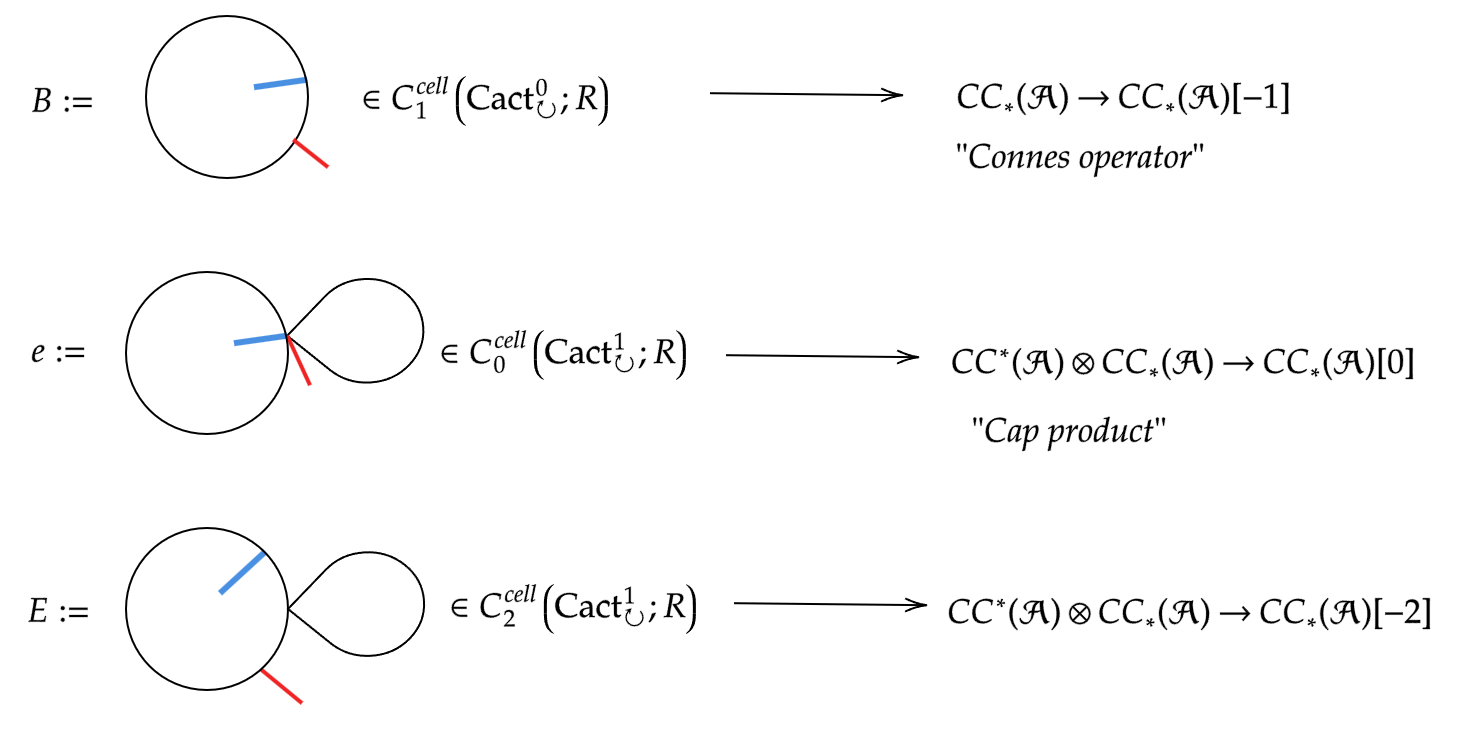}
 \caption{}
 \label{fig:cells_in_cyclic_cacti_and_their_actions}
\end{figure}
\textbf{The operadic structure on cyclic cacti}. We define structure maps
\begin{equation}\label{eq:cyclic cacti operadic structure type 0}
\theta^{\circlearrowright,0}_j: \mathrm{Cact}^l\times \mathrm{Cact}^k_{\circlearrowright}\rightarrow \mathrm{Cact}^{k+l-1}_{\circlearrowright}, \quad1\leq j\leq k    
\end{equation}
and 
\begin{equation}\label{eq:cyclic cacti operadic structure type 1}
\theta^{\circlearrowright,1}: \mathrm{Cact}^l_{\circlearrowright}\times \mathrm{Cact}^k_{\circlearrowright}\rightarrow \mathrm{Cact}^{k+l}_{\circlearrowright} 
\end{equation}
such that when combined with $\theta$ from \eqref{eq:operadic structure on cacti} makes $\{\mathrm{Cact}^l,\mathrm{Cact}^k_{\circlearrowright}\}_{l,k\geq 1}$ into a two colored topological operad up to homotopy. \\
\begin{mydef}\label{thm:operadic structure on cyclic cacti}
\begin{itemize}
    \item $\theta^{\circlearrowright,0}_j$ inserts the first cactus into the $j$-th lobe of the second cyclic cactus (upon applying a piecewise linear oriented isometry) so that the basepoint of the former matches the adjusted local basepoint of the latter. 
    \item $\theta^{\circlearrowright,1}$ inserts the second cyclic cactus into the base lobe of the first cyclic cactus (upon applying a piecewise linear oriented isometry) so that the input basepoint of the second cyclic cactus matches the output basepoint of the first cyclic cactus. The base lobe of the second cyclic cactus will become the base lobe of the composition. 
\end{itemize}
See Figure \ref{fig:composition_of_cyclic_cacti} for an illustration and \cite[Equation (4.35) and (4.37)]{Che3} for a formulaic definition in terms of the cactus map. 
\begin{figure}[H]
 \centering
 \includegraphics[width=0.9\textwidth]{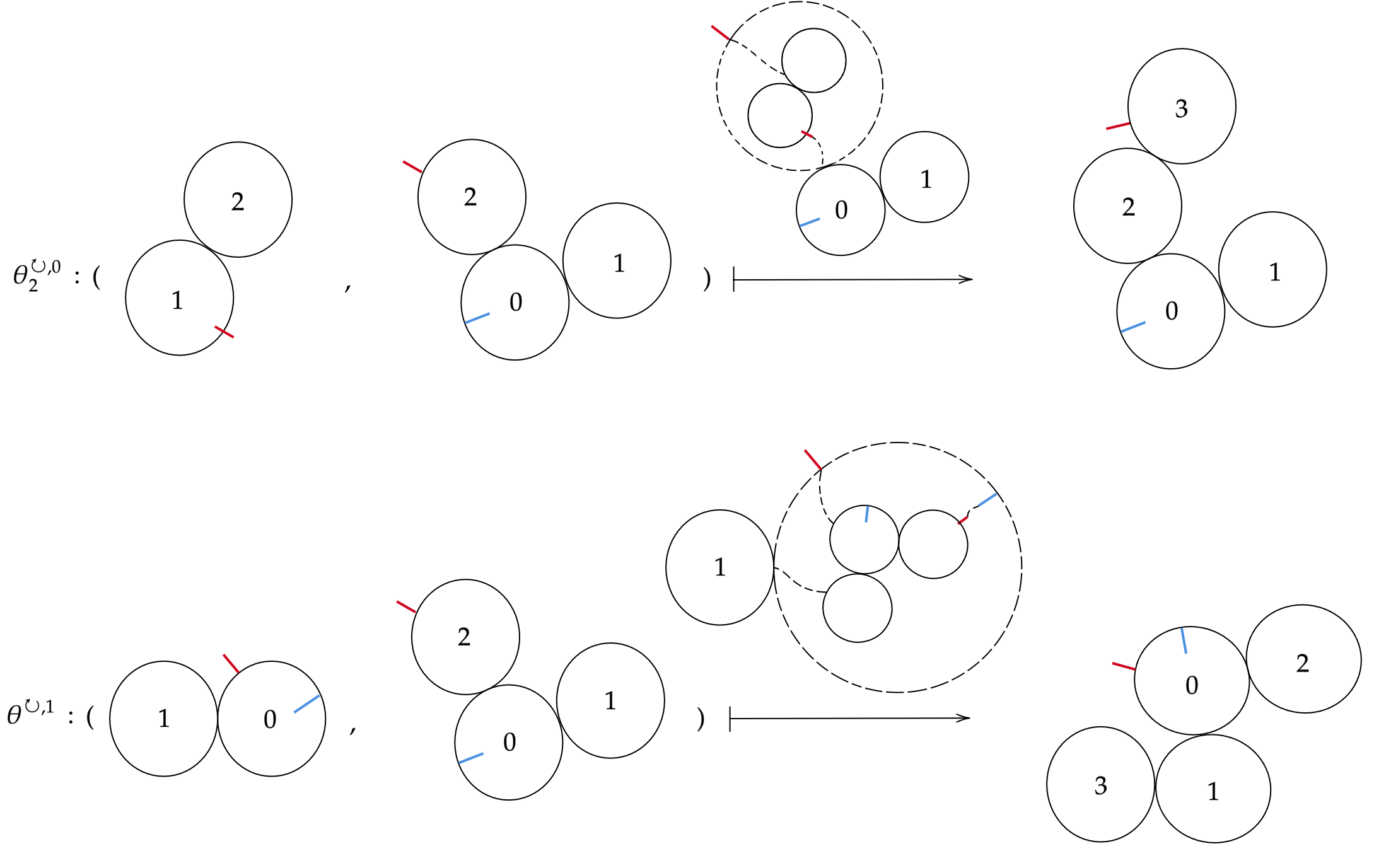}
 \caption{}
 \label{fig:composition_of_cyclic_cacti}
\end{figure}
\end{mydef}
Parallel to the case of ordinary cacti (cf. Theorem \autoref{thm:associativity of cacti operad}), $\theta^{\circlearrowright,0},\theta^{\circlearrowright,1}$ are only associative up to homotopy. Nonetheless, they are cellular maps and moreover induce a strictly associative operadic structure on the associated cellular operad. We record this as Lemma \ref{thm: associativity and equivariance of cyclic cacti operad}, whose proof follows from exactly the same argument as \cite[Theorem 4.10 and Corollary 4.11]{Sal1}.  \\
\begin{lemma}\label{thm: associativity and equivariance of cyclic cacti operad}
\begin{itemize}
    \item The maps $\theta, \theta^{\circlearrowright,0},\theta^{\circlearrowright,1}$ induce on
    \begin{equation}
     \mathrm{Cact}_{2col}:=\{\mathrm{Cact}^l,\mathrm{Cact}^k_{\circlearrowright}\}_{l,k\geq 1}  
    \end{equation}the structure of a two-colored topological operad up to homotopy.
    \item The maps on cellular complexes induced by $\theta, \theta^{\circlearrowright,0},\theta^{\circlearrowright,1}$, denoted $\theta_{cell}, \theta^{\circlearrowright,0}_{cell},\theta^{\circlearrowright,1}_{cell}$, equip
    \begin{equation}
     \mathrm{Cact}_{2col,R}:=\{C_{-*}^{cell}(\mathrm{Cact}^l;R),C_{-*}^{cell}(\mathrm{Cact}^k_{\circlearrowright};R)\}_{l,k\geq 1}   
    \end{equation}
    the structure of a (strict) two-colored dg operad.  \qed
\end{itemize}
\end{lemma}
\begin{figure}[H]
 \centering
 \includegraphics[width=1.0\textwidth]{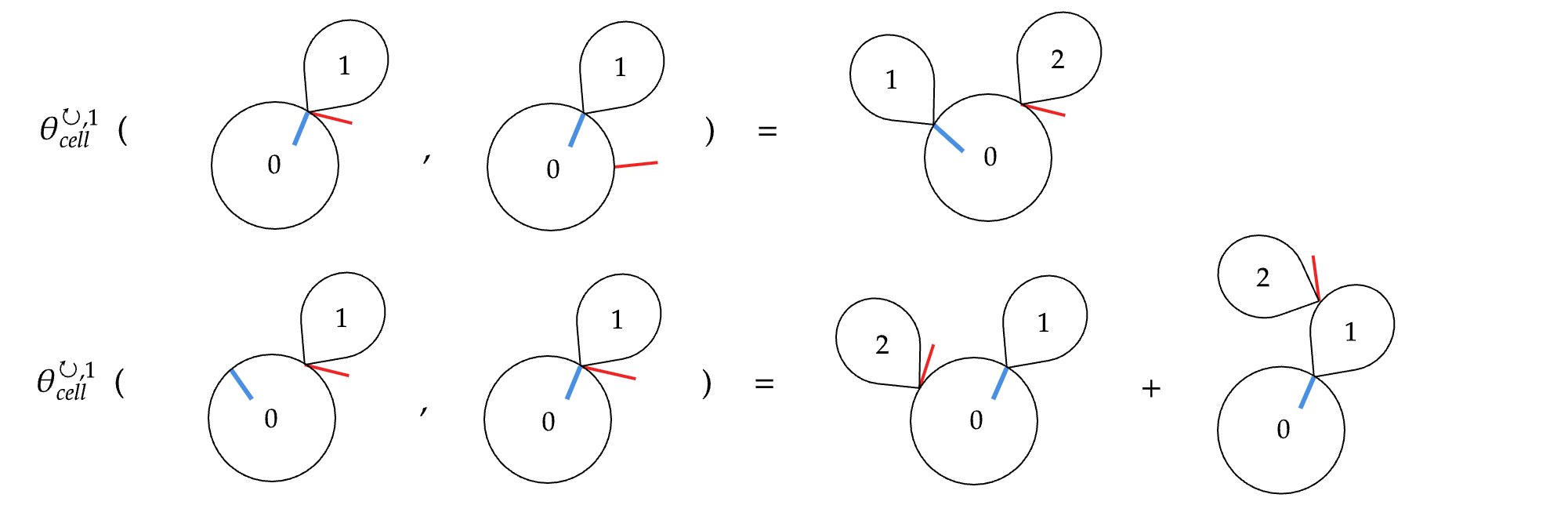}
 \caption{Some examples of the composition $\theta^{\circlearrowright,1}_{cell}$.}
 \label{fig:example_of_cyclic_cacti_cellular_composition}
\end{figure}

\textbf{Action on (normalized) Hochschild cochains and chains}. The following theorem is a generalization of Theorem \ref{thm: action of cacti on HH^*} to the case of the two-colored operad $\mathrm{Cact}_{2col,R}$, and was proved in \cite[Lemma 4.28 and Theorem 4.26]{Che3}. We remark that this statement also has counterparts using `minimal operads', which appeared in earlier works of Kontsevich-Soibelman e.g. \cite[Theorem 11.3.1]{KS1}.\\
\begin{thm} \label{thm: action of cyclic cacti on HH_* and HH^*}
For a dg or d$(\mathbb{Z}/2$)g algebra $\mathcal{A}$ over $R$, there are chain maps
\begin{equation}\label{eq: action of cyclic cacti on HH_* and HH^*}
\mathrm{Act}_{\mathcal{A}}: \mathrm{Cact}^k_{\circlearrowright,R}\rightarrow \mathrm{Hom}_R(CC^*(\mathcal{A})^{\otimes_R k}\otimes_R CC_*(\mathcal{A}), CC_*(\mathcal{A}))    
\end{equation}
such that when combined with \eqref{eq: action of cacti on HH^*}, make the pair $(CC^*(\mathcal{A}),CC_*(\mathcal{A}))$ into an algebra over the two-colored operad $\mathrm{Cact}_{2col,R}$. Moreover, the chain map \eqref{eq: action of cyclic cacti on HH_* and HH^*} is $\Sigma_k\times S^1\times S^1$-equivariant. Here, the $\Sigma_k$-action on the left hand side is given by permuting the labels of the lobes numbered $1,\cdots,k$ in the cyclic cactus, and on the right hand side is given by permuting the tensor factors of $CC^*(\mathcal{A})$; the $S^1\times S^1$-action on the left hand side is induced by rotations of the output/input marked points (cf. Definition \ref{thm:space of cyclic cacti}) and on the right hand side is induced by the standard circle actions on the two copies of $CC_*(\mathcal{A})$. \qed
\end{thm}
In words, the action \eqref{eq: action of cyclic cacti on HH_* and HH^*} is defined as follows. Fix Hochschild cochains $\varphi_1,\cdots,\varphi_k$, Hochschild chain $a_0\otimes a_1\otimes\cdots\otimes a_n$ and a cyclic cacti cell $\alpha$. Put $\varphi_i$ in the $i$-th lobe of $\alpha$; sum (with appropriate Koszul signs) over all ways to insert $a_0,a_1,\cdots,a_n$ in a clockwise manner along the circumference of $\alpha$ with $a_0$ inserted at the input (red) basepoint (and that no other $a_i$ coincide with a basepoint); one then evaluates the inserted $a_0,\cdots,a_n$ by the various $\varphi_i$'s according to the shape of $\alpha$, and read off the resulting Hochschild chain $b_0\otimes b_1\otimes\cdots\otimes b_m$ lying on the $0$-th lobe clockwisely, with $b_0$ being the entry at the output (blue) basepoint. See \cite[Definition 4.27]{Che3} for more details. \par\indent
See Figure \ref{fig:cells_in_cyclic_cacti_and_their_actions} right for the cyclic cacti cells corresponding to the various chain-level operations appearing in the formula for the Getzler-Gauss-Manin connection (cf. Section 2.3). We refer the readers to \cite[Section 4.2]{Che3} for a more detailed discussion of the circle actions and the equivariance property in Theorem \ref{thm: action of cyclic cacti on HH_* and HH^*}. \par\indent
Since \eqref{eq: action of cyclic cacti on HH_* and HH^*} is $\Sigma_k$-equivariant, it induces a chain map
\begin{equation}\label{eq: unordered action of cyclic cacti on HH_* and HH^*}
\mathrm{Act}_{\mathcal{A}}^{\Sigma_k}: \mathrm{UCact}^k_{\circlearrowright,R}\rightarrow \mathrm{Hom}_R((CC^*(\mathcal{A})^{\otimes_R k})^{h\Sigma_k}\otimes_R CC_*(\mathcal{A}), CC_*(\mathcal{A})),
\end{equation}
where we denote
\begin{equation}\label{eq:unordered cyclic cacti}
 \mathrm{UCact}^k_{\circlearrowright,R}:=\mathrm{Cact}^k_{\circlearrowright,R}/\Sigma_k\simeq (\mathrm{Cact}^k_{\circlearrowright,R})_{h\Sigma_k},
\end{equation}
the last equivalence following from the freeness of the $\Sigma_k$-action. Intuitively, generators of $\mathrm{UCact}^k_{\circlearrowright,R}$ correspond to \emph{unordered cyclic cacti}, i.e. the labels on the lobes numbered $1,\cdots,k$ are forgotten (but there is still a distinguished $0$-th lobe/baselobe). 

\subsection{A topological model}
In this subsection, we recall the construction of a topological model for the Kontsevich-Soibelman operad based on configuration spaces \cite{KS1}\cite{DTT}, which has the advantage of being amenable to computations. \par\indent
The values of the two-colored operad $\{\mathrm{Cyl}(l,0),\mathrm{Cyl(k,1)}\}_{l\geq 1,k\geq 0}$ is defined as follows. We set $\{\mathrm{Cyl}(l,0)\}_{l\geq 1}$ to be the little disk operad.\par\indent
On the other hand, $\mathrm{Cyl}(n,1), n\geq 0$ is the topological space of 
\begin{itemize}
    \item When $n\geq1$, cylinder $S^1\times [a,c],a<c$ together with a configuration of $n$ disks on $S^1\times(a,c)$ and two marked points $in\in S^1\times\{c\}, out\in S^1\times\{a\}$. When $n=0$, $\mathrm{Cyl}(0,1)$ is the space of two points $in,out\in S^1$ modulo simultaneous rotation. 
    \item These configurations are considered up to equivalence generated by parallel shifts $S^1\times[a,c]\rightarrow S^1\times [a+l,c+l]$ and overall $S^1$-rotation of the cylinder. 
\end{itemize}
The two kinds of operadic compositions involving $\mathrm{Cyl}(n,1)$
\begin{equation}\label{eq:cyl operadic 1}
\lambda_i:\mathrm{Cyl}(n,0)\times_i \mathrm{Cyl}(n',1)\rightarrow \mathrm{Cyl}(n+n'-1,1)  
\end{equation}
\begin{equation}\label{eq:cyl operadic 2}
\eta:\mathrm{Cyl}(n,1)\times \mathrm{Cyl}(n',1)\rightarrow \mathrm{Cyl}(n+n',1)    
\end{equation}
are given by insertion of configuration of little disks, and stacking two cylinders on top of each other while matching the $out$ of the first cylinder with the $in$ of the second cylinder (after a rotation), respectively; see Figure \ref{fig:figure3.5}. These should be viewed as counterparts of the operadic compositions on $\mathrm{Cact}_{2col}$ introduced in Definition \ref{thm:operadic structure on cyclic cacti}.
\begin{figure}[H]
 \centering
 \includegraphics[width=0.9\textwidth]{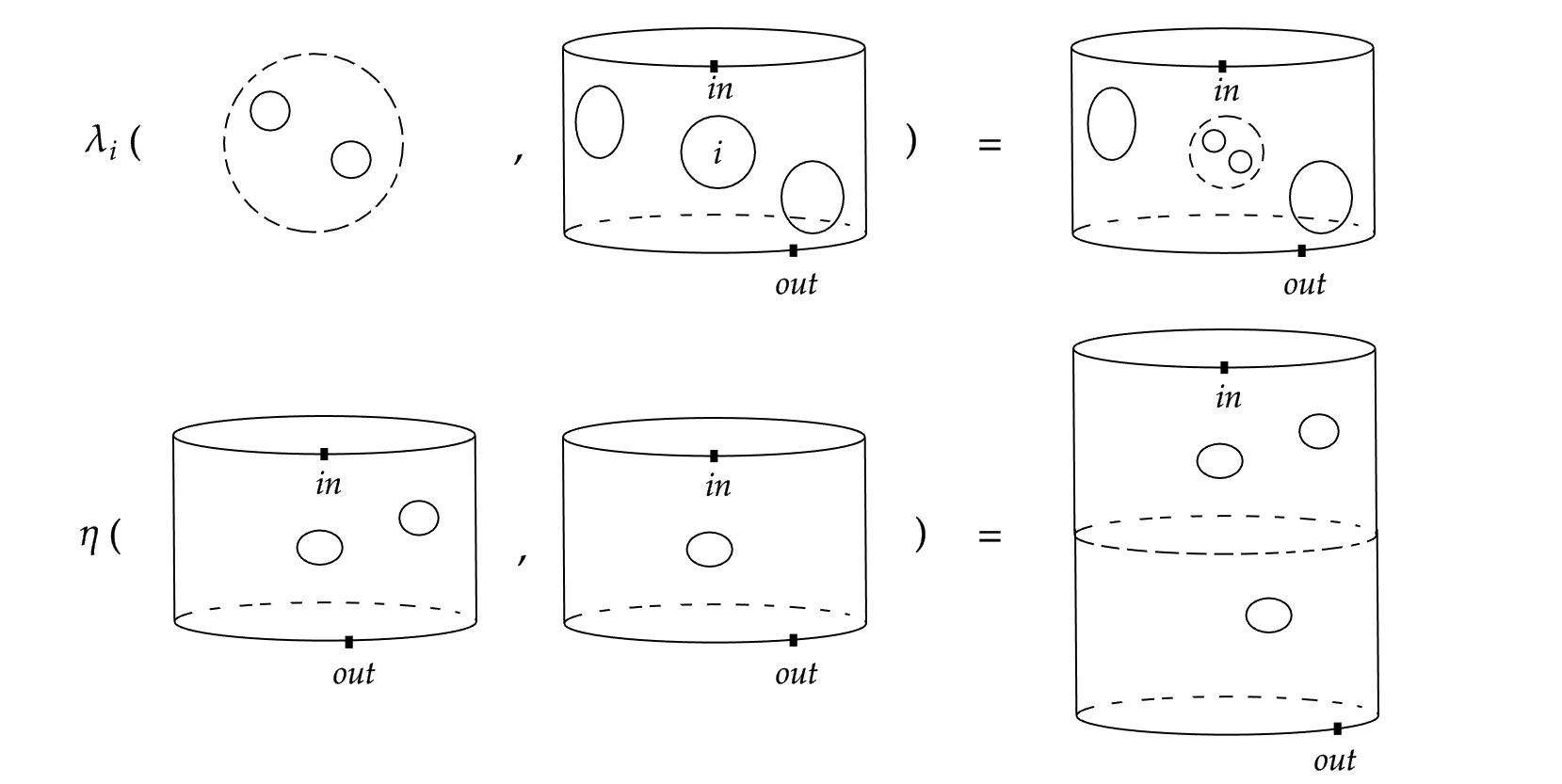}
 \caption{}
 \label{fig:figure3.5}
\end{figure}
Note that there is a homotopy equivalence 
\begin{equation}\label{eq:cyl as conf}
\mathrm{Cyl}(n,1)\simeq \mathrm{Conf}_n(\mathbb{C}^*)\times S^1,
\end{equation}
given as follows. Up to homotopy, we may assume everything lies on the standard cylinder $S^1\times [0,1]$ and up to $S^1$-rotation, we fix $in$ to be at $(0,1)\in S^1\times\{1\}$. Then, upon identifying $S^1\times (0,1)\cong \mathbb{C}^*$, we record the positions of the centers of the disks and the position of $out$, which gives the desired homotopy equivalence. \par\indent
The following comparison result relates the two different models for the Kontsevich-Soibelman operad. For details of the proof, we refer the readers to \cite[Theorem 4.26 and Theorem 4.30]{Che3}; we also remark that a version using the `minimal operad' model was proved in earlier works of \cite{KS1} and \cite{Wil}.\\ 
\begin{thm}\label{thm:comparison of Cact with Cyl}
$\mathrm{Cact}_{2col,R}$ is quasi-equivalent to the two-colored dg operad of singular chains on $\mathrm{Cyl}$. Moreover, the quasi-equivalence $\mathrm{Cact}^n_R\simeq C_{-*}(\mathrm{Cyl}(n,0);R)$ is $\Sigma_n$-equivariant and the quasi-equivalence $\mathrm{Cact}^n_{\circlearrowright,R}\simeq C_{-*}(\mathrm{Cyl}(n,1);R)$ is $\Sigma_n\times S^1\times S^1$-equivariant. 
\end{thm}\qed

\subsection{Kontsevich-Soibelman operations}
Consider the action map \eqref{eq: unordered action of cyclic cacti on HH_* and HH^*}, which we recall is $\Sigma_n\times S^1\times S^1$-equivariant. By taking appropriate $S^1$ homotopy orbits or fixed points (cf. \cite[Section 4.4]{Che3} for details), one obtains a map
\begin{equation}\label{eq: equivariant action of cyclic cacti on HH_* and HH^*}
\Xi_{\mathcal{A}}: H^*((\mathrm{UCact}^k_{\circlearrowright,R})^{hS^1})\rightarrow \mathrm{Hom}_{H^*_{S^1}(pt;R)}(H^*((CC^*(\mathcal{A})^{\otimes_R k})^{h\Sigma_k})\otimes_R HH^-_*(\mathcal{A}), HH^-_*(\mathcal{A})).        
\end{equation}
Sometimes it is convenient to consider the localized version, i.e. after inverting the generator $t\in H^2_{S^1}(pt;R)$, one obtains a map (still denoted by the same letter)
\begin{equation}\label{eq: equivariant action of cyclic cacti on HH_* and HH^*}
\Xi_{\mathcal{A}}: H^*((\mathrm{UCact}^k_{\circlearrowright,R})^{tS^1})\rightarrow \mathrm{Hom}_{H^*_{S^1}(pt;R)}(H^*((CC^*(\mathcal{A})^{\otimes_R k})^{h\Sigma_k})\otimes_R HH^{per}_*(\mathcal{A}), HH^{per}_*(\mathcal{A})).        
\end{equation}
In particular, given elements 
\begin{equation}
[\phi]\in H^*((CC^*(\mathcal{A})^{\otimes_R k})^{h\Sigma_k}),\quad[\alpha]\in H^*((\mathrm{UCact}^k_{\circlearrowright,R})^{hS^1})\;\;(\mathrm{resp.}\;\;H^*((\mathrm{UCact}^k_{\circlearrowright,R})^{tS^1}),
\end{equation}
there is an $H^*_{S^1}(pt;R)$-linear endomorphism of $ HH^-_*(\mathcal{A})$ (resp.  $HH^{per}_*(\mathcal{A})$) 
\begin{equation}
 \Xi_{\mathcal{A}}([\alpha])([\phi],-).   
\end{equation}
Such endomorphisms are called \emph{Kontsevich-Soibelman operations} on the negative (resp. periodic) cyclic homology of $\mathcal{A}$.\par\indent
When the base ring $R$ has characteristic $p$, it is often useful to study variants of Kontsevich-Soibelman operations by restricting to the finite cyclic subgroup $C_p\subset S^1$, and we denote $\Xi^p_{\mathcal{A}}$ for the analogue of $\Xi_{\mathcal{A}}$ in this case. We record the following commutative diagram induced from the restriction of homotopy fixed points along $C_p\subset S^1$, cf. \cite[Section 5.1]{Che3}.
\begin{equation}\label{eq:restricting KS action}
\begin{tikzcd}[row sep=1.2cm, column sep=0.8cm]
H^*(\mathrm{UCact}^k_{\circlearrowright,R})^{tS^1})\otimes_{H^*_{S^1}(pt;R)} HH_*^{per}(\mathcal{A})\arrow[d,"{\mathrm{Res}_{C_p\subset S^1}}\;\otimes\; {\mathrm{Res}_{C_p\subset S^1}}"]\arrow[rrr,"{\Xi_{\mathcal{A}}(-)([\phi],-)}"]& & & HH_*^{per}(\mathcal{A})\arrow[d,"{\mathrm{Res}_{C_p\subset S^1}}"]    \\
H^*(\mathrm{UCact}^k_{\circlearrowright,R})^{tC_p})\otimes_{H^*_{C_p}(pt;R)} HH_*^{C_p,per}(\mathcal{A})\arrow[rrr,"{\Xi^p_{\mathcal{A}}(-)([\phi],-)}"]& & &HH_*^{C_p,per}(\mathcal{A})
\end{tikzcd}    
\end{equation}
Here, $HH_*^{C_p,per}(\mathcal{A})$ stands for the $C_p$-Tate fixed points of the $C_p\subset S^1$ action on $CC_*(\mathcal{A})$; see \cite[Section 2.3]{Che3} for an explicit chain model.

\section{Multiplicative property of $p$-curvature}
Throughout this section, we work over a base field $\mathbf{k}$ of odd characteristic $p$ unless otherwise specified. The main goal of this section is to prove the following theorem. \\
\begin{thm}\label{thm: p curvature equals multiplicative operation}
Let $\mathcal{A}$ be a d($\mathbb{Z}/2$)g algebra over $\mathbf{k}$, and denote by $F^{\mathcal{A}}_{2t^2\partial_t}$ the $p$-curvature of the $t$-connection (cf. Definition \ref{thm:t connection}) on the periodic cyclic homology along the vector field $2t^2\partial_t$. Then, there exists a map
\begin{equation}
\bigcap\nolimits^{C_p}: HH^{even}(\mathcal{A})\times HH_*^{per}(\mathcal{A})\rightarrow  HH_*^{per}(\mathcal{A})   
\end{equation}
such that
\begin{equation}\label{eq:multiplicative plus Frobenius linear plus zero property}
\bigcap\nolimits^{C_p}([\phi]\cup[\varphi],\alpha)=  \bigcap\nolimits^{C_p}([\phi],\bigcap\nolimits^{C_p}([\varphi],\alpha))\quad,\quad\bigcap\nolimits^{C_p}([a\phi],\alpha)=a^p\bigcap\nolimits^{C_p}([\phi],\alpha)\quad,\quad \bigcap\nolimits^{C_p}(0,\alpha)=0
\end{equation}
for all $[\phi],[\varphi]\in HH^{even}(\mathcal{A}), \alpha\in HH^{per}(\mathcal{A}), a\in \mathbf{k}$, together with a constant $c\in \mathbf{k}$ such that 
\begin{equation}
 F^{\mathcal{A}}_{2t^2\partial_t}=c\cdot\bigcap\nolimits^{C_p}([d],-).
\end{equation}
Here, $HH^{even}(\mathcal{A})$ denotes the even part of Hochschild cohomology, $\cup$ denotes the cup product and the differential $d$ is viewed as a Hochschild cohomology class of length one. 
\end{thm}
The proof of Theorem \ref{thm: p curvature equals multiplicative operation} will take up the rest of this section. Before that, we record the following corollary.\\
\begin{cor}\label{thm: nilpotence of p curvature}
Let $\mathcal{A}$ be a smooth d($\mathbb{Z}/2$)g algebra over $\mathbf{k}$, then the $p$-curvature of its categorical $t$-connection is nilpotent. 
\end{cor}
\emph{Proof}. By Theorem \ref{thm: p curvature equals multiplicative operation}, it suffices to show that $[d]\in HH^*(\mathcal{A})$ is nilpotent with respect to the cup product. To see this, recall that 
\begin{equation}
HH^*(\mathcal{A})=\mathrm{Hom}_{D(\mathcal{A}^e)}(\mathcal{A},\mathcal{A}),    
\end{equation}
where $\mathcal{A}^e=\mathcal{A}^{op}\otimes^{\mathbb{L}}_{\mathbf{k}} \mathcal{A}$ and $D(-)$ denotes its enhanced derived category. Let $B(\mathcal{A})\xrightarrow{\epsilon}\mathcal{A}$ be the bar resolution, and consider the increasing length filtration on $B(A)$ given by
\begin{equation}\label{eq:bar resolution}
B_{\leq N}(\mathcal{A}):=\bigoplus_{i\leq N} \mathcal{A}\otimes \mathcal{A}[1]^{\otimes i}\otimes \mathcal{A} \subset B(\mathcal{A}).
\end{equation}
By the smoothness assumption, $\mathcal{A}\in \mathrm{Perf}(\mathcal{A}^e)$ and hence $\mathrm{Hom}_{D(\mathcal{A}^e)}(\mathcal{A},-)$ commutes with filtered homotopy colimits. Applying this to the length filtration, we deduce that 
\begin{equation}
 \mathrm{Hom}_{D(\mathcal{A}^e)}(\mathcal{A},B(\mathcal{A})) =\mathrm{hocolim}_N\;\mathrm{Hom}_{D(\mathcal{A}^e)}(\mathcal{A},B_{\leq N}(\mathcal{A})).   
\end{equation}
In particular, the quasi-inverse $s: \mathcal{A}\rightarrow B(\mathcal{A})$ of $\epsilon$ factors through $\mathcal{A}\rightarrow B_{\leq N_0}(\mathcal{A})\subset B(\mathcal{A})$ for some $N_0$. On the other hand, there is a decreasing filtration on Hochschild cohomology induced by
\begin{equation}\label{eq:length filtration on HH^*}
F^{N}CC^*(\mathcal{A}):=\{\phi\in \mathrm{Hom}_{\mathcal{A}^e}(B(\mathcal{A}),\mathcal{A})\,|\,\phi|_{B_{\leq N-1}(\mathcal{A})}=0\},
\end{equation}
which moreover satisfies $F^i \cup F^j\subset F^{i+j}$. In particular, since $[d]\in F^1HH^*(\mathcal{A})$, $[d]^{\cup N_0+1}\in F^{N_0+1}HH^*(\mathcal{A})$. Consider the composition
\begin{equation}\label{eq:[d]^N_0 composed with s}
\mathcal{A}\xrightarrow{s} B(\mathcal{A})\xrightarrow{[d]^{\cup N_0+1}}\mathcal{A}.    
\end{equation}
The fact that $[d]^{\cup N_0+1}\in F^{N_0+1}HH^*(\mathcal{A})$ and that $s$ factors through $B_{\leq N_0}(\mathcal{A})$ implies that \eqref{eq:[d]^N_0 composed with s} is null, and hence (as $s$ has a quasi-inverse) so is $[d]^{\cup N_0+1}$. \qed

\subsection{$p$-curvature as Kontsevich-Soibelman operation}
The key to the proof of Theorem \ref{thm: p curvature equals multiplicative operation} is the interpretation of $p$-curvature (of the $t$-connection) in terms of the Kontsevich-Soibelman operad.  \par\indent
We first make some simplifications. Since $(t\partial_t)^p=t\partial_t$,
\begin{align}
F^{\mathcal{A}}_{t\partial_t}&=(\nabla^{\mathcal{A}}_{t\partial_t})^p-\nabla^{\mathcal{A}}_{t\partial_t}\\
&=(C+D)^p-(C+D),
\end{align}
where we denote
\begin{equation}
C:= t\partial_t+\frac{\Gamma}{2}\quad,\quad D:=\frac{\iota_d}{2t}.     
\end{equation}
To simplify this expression, we notice that
\begin{equation}
[C,D]=-\frac{1}{2}D.     
\end{equation}
\begin{lemma}\label{thm: simplify (A+B)^p using Jacobson}
Let $C,D$ be two elements in an algebra over a field of characteristic $p>2$ such that $[C,D]=-\frac{1}{2}D$. Then
\begin{equation}
(C+D)^p=C^p+D^p+D.     
\end{equation}
\end{lemma}
\emph{Proof}. Recall the Jacobson identity
\begin{equation}
(C+D)^p=C^p+D^p+\sum_{i=1}^{p-1} s_i(C,D),    
\end{equation}
where $s_i(C,D)$ are Lie polynomials in $C,D$ defined via
\begin{equation}\label{eq: lie polynomial identity}
\sum_{i=1}^{p-1}is_i(C,D)\lambda^{i-1}=ad_{\lambda C+D}^{p-1}(C).    
\end{equation}
Since $[C,D]=-\frac{1}{2}D$, a straightforward computation shows that
\begin{equation}
ad_{\lambda C+D}^{p-1}(C)= (-1)^{p-2}(\frac{1}{2})^{p-1}\lambda^{p-2}D=-\lambda^{p-2}D.
\end{equation}
Matching coefficients in \eqref{eq: lie polynomial identity} implies $s_{p-1}(C,D)=D$ and $s_i=0$ for $i\neq p-1$, which completes the proof. \qed\par\indent
Therefore, 
\begin{equation}
F^{\mathcal{A}}_{t\partial_t}=(C^p+D^p+D)-(C+D)=C^p-C+D^p=D^p,  
\end{equation}
where in the last equality we used the fact that the operation $C=t\partial_t+\frac{\Gamma}{2}$ simply multiplies a cyclic chain by a half integer, and hence by Fermat's little theorem $C^p-C=0$ (recall that $p>2$). Finally, since $p$-curvature is Frobenius-linear with respect to the vector field, we immediately obtain that\\
\begin{cor}\label{thm: simplified p curvature}
The $p$-curvature of the chain level $t$-connection satisfies
\begin{equation}\label{eq: simplified p curvature}
F^{\mathcal{A}}_{2t^2\partial_t}=(\iota_d)^p=(e_d+tE_d)^p.
\end{equation}    
\end{cor}\qed\par\indent
We rewrite \eqref{eq: simplified p curvature} pedagogically as
\begin{equation}
\mathrm{Act}_{\mathcal{A}}^{\Sigma_p}((e+tE)^p)(d,-),
\end{equation}
where $e, E$ are the cyclic cacti cells from Figure \ref{fig:cells_in_cyclic_cacti_and_their_actions} and $\mathrm{Act}_{\mathcal{A}}^{\Sigma_p}$ is the chain-level action map (extended $t$-linearly) from \eqref{eq: unordered action of cyclic cacti on HH_* and HH^*}. \par\indent
We view
\begin{equation}\label{eq: cell expression of p curvature}
 (e+tE)^p   
\end{equation}
as a degree $0$ element (treating $t$ as a formal variable of degree $2$) of the chain complex
\begin{equation}\label{eq: explicit chain model for S^1 fixed point of UCact}
(C_{-*}^{cell}(\mathrm{UCact}^p_{\circlearrowright};\mathbf{k})((t)), d_{cell}-t[B,-])
\end{equation}
computing $(\mathrm{UCact}^p_{\circlearrowright,\mathbf{k}})^{tS^1}$, where $[B,-]$ denotes the graded commutator. We elaborate on the notations in \eqref{eq: cell expression of p curvature} and \eqref{eq: explicit chain model for S^1 fixed point of UCact}. 
\begin{itemize}
    \item By definition
\begin{equation}
 (e+tE)^p:= \overbrace{(e+tE)\circ \cdots\circ (e+tE)}^{p\;\mathrm{times}},   
\end{equation}
where $\circ$ is shorthand for the operadic composition $\theta^{\circlearrowright,1}_{cell}$ of Definition \ref{thm:operadic structure on cyclic cacti} with the convention 
\begin{equation}\label{eq:change of composition order}
\alpha\circ \beta:=(-1)^{|\alpha||\beta|}\theta_{cell}^{\circlearrowright,1}(\beta,\alpha)
\end{equation}
(or more precisely, the induced composition on \emph{unordered} cyclic cacti cells).
\item $d_{cell}$ is the differential of the cellular chain complex.
    \item Recall there is an $S^1\times S^1$-action on $\mathrm{UCact}^p_{\circlearrowright}$ where the two copies of $S^1$ rotate the input (resp. output) marked points; at the dg level, the induced (anti-commuting) actions by the fundamental classes of the two copies of $S^1$ are given by $(-1)^{|(-)|}(-)\circ B$ and $-B\circ (-)$, respectively, where $B\in C_1^{cell}(\mathrm{Cact^0_{\circlearrowright}})$ is the generator representing the `Connes operator', cf. Figure \ref{fig:cells_in_cyclic_cacti_and_their_actions}. Therefore, $(\mathrm{UConf}^p_{\circlearrowright,\mathbf{k}})^{tS^1}$ (recall the $S^1$ action is the diagonal one) is indeed computed by  
    \begin{equation}
     (C_{-*}^{cell}(\mathrm{UCact}^p_{\circlearrowright};\mathbf{k})((t)), d_{cell}-t[B,-]).   
    \end{equation}
\end{itemize}\par\indent
\begin{mydef}
Fix a d($\mathbb{Z}/2$)g algebra $\mathcal{A}$ over $\mathbf{k}$ and a cocycle $[\phi]\in (CC^*(\mathcal{A})^{\otimes k})^{h\Sigma_k}$. We say that $x,y\in C_{-*}^{cell}(\mathrm{UCact}^k_{\circlearrowright};\mathbf{k})$ are \emph{$(\mathcal{A},[\phi])$-equivalent}, denoted
\begin{equation}
x\overset{(\mathcal{A},[\phi])}{\sim} y,
\end{equation}
if $\mathrm{Act}^{\Sigma_k}_{\mathcal{A}}(x)([\phi],-)=\mathrm{Act}^{\Sigma_k}_{\mathcal{A}}(y)([\phi],-)$ as endomorphisms of the $\mathbf{k}$-vector space $CC_*(\mathcal{A})$. We extend this to an equivalence relation on $C_{-*}^{cell}(\mathrm{UCact}^k_{\circlearrowright};\mathbf{k})((t))$ by $t$-linearity.
\end{mydef}
Note that since $[\phi]$ is a cocycle and $\mathrm{Act}^{\Sigma_k}_{\mathcal{A}}$ is a map of (two-colored) dg operads, $\overset{(\mathcal{A},[\phi])}{\sim}$ is closed under differential $d_{cell}$ and operadic composition $\circ$.\\
\begin{rmk}
In plain words, $x,y$ are $(\mathcal{A},[\phi])$-equivalent if they cannot be told apart by their action (with respect to $[\phi]$) on the Hochschild/cyclic chain complex of $\mathcal{A}$. We emphasize that $(\mathcal{A},[\phi])$-equivalence is a fundamentally chain-level notion, and is dependent on the specific chain models we are using for e.g. the Kontsevich-Soibelman operad, Hochschild invariants.  
\end{rmk}
For subsequent graphical calculations, it is convenient to introduce another set of (chain-level) equivalence relations for $C_{-*}^{cell}(\mathrm{UCact}^k_{\circlearrowright};\mathbf{k})((t))$.\\
\begin{mydef}
Let $\sim$ denote the linear equivalence relation on $C_{-*}^{cell}(\mathrm{UCact}^k_{\circlearrowright};\mathbf{k})$ generated by local equivalences described in Figure \ref{fig:local_relations_1}, where the lobe drawn in the center of the figure is \emph{not} the baselobe. Extend this to an equivalence relation on $C_{-*}^{cell}(\mathrm{UCact}^p_{\circlearrowright};\mathbf{k})((t))$ by $t$-linearity.
\end{mydef}
It is straightforward to check that $\sim$ is closed under the differential and operadic compositions. \\
\begin{lemma}\label{thm:sim equivalence implies other equivalence}
$x\sim y$ implies $x\overset{(\mathcal{A},[d^{\otimes k}])}{\sim} y$, where $d$ denotes the differential of $\mathcal{A}$ viewed as a Hochschild cocycle.     
\end{lemma}
\emph{Proof}. We check the statement for each of the relations in Figure \ref{fig:local_relations_1}. Relations of type (a) and (c) follow from $d^2=0$ and the fact $d$ is a derivation for the graded algebra $(\mathcal{A},\cdot)$. Relation of type (b) follows from the fact that $d$ is a length one Hochschild cochain and hence vanishes when evaluated at more than one inputs. \qed\par\indent
\begin{figure}[H]
 \centering
 \includegraphics[width=1.0\textwidth]{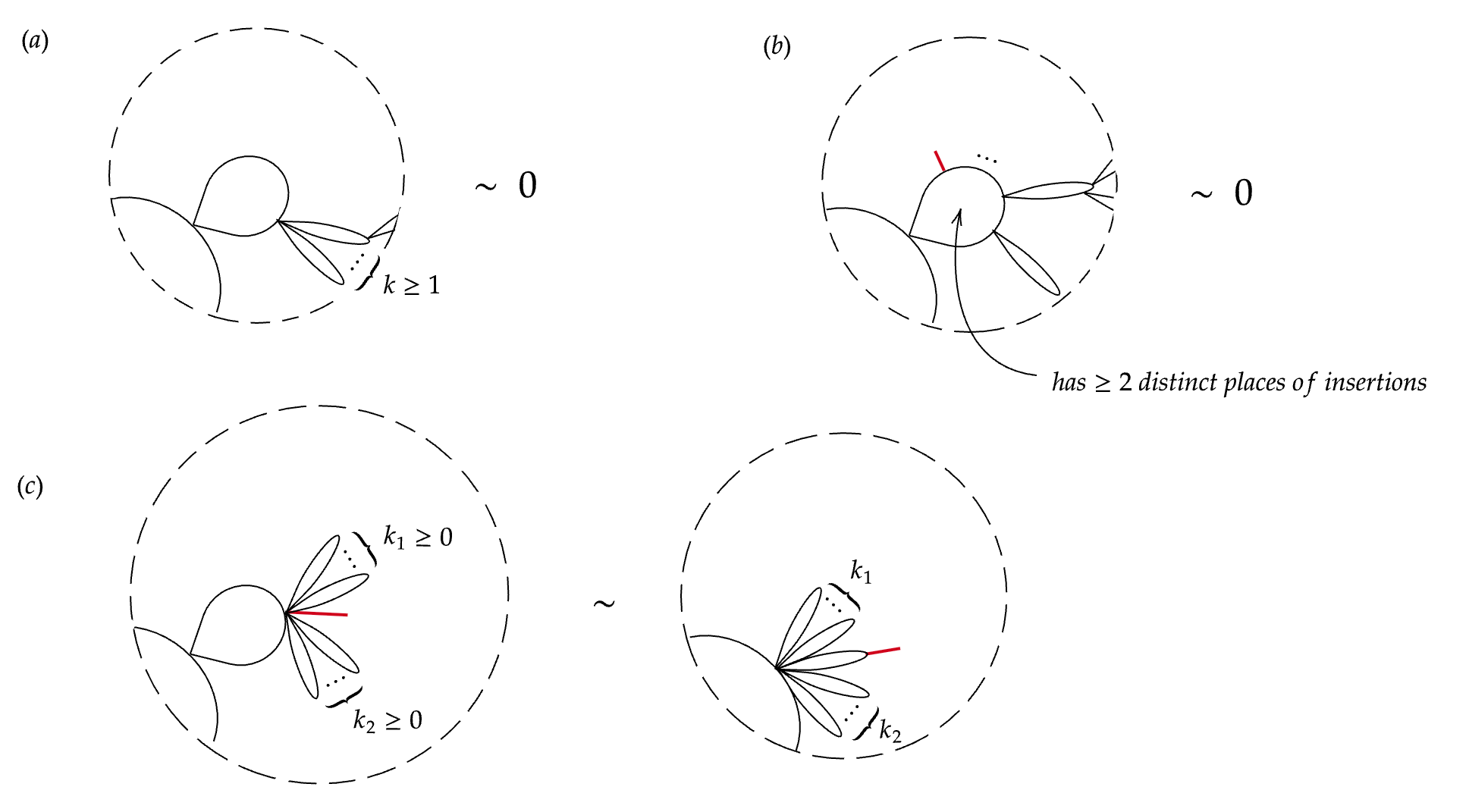}
 \caption{}
 \label{fig:local_relations_1}
\end{figure}   

Another local relation involving the composition $\theta^{\circlearrowright,1}_{cell}$ of two cyclic cacti cells that will often be used is
\begin{figure}[H]
 \centering
 \includegraphics[width=1.0\textwidth]{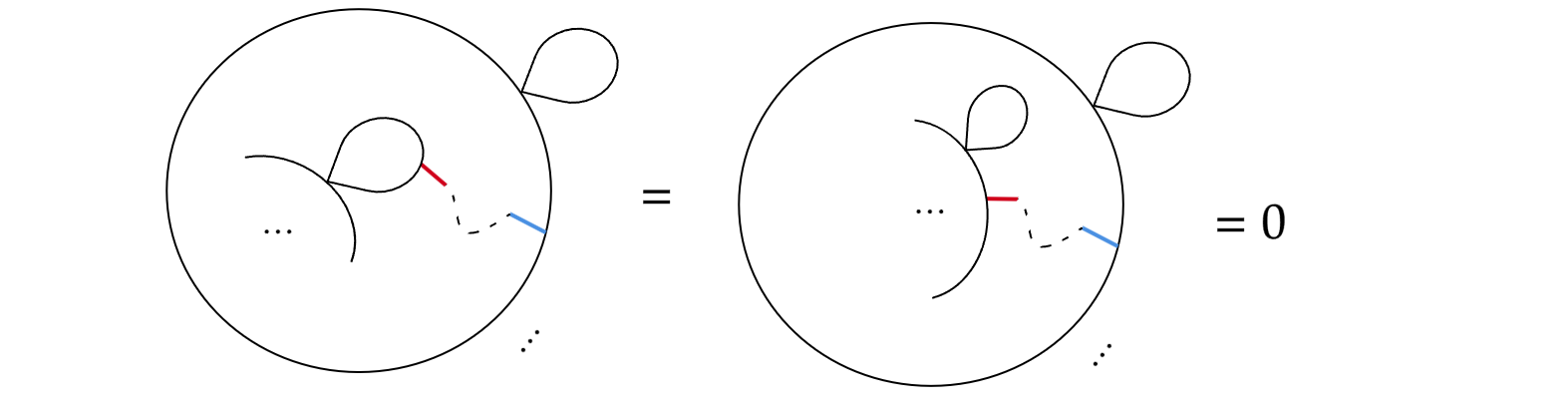}
 \caption{}
 \label{fig:local_relations_2}
\end{figure}  
The relation \ref{fig:local_relations_2} reflects the fact that composing via $\theta^{\circlearrowright,1}_{cell}$ (cf. Definition \ref{thm:operadic structure on cyclic cacti}) two cyclic cacti cells such that the matching input/output marked points are both unconstrained gives a degenerate cell. Algebraically, \ref{fig:local_relations_2} corresponds to the fact that inserting the unit into a normalized Hochschild cochain gives zero. \par\indent
We say that a cyclic cacti cell is \emph{reduced} if the adjusted local basepoint (cf. Definition \ref{thm:space of cyclic cacti}) of the $j$-th lobe, for each $1\leq j\leq k$, lies on the baselobe (note this also makes sense for unordered cyclic cacti cell). By induction on the maximal depth (with respect to the dual tree of the cyclic cactus cell) of a lobe, one obtains the following lemma, see Figure \ref{fig:reduction_of_cyclic_cacti_cell} for an illustration.\\ 
\begin{lemma}\label{thm:reduction of cyclic cacti}
For each (unordered) cyclic cacti cell $x$ either $x\sim 0$ or there is a unique reduced (unordered) cyclic cacti cell $x'$ such that $x\sim x'$. In either case we call $0$ (resp. $x'$) the \emph{reduction} of $x$ and denote it $x_{red}$, and extend $(-)_{red}$ by linearity to the $C_{-*}^{cell}(\mathrm{UCact}^k_{\circlearrowright};\mathbf{k})((t))$. \qed
\end{lemma}
\begin{figure}[H]
 \centering
 \includegraphics[width=1.0\textwidth]{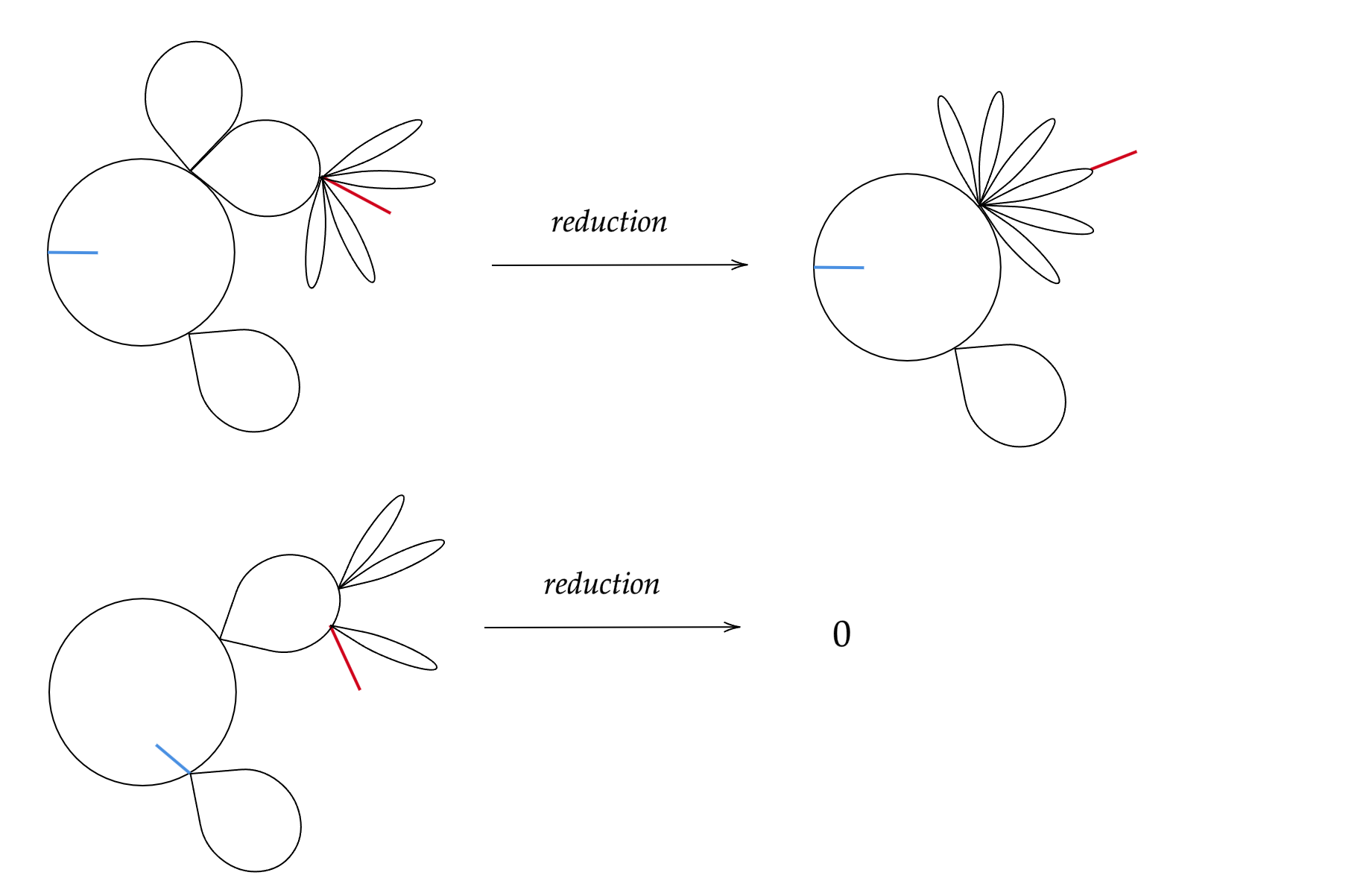}
 \caption{}
 \label{fig:reduction_of_cyclic_cacti_cell}
\end{figure} 
We now make the following simple yet key observation.\\
\begin{prop}\label{thm: main proposition}
$(e+tE)^p$ is $\sim$ (and hence $(\mathcal{A},[d^{\otimes p}])$-equivalent to) a degree $0$ cocycle of $(C_{-*}^{cell}(\mathrm{UCact}^p_{\circlearrowright};\mathbf{k})((t)), d_{cell}-t[B,-])$.    
\end{prop}
\emph{Proof.} Note that by Lemma \ref{thm:reduction of cyclic cacti}, it suffices to show that
\begin{equation}\label{eq:differential reduces to 0}
d_{cell}(e+tE)^p-t[B,(e+tE)^p]\sim 0 \in    C_{-*}^{cell}(\mathrm{UCact}^p_{\circlearrowright};\mathbf{k})((t)).  
\end{equation}
Indeed, by the compatibility of $\sim$ with $d_{cell}$ and operadic compositions, \eqref{eq:differential reduces to 0} will imply that
\begin{equation}d_{cell}((e+tE)^p)_{red}-t[B,((e+tE)^p)_{red}]\sim 0.\end{equation} 
On the other hand, $((e+tE)^p)_{red}$ is a linear combination of reduced cells and applying $d_{cell}$ or $[B,-]$ preserves reducedness, and thus by uniqueness of reduction one must have
\begin{equation}d_{cell}((e+tE)^p)_{red}-t[B,((e+tE)^p)_{red}]= 0,\end{equation} 
which gives the desired statement.\par\indent
Since $C_{-*}^{cell}(\mathrm{UCact}^p_{\circlearrowright};\mathbf{k})$ is part of a two-colored dg operad,
\begin{equation}\label{eq:boundary of e+tE pth power}
d_{cell}(e+tE)^p=\sum_{i=1}^p(e+tE)^{i-1}d_{cell}(e+tE)(e+tE)^{p-i}.
\end{equation}
On the other hand, taking a telescoping sum gives
\begin{equation}\label{eq:[-,B] of e+tE pth power}
[(e+tE)^p,B]=   \sum_{i=1}^p(e+tE)^{i-1}[(e+tE),B](e+tE)^{p-i}.
\end{equation}
Note that no Koszul sign appears since $e+tE$ has even degree. Combining \eqref{eq:boundary of e+tE pth power} and \eqref{eq:[-,B] of e+tE pth power} gives 
\begin{equation}\label{eq:equivariant boundary of e+tE pth power}
(d_{cell}-t[B,-]) (e+tE)^p=   \sum_{i=1}^p(e+tE)^{i-1}\big(d_{cell}(e+tE)-t[B,(e+tE)]\big)(e+tE)^{p-i}
\end{equation}
Define $L\in C^{cell}_{1}(\mathrm{UCact^p_{\circlearrowright};\mathbf{k})}$ by
\begin{equation}\label{eq:Lie action L}
tL:=  d_{cell}(e+tE)-t[B,(e+tE)]= t(d_{cell}(E)+[e,B]), 
\end{equation}
where the second equality above follows from the fact that $e$ has dimension $0$ and that $[E,B]=0$ (by the relation in Figure \ref{fig:local_relations_2}). An explicit graphical calculation then shows

\tikzset{every picture/.style={line width=0.75pt}} 

\begin{equation}\label{eq:graphical Cartan homotopy formula}
\begin{tikzpicture}[x=0.75pt,y=0.75pt,yscale=-0.7,xscale=0.7]

\draw [color={rgb, 255:red, 239; green, 36; blue, 36 }  ,draw opacity=1 ][line width=1.5]    (263,125.5) -- (284,123.6) ;
\draw [color={rgb, 255:red, 74; green, 144; blue, 226 }  ,draw opacity=1 ][line width=2.25]    (223,108.6) -- (222.5,85) ;
\draw   (182,125.5) .. controls (182,103.13) and (200.13,85) .. (222.5,85) .. controls (244.87,85) and (263,103.13) .. (263,125.5) .. controls (263,147.87) and (244.87,166) .. (222.5,166) .. controls (200.13,166) and (182,147.87) .. (182,125.5) -- cycle ;
\draw   (241.12,64.64) .. controls (241.12,64.64) and (241.12,64.64) .. (241.12,64.64) .. controls (239.84,49.8) and (250.78,35.71) .. (265.55,33.17) .. controls (280.33,30.63) and (293.35,40.61) .. (294.63,55.45) .. controls (295.92,70.29) and (284.98,84.38) .. (270.2,86.92) .. controls (252.36,89.98) and (243.42,91.55) .. (243.38,91.6) .. controls (243.42,91.55) and (242.66,82.55) .. (241.12,64.64) -- cycle ;
\draw [color={rgb, 255:red, 239; green, 36; blue, 36 }  ,draw opacity=1 ][line width=1.5]    (239.38,258.6) -- (267.7,265.6) ;
\draw [color={rgb, 255:red, 74; green, 144; blue, 226 }  ,draw opacity=1 ][line width=2.25]    (230.7,279.6) -- (239.38,258.6) ;
\draw   (178,292.5) .. controls (178,270.13) and (196.13,252) .. (218.5,252) .. controls (240.87,252) and (259,270.13) .. (259,292.5) .. controls (259,314.87) and (240.87,333) .. (218.5,333) .. controls (196.13,333) and (178,314.87) .. (178,292.5) -- cycle ;
\draw   (237.12,231.64) .. controls (237.12,231.64) and (237.12,231.64) .. (237.12,231.64) .. controls (237.12,231.64) and (237.12,231.64) .. (237.12,231.64) .. controls (235.84,216.8) and (246.78,202.71) .. (261.55,200.17) .. controls (276.33,197.63) and (289.35,207.61) .. (290.63,222.45) .. controls (291.92,237.29) and (280.98,251.38) .. (266.2,253.92) .. controls (248.36,256.98) and (239.42,258.55) .. (239.38,258.6) .. controls (239.42,258.55) and (238.66,249.55) .. (237.12,231.64) -- cycle ;
\draw [color={rgb, 255:red, 239; green, 36; blue, 36 }  ,draw opacity=1 ][line width=1.5]    (420,279.5) -- (441,277.6) ;
\draw [color={rgb, 255:red, 74; green, 144; blue, 226 }  ,draw opacity=1 ][line width=2.25]    (380,262.6) -- (379.5,239) ;
\draw   (339,279.5) .. controls (339,257.13) and (357.13,239) .. (379.5,239) .. controls (401.87,239) and (420,257.13) .. (420,279.5) .. controls (420,301.87) and (401.87,320) .. (379.5,320) .. controls (357.13,320) and (339,301.87) .. (339,279.5) -- cycle ;
\draw [color={rgb, 255:red, 239; green, 36; blue, 36 }  ,draw opacity=1 ][line width=1.5]    (455,125.5) -- (476,123.6) ;
\draw [color={rgb, 255:red, 74; green, 144; blue, 226 }  ,draw opacity=1 ][line width=2.25]    (427.7,128.6) -- (455,125.5) ;
\draw   (374,125.5) .. controls (374,103.13) and (392.13,85) .. (414.5,85) .. controls (436.87,85) and (455,103.13) .. (455,125.5) .. controls (455,147.87) and (436.87,166) .. (414.5,166) .. controls (392.13,166) and (374,147.87) .. (374,125.5) -- cycle ;
\draw   (433.12,64.64) .. controls (433.12,64.64) and (433.12,64.64) .. (433.12,64.64) .. controls (431.84,49.8) and (442.78,35.71) .. (457.55,33.17) .. controls (472.33,30.63) and (485.35,40.61) .. (486.63,55.45) .. controls (487.92,70.29) and (476.98,84.38) .. (462.2,86.92) .. controls (444.36,89.98) and (435.42,91.55) .. (435.38,91.6) .. controls (435.42,91.55) and (434.66,82.55) .. (433.12,64.64) -- cycle ;
\draw [color={rgb, 255:red, 239; green, 36; blue, 36 }  ,draw opacity=1 ][line width=1.5]    (640,124.5) -- (661,122.6) ;
\draw [color={rgb, 255:red, 74; green, 144; blue, 226 }  ,draw opacity=1 ][line width=2.25]    (607.7,112.6) -- (620.38,90.6) ;
\draw   (559,124.5) .. controls (559,102.13) and (577.13,84) .. (599.5,84) .. controls (621.87,84) and (640,102.13) .. (640,124.5) .. controls (640,146.87) and (621.87,165) .. (599.5,165) .. controls (577.13,165) and (559,146.87) .. (559,124.5) -- cycle ;
\draw   (618.12,63.64) .. controls (618.12,63.64) and (618.12,63.64) .. (618.12,63.64) .. controls (616.84,48.8) and (627.78,34.71) .. (642.55,32.17) .. controls (657.33,29.63) and (670.35,39.61) .. (671.63,54.45) .. controls (672.92,69.29) and (661.98,83.38) .. (647.2,85.92) .. controls (629.36,88.98) and (620.42,90.55) .. (620.38,90.6) .. controls (620.42,90.55) and (619.66,81.55) .. (618.12,63.64) -- cycle ;
\draw [color={rgb, 255:red, 239; green, 36; blue, 36 }  ,draw opacity=1 ][line width=1.5]    (795.38,91.6) -- (824.7,97.6) ;
\draw [color={rgb, 255:red, 74; green, 144; blue, 226 }  ,draw opacity=1 ][line width=2.25]    (775,108.6) -- (774.5,85) ;
\draw   (734,125.5) .. controls (734,103.13) and (752.13,85) .. (774.5,85) .. controls (796.87,85) and (815,103.13) .. (815,125.5) .. controls (815,147.87) and (796.87,166) .. (774.5,166) .. controls (752.13,166) and (734,147.87) .. (734,125.5) -- cycle ;
\draw   (793.12,64.64) .. controls (793.12,64.64) and (793.12,64.64) .. (793.12,64.64) .. controls (791.84,49.8) and (802.78,35.71) .. (817.55,33.17) .. controls (832.33,30.63) and (845.35,40.61) .. (846.63,55.45) .. controls (847.92,70.29) and (836.98,84.38) .. (822.2,86.92) .. controls (804.36,89.98) and (795.42,91.55) .. (795.38,91.6) .. controls (795.42,91.55) and (794.66,82.55) .. (793.12,64.64) -- cycle ;
\draw [color={rgb, 255:red, 239; green, 36; blue, 36 }  ,draw opacity=1 ][line width=1.5]    (232.38,485.6) -- (251.7,499.6) ;
\draw [color={rgb, 255:red, 74; green, 144; blue, 226 }  ,draw opacity=1 ][line width=2.25]    (210.7,449.6) -- (219.38,428.6) ;
\draw   (158,462.5) .. controls (158,440.13) and (176.13,422) .. (198.5,422) .. controls (220.87,422) and (239,440.13) .. (239,462.5) .. controls (239,484.87) and (220.87,503) .. (198.5,503) .. controls (176.13,503) and (158,484.87) .. (158,462.5) -- cycle ;
\draw   (217.12,401.64) .. controls (215.84,386.8) and (226.78,372.71) .. (241.55,370.17) .. controls (256.33,367.63) and (269.35,377.61) .. (270.63,392.45) .. controls (271.92,407.29) and (260.98,421.38) .. (246.2,423.92) .. controls (228.36,426.98) and (219.42,428.55) .. (219.38,428.6) .. controls (219.42,428.55) and (218.66,419.55) .. (217.12,401.64) -- cycle ;
\draw [color={rgb, 255:red, 239; green, 36; blue, 36 }  ,draw opacity=1 ][line width=1.5]    (465.38,375.6) -- (480.7,359.6) ;
\draw [color={rgb, 255:red, 74; green, 144; blue, 226 }  ,draw opacity=1 ][line width=2.25]    (414.7,450.6) -- (423.38,429.6) ;
\draw   (362,463.5) .. controls (362,441.13) and (380.13,423) .. (402.5,423) .. controls (424.87,423) and (443,441.13) .. (443,463.5) .. controls (443,485.87) and (424.87,504) .. (402.5,504) .. controls (380.13,504) and (362,485.87) .. (362,463.5) -- cycle ;
\draw   (421.12,402.64) .. controls (419.84,387.8) and (430.78,373.71) .. (445.55,371.17) .. controls (460.33,368.63) and (473.35,378.61) .. (474.63,393.45) .. controls (475.92,408.29) and (464.98,422.38) .. (450.2,424.92) .. controls (432.36,427.98) and (423.42,429.55) .. (423.38,429.6) .. controls (423.42,429.55) and (422.66,420.55) .. (421.12,402.64) -- cycle ;
\draw [color={rgb, 255:red, 239; green, 36; blue, 36 }  ,draw opacity=1 ][line width=1.5]    (637.38,431.6) -- (665.7,438.6) ;
\draw [color={rgb, 255:red, 74; green, 144; blue, 226 }  ,draw opacity=1 ][line width=2.25]    (601.7,461.6) -- (579.38,450.6) ;
\draw   (576,465.5) .. controls (576,443.13) and (594.13,425) .. (616.5,425) .. controls (638.87,425) and (657,443.13) .. (657,465.5) .. controls (657,487.87) and (638.87,506) .. (616.5,506) .. controls (594.13,506) and (576,487.87) .. (576,465.5) -- cycle ;
\draw   (635.12,404.64) .. controls (633.84,389.8) and (644.78,375.71) .. (659.55,373.17) .. controls (674.33,370.63) and (687.35,380.61) .. (688.63,395.45) .. controls (689.92,410.29) and (678.98,424.38) .. (664.2,426.92) .. controls (646.36,429.98) and (637.42,431.55) .. (637.38,431.6) .. controls (637.42,431.55) and (636.66,422.55) .. (635.12,404.64) -- cycle ;
\draw [color={rgb, 255:red, 239; green, 36; blue, 36 }  ,draw opacity=1 ][line width=1.5]    (381,626.35) -- (402,624.45) ;
\draw [color={rgb, 255:red, 74; green, 144; blue, 226 }  ,draw opacity=1 ][line width=2.25]    (353.7,629.45) -- (381,626.35) ;
\draw   (300,626.35) .. controls (300,603.98) and (318.13,585.85) .. (340.5,585.85) .. controls (362.87,585.85) and (381,603.98) .. (381,626.35) .. controls (381,648.71) and (362.87,666.85) .. (340.5,666.85) .. controls (318.13,666.85) and (300,648.71) .. (300,626.35) -- cycle ;
\draw   (359.12,565.49) .. controls (357.84,550.65) and (368.78,536.56) .. (383.55,534.02) .. controls (398.33,531.48) and (411.35,541.45) .. (412.63,556.29) .. controls (413.92,571.13) and (402.98,585.22) .. (388.2,587.76) .. controls (370.36,590.83) and (361.42,592.39) .. (361.38,592.45) .. controls (361.42,592.39) and (360.67,583.41) .. (359.12,565.49) -- cycle ;
\draw [color={rgb, 255:red, 239; green, 36; blue, 36 }  ,draw opacity=1 ][line width=1.5]    (593.38,537.6) -- (608.7,521.6) ;
\draw [color={rgb, 255:red, 74; green, 144; blue, 226 }  ,draw opacity=1 ][line width=2.25]    (542.7,612.6) -- (551.38,591.6) ;
\draw   (490,625.5) .. controls (490,603.13) and (508.13,585) .. (530.5,585) .. controls (552.87,585) and (571,603.13) .. (571,625.5) .. controls (571,647.87) and (552.87,666) .. (530.5,666) .. controls (508.13,666) and (490,647.87) .. (490,625.5) -- cycle ;
\draw   (549.12,564.64) .. controls (547.84,549.8) and (558.78,535.71) .. (573.55,533.17) .. controls (588.33,530.63) and (601.35,540.61) .. (602.63,555.45) .. controls (603.92,570.29) and (592.98,584.38) .. (578.2,586.92) .. controls (560.36,589.98) and (551.42,591.55) .. (551.38,591.6) .. controls (551.42,591.55) and (550.66,582.55) .. (549.12,564.64) -- cycle ;

\draw (16,86.4) node [anchor=north west][inner sep=0.75pt]  [font=\Large]  {$d_{cell} E\ =$};
\draw (129,86.4) node [anchor=north west][inner sep=0.75pt]  [font=\Large]  {$d_{cell}$};
\draw (153,259.4) node [anchor=north west][inner sep=0.75pt]  [font=\Large]  {$[$};
\draw (307,267.4) node [anchor=north west][inner sep=0.75pt]  [font=\Large]  {$,$};
\draw (463,258.4) node [anchor=north west][inner sep=0.75pt]  [font=\Large]  {$]$};
\draw (320,92.4) node [anchor=north west][inner sep=0.75pt]  [font=\Large]  {$=$};
\draw (522,92.4) node [anchor=north west][inner sep=0.75pt]  [font=\Large]  {$-$};
\draw (696,92.4) node [anchor=north west][inner sep=0.75pt]  [font=\Large]  {$+$};
\draw (16,259.4) node [anchor=north west][inner sep=0.75pt]  [font=\Large]  {$[ e,B] \ =$};
\draw (90,406.4) node [anchor=north west][inner sep=0.75pt]  [font=\Large]  {$=$};
\draw (310,429.4) node [anchor=north west][inner sep=0.75pt]  [font=\Large]  {$+$};
\draw (520,427.4) node [anchor=north west][inner sep=0.75pt]  [font=\Large]  {$-$};
\draw (6,588.4) node [anchor=north west][inner sep=0.75pt]  [font=\Large]  {$\Rightarrow \ L=d_{cell} E+[ e,B] =$};
\draw (440,591.4) node [anchor=north west][inner sep=0.75pt]  [font=\Large]  {$+$};

\end{tikzpicture}
\end{equation}

\textbf{Claim}. $[e+tE,L]\sim 0$. \par\indent
\emph{Proof of Claim}. This follows from the following graphical calculations.

\tikzset{every picture/.style={line width=0.75pt}} 

\tikzset{every picture/.style={line width=0.75pt}} 
\begin{equation}
\begin{tikzpicture}[x=0.75pt,y=0.75pt,yscale=-0.7,xscale=0.7]

\draw [color={rgb, 255:red, 239; green, 36; blue, 36 }  ,draw opacity=1 ][line width=1.5]    (189.38,69.6) -- (217.7,76.6) ;
\draw [color={rgb, 255:red, 74; green, 144; blue, 226 }  ,draw opacity=1 ][line width=2.25]    (180.7,90.6) -- (189.38,69.6) ;
\draw   (128,103.5) .. controls (128,81.13) and (146.13,63) .. (168.5,63) .. controls (190.87,63) and (209,81.13) .. (209,103.5) .. controls (209,125.87) and (190.87,144) .. (168.5,144) .. controls (146.13,144) and (128,125.87) .. (128,103.5) -- cycle ;
\draw   (187.12,42.64) .. controls (185.84,27.8) and (196.78,13.71) .. (211.55,11.17) .. controls (226.33,8.63) and (239.35,18.61) .. (240.63,33.45) .. controls (241.92,48.29) and (230.98,62.38) .. (216.2,64.92) .. controls (198.36,67.98) and (189.42,69.55) .. (189.38,69.6) .. controls (189.42,69.55) and (188.66,60.55) .. (187.12,42.64) -- cycle ;
\draw [color={rgb, 255:red, 239; green, 36; blue, 36 }  ,draw opacity=1 ][line width=1.5]    (444,106.1) -- (465,104.2) ;
\draw [color={rgb, 255:red, 74; green, 144; blue, 226 }  ,draw opacity=1 ][line width=2.25]    (416.7,109.2) -- (444,106.1) ;
\draw   (363,106.1) .. controls (363,83.73) and (381.13,65.6) .. (403.5,65.6) .. controls (425.87,65.6) and (444,83.73) .. (444,106.1) .. controls (444,128.47) and (425.87,146.6) .. (403.5,146.6) .. controls (381.13,146.6) and (363,128.47) .. (363,106.1) -- cycle ;
\draw   (422.12,45.24) .. controls (420.84,30.4) and (431.78,16.31) .. (446.55,13.77) .. controls (461.33,11.23) and (474.35,21.21) .. (475.63,36.05) .. controls (476.92,50.89) and (465.98,64.98) .. (451.2,67.52) .. controls (433.36,70.58) and (424.42,72.15) .. (424.38,72.2) .. controls (424.42,72.15) and (423.66,63.15) .. (422.12,45.24) -- cycle ;
\draw [color={rgb, 255:red, 239; green, 36; blue, 36 }  ,draw opacity=1 ][line width=1.5]    (656.38,17.36) -- (671.7,1.35) ;
\draw [color={rgb, 255:red, 74; green, 144; blue, 226 }  ,draw opacity=1 ][line width=2.25]    (605.7,92.35) -- (614.38,71.36) ;
\draw   (553,105.25) .. controls (553,82.89) and (571.13,64.75) .. (593.5,64.75) .. controls (615.87,64.75) and (634,82.89) .. (634,105.25) .. controls (634,127.62) and (615.87,145.75) .. (593.5,145.75) .. controls (571.13,145.75) and (553,127.62) .. (553,105.25) -- cycle ;
\draw   (612.12,44.4) .. controls (610.84,29.56) and (621.78,15.47) .. (636.55,12.93) .. controls (651.33,10.39) and (664.35,20.36) .. (665.63,35.2) .. controls (666.92,50.04) and (655.98,64.13) .. (641.2,66.67) .. controls (623.36,69.74) and (614.42,71.29) .. (614.38,71.36) .. controls (614.42,71.29) and (613.66,62.31) .. (612.12,44.4) -- cycle ;
\draw [color={rgb, 255:red, 239; green, 36; blue, 36 }  ,draw opacity=1 ][line width=1.5]    (209.38,257.6) -- (237.7,264.6) ;
\draw [color={rgb, 255:red, 74; green, 144; blue, 226 }  ,draw opacity=1 ][line width=2.25]    (200.7,278.6) -- (209.38,257.6) ;
\draw   (148,291.5) .. controls (148,269.13) and (166.13,251) .. (188.5,251) .. controls (210.87,251) and (229,269.13) .. (229,291.5) .. controls (229,313.87) and (210.87,332) .. (188.5,332) .. controls (166.13,332) and (148,313.87) .. (148,291.5) -- cycle ;
\draw   (207.12,230.64) .. controls (207.12,230.64) and (207.12,230.64) .. (207.12,230.64) .. controls (207.12,230.64) and (207.12,230.64) .. (207.12,230.64) .. controls (205.84,215.8) and (216.78,201.71) .. (231.55,199.17) .. controls (246.33,196.63) and (259.35,206.61) .. (260.63,221.45) .. controls (261.92,236.29) and (250.98,250.38) .. (236.2,252.92) .. controls (218.36,255.98) and (209.42,257.55) .. (209.38,257.6) .. controls (209.42,257.55) and (208.66,248.55) .. (207.12,230.64) -- cycle ;
\draw   (147.62,247.6) .. controls (147.62,247.6) and (147.62,247.6) .. (147.62,247.6) .. controls (147.62,247.6) and (147.62,247.6) .. (147.62,247.6) .. controls (133.03,244.58) and (122.67,230.06) .. (124.47,215.18) .. controls (126.26,200.29) and (139.55,190.67) .. (154.13,193.69) .. controls (168.72,196.71) and (179.09,211.23) .. (177.29,226.11) .. controls (175.11,244.09) and (174.06,253.1) .. (174.1,253.16) .. controls (174.06,253.1) and (165.22,251.25) .. (147.62,247.6) -- cycle ;
\draw [color={rgb, 255:red, 239; green, 36; blue, 36 }  ,draw opacity=1 ][line width=1.5]    (403.38,258.6) -- (431.7,265.6) ;
\draw [color={rgb, 255:red, 74; green, 144; blue, 226 }  ,draw opacity=1 ][line width=2.25]    (394.7,279.6) -- (403.38,258.6) ;
\draw   (342,292.5) .. controls (342,270.13) and (360.13,252) .. (382.5,252) .. controls (404.87,252) and (423,270.13) .. (423,292.5) .. controls (423,314.87) and (404.87,333) .. (382.5,333) .. controls (360.13,333) and (342,314.87) .. (342,292.5) -- cycle ;
\draw   (401.12,231.64) .. controls (401.12,231.64) and (401.12,231.64) .. (401.12,231.64) .. controls (401.12,231.64) and (401.12,231.64) .. (401.12,231.64) .. controls (399.84,216.8) and (410.78,202.71) .. (425.55,200.17) .. controls (440.33,197.63) and (453.35,207.61) .. (454.63,222.45) .. controls (455.92,237.29) and (444.98,251.38) .. (430.2,253.92) .. controls (412.36,256.98) and (403.42,258.55) .. (403.38,258.6) .. controls (403.42,258.55) and (402.66,249.55) .. (401.12,231.64) -- cycle ;
\draw   (374.5,218.74) .. controls (374.5,218.74) and (374.5,218.74) .. (374.5,218.74) .. controls (374.5,218.74) and (374.5,218.74) .. (374.5,218.74) .. controls (359.6,218.38) and (346.81,205.96) .. (345.91,190.99) .. controls (345.02,176.02) and (356.37,164.18) .. (371.26,164.55) .. controls (386.15,164.91) and (398.95,177.33) .. (399.84,192.3) .. controls (400.92,210.37) and (401.49,219.42) .. (401.54,219.48) .. controls (401.49,219.42) and (392.47,219.19) .. (374.5,218.74) -- cycle ;
\draw [color={rgb, 255:red, 74; green, 144; blue, 226 }  ,draw opacity=1 ][line width=2.25]    (592.7,279.6) -- (601.38,258.6) ;
\draw   (540,292.5) .. controls (540,270.13) and (558.13,252) .. (580.5,252) .. controls (602.87,252) and (621,270.13) .. (621,292.5) .. controls (621,314.87) and (602.87,333) .. (580.5,333) .. controls (558.13,333) and (540,314.87) .. (540,292.5) -- cycle ;
\draw   (578.44,244.25) .. controls (578.44,244.25) and (578.44,244.25) .. (578.44,244.25) .. controls (578.44,244.25) and (578.44,244.25) .. (578.44,244.25) .. controls (565.8,236.38) and (561.08,219.17) .. (567.9,205.82) .. controls (574.73,192.47) and (590.51,188.03) .. (603.16,195.9) .. controls (615.81,203.77) and (620.53,220.97) .. (613.7,234.32) .. controls (605.46,250.43) and (601.36,258.53) .. (601.38,258.6) .. controls (601.36,258.53) and (593.71,253.74) .. (578.44,244.25) -- cycle ;
\draw [color={rgb, 255:red, 239; green, 36; blue, 36 }  ,draw opacity=1 ][line width=1.5]    (669.09,241.7) -- (691.07,238.89) ;
\draw   (617.68,234.91) .. controls (617.68,234.91) and (617.68,234.91) .. (617.68,234.91) .. controls (626.01,222.57) and (643.37,218.49) .. (656.46,225.8) .. controls (669.55,233.12) and (673.41,249.06) .. (665.07,261.4) .. controls (656.74,273.75) and (639.38,277.83) .. (626.29,270.52) .. controls (610.49,261.69) and (602.56,257.28) .. (602.48,257.3) .. controls (602.56,257.28) and (607.62,249.82) .. (617.68,234.91) -- cycle ;
\draw [color={rgb, 255:red, 239; green, 36; blue, 36 }  ,draw opacity=1 ][line width=1.5]    (203.38,420.6) -- (231.7,427.6) ;
\draw [color={rgb, 255:red, 74; green, 144; blue, 226 }  ,draw opacity=1 ][line width=2.25]    (194.7,441.6) -- (203.38,420.6) ;
\draw   (142,454.5) .. controls (142,432.13) and (160.13,414) .. (182.5,414) .. controls (204.87,414) and (223,432.13) .. (223,454.5) .. controls (223,476.87) and (204.87,495) .. (182.5,495) .. controls (160.13,495) and (142,476.87) .. (142,454.5) -- cycle ;
\draw   (201.12,393.64) .. controls (199.84,378.8) and (210.78,364.71) .. (225.55,362.17) .. controls (240.33,359.63) and (253.35,369.61) .. (254.63,384.45) .. controls (255.92,399.29) and (244.98,413.38) .. (230.2,415.92) .. controls (212.36,418.98) and (203.42,420.55) .. (203.38,420.6) .. controls (203.42,420.55) and (202.66,411.55) .. (201.12,393.64) -- cycle ;
\draw   (141.62,410.6) .. controls (127.03,407.58) and (116.67,393.06) .. (118.47,378.18) .. controls (120.26,363.29) and (133.55,353.67) .. (148.13,356.69) .. controls (162.72,359.71) and (173.09,374.23) .. (171.29,389.11) .. controls (169.11,407.09) and (168.06,416.1) .. (168.1,416.16) .. controls (168.06,416.1) and (159.22,414.25) .. (141.62,410.6) -- cycle ;
\draw [color={rgb, 255:red, 74; green, 144; blue, 226 }  ,draw opacity=1 ][line width=2.25]    (370.7,450.6) -- (379.38,429.6) ;
\draw   (318,463.5) .. controls (318,441.13) and (336.13,423) .. (358.5,423) .. controls (380.87,423) and (399,441.13) .. (399,463.5) .. controls (399,485.87) and (380.87,504) .. (358.5,504) .. controls (336.13,504) and (318,485.87) .. (318,463.5) -- cycle ;
\draw   (356.44,415.25) .. controls (343.8,407.38) and (339.08,390.17) .. (345.9,376.82) .. controls (352.73,363.47) and (368.51,359.03) .. (381.16,366.9) .. controls (393.81,374.77) and (398.53,391.97) .. (391.7,405.32) .. controls (383.46,421.43) and (379.36,429.53) .. (379.38,429.6) .. controls (379.36,429.53) and (371.71,424.74) .. (356.44,415.25) -- cycle ;
\draw [color={rgb, 255:red, 239; green, 36; blue, 36 }  ,draw opacity=1 ][line width=1.5]    (447.09,412.7) -- (469.07,409.89) ;
\draw   (395.68,405.91) .. controls (395.68,405.91) and (395.68,405.91) .. (395.68,405.91) .. controls (395.68,405.91) and (395.68,405.91) .. (395.68,405.91) .. controls (404.01,393.57) and (421.37,389.49) .. (434.46,396.8) .. controls (447.55,404.12) and (451.41,420.06) .. (443.07,432.4) .. controls (434.74,444.75) and (417.38,448.83) .. (404.29,441.52) .. controls (388.49,432.69) and (380.56,428.28) .. (380.48,428.3) .. controls (380.56,428.28) and (385.62,420.82) .. (395.68,405.91) -- cycle ;

\draw (30,67.4) node [anchor=north west][inner sep=0.75pt]  [font=\Large]  {$e\circ L=$};
\draw (503,71.15) node [anchor=north west][inner sep=0.75pt]  [font=\Large]  {$+$};
\draw (262,68.4) node [anchor=north west][inner sep=0.75pt]  [font=\Large]  {$\circ $};
\draw (321,68.4) node [anchor=north west][inner sep=0.75pt]  [font=\Large]  {$($};
\draw (693,68.4) node [anchor=north west][inner sep=0.75pt]  [font=\Large]  {$)$};
\draw (80,249.4) node [anchor=north west][inner sep=0.75pt]  [font=\Large]  {$=$};
\draw (279,249.15) node [anchor=north west][inner sep=0.75pt]  [font=\Large]  {$+$};
\draw (490,251.15) node [anchor=north west][inner sep=0.75pt]  [font=\Large]  {$+$};
\draw (80,421.4) node [anchor=north west][inner sep=0.75pt]  [font=\Large]  {$\sim $};
\draw (278,421.15) node [anchor=north west][inner sep=0.75pt]  [font=\Large]  {$+$};

\end{tikzpicture}
\end{equation}
\tikzset{every picture/.style={line width=0.75pt}} 
\begin{equation}
\begin{tikzpicture}[x=0.75pt,y=0.75pt,yscale=-0.7,xscale=0.7]

\draw [color={rgb, 255:red, 239; green, 36; blue, 36 }  ,draw opacity=1 ][line width=1.5]    (621.38,72.45) -- (648.7,77.6) ;
\draw [color={rgb, 255:red, 74; green, 144; blue, 226 }  ,draw opacity=1 ][line width=2.25]    (610.7,93.6) -- (621.38,72.45) ;
\draw   (560,106.35) .. controls (560,83.98) and (578.13,65.85) .. (600.5,65.85) .. controls (622.87,65.85) and (641,83.98) .. (641,106.35) .. controls (641,128.71) and (622.87,146.85) .. (600.5,146.85) .. controls (578.13,146.85) and (560,128.71) .. (560,106.35) -- cycle ;
\draw   (619.12,45.49) .. controls (617.84,30.65) and (628.78,16.56) .. (643.55,14.02) .. controls (658.33,11.48) and (671.35,21.45) .. (672.63,36.29) .. controls (673.92,51.13) and (662.98,65.22) .. (648.2,67.76) .. controls (630.36,70.83) and (621.42,72.39) .. (621.38,72.45) .. controls (621.42,72.39) and (620.67,63.41) .. (619.12,45.49) -- cycle ;
\draw [color={rgb, 255:red, 239; green, 36; blue, 36 }  ,draw opacity=1 ][line width=1.5]    (243,110.1) -- (264,108.2) ;
\draw [color={rgb, 255:red, 74; green, 144; blue, 226 }  ,draw opacity=1 ][line width=2.25]    (215.7,113.2) -- (243,110.1) ;
\draw   (162,110.1) .. controls (162,87.73) and (180.13,69.6) .. (202.5,69.6) .. controls (224.87,69.6) and (243,87.73) .. (243,110.1) .. controls (243,132.47) and (224.87,150.6) .. (202.5,150.6) .. controls (180.13,150.6) and (162,132.47) .. (162,110.1) -- cycle ;
\draw   (221.12,49.24) .. controls (219.84,34.4) and (230.78,20.31) .. (245.55,17.77) .. controls (260.33,15.23) and (273.35,25.21) .. (274.63,40.05) .. controls (275.92,54.89) and (264.98,68.98) .. (250.2,71.52) .. controls (232.36,74.58) and (223.42,76.15) .. (223.38,76.2) .. controls (223.42,76.15) and (222.66,67.15) .. (221.12,49.24) -- cycle ;
\draw [color={rgb, 255:red, 74; green, 144; blue, 226 }  ,draw opacity=1 ][line width=2.25]    (404.7,96.35) -- (413.38,75.36) ;
\draw   (352,109.25) .. controls (352,86.89) and (370.13,68.75) .. (392.5,68.75) .. controls (414.87,68.75) and (433,86.89) .. (433,109.25) .. controls (433,131.62) and (414.87,149.75) .. (392.5,149.75) .. controls (370.13,149.75) and (352,131.62) .. (352,109.25) -- cycle ;
\draw   (411.12,48.4) .. controls (409.84,33.56) and (420.78,19.47) .. (435.55,16.93) .. controls (450.33,14.39) and (463.35,24.36) .. (464.63,39.2) .. controls (465.92,54.04) and (454.98,68.13) .. (440.2,70.67) .. controls (422.36,73.74) and (413.42,75.29) .. (413.38,75.36) .. controls (413.42,75.29) and (412.66,66.31) .. (411.12,48.4) -- cycle ;
\draw [color={rgb, 255:red, 239; green, 36; blue, 36 }  ,draw opacity=1 ][line width=1.5]    (457.38,22.36) -- (472.7,6.35) ;
\draw [color={rgb, 255:red, 74; green, 144; blue, 226 }  ,draw opacity=1 ][line width=2.25]    (191.7,285.2) -- (219,282.1) ;
\draw   (138,282.1) .. controls (138,259.73) and (156.13,241.6) .. (178.5,241.6) .. controls (200.87,241.6) and (219,259.73) .. (219,282.1) .. controls (219,304.47) and (200.87,322.6) .. (178.5,322.6) .. controls (156.13,322.6) and (138,304.47) .. (138,282.1) -- cycle ;
\draw   (197.12,221.24) .. controls (197.12,221.24) and (197.12,221.24) .. (197.12,221.24) .. controls (197.12,221.24) and (197.12,221.24) .. (197.12,221.24) .. controls (195.84,206.4) and (206.78,192.31) .. (221.55,189.77) .. controls (236.33,187.23) and (249.35,197.21) .. (250.63,212.05) .. controls (251.92,226.89) and (240.98,240.98) .. (226.2,243.52) .. controls (208.36,246.58) and (199.42,248.15) .. (199.38,248.2) .. controls (199.42,248.15) and (198.66,239.15) .. (197.12,221.24) -- cycle ;
\draw [color={rgb, 255:red, 239; green, 36; blue, 36 }  ,draw opacity=1 ][line width=1.5]    (218.24,281.23) -- (230.85,306.01) ;
\draw   (238.21,262.98) .. controls (249.18,252.9) and (267.02,252.96) .. (278.05,263.11) .. controls (289.09,273.26) and (289.14,289.66) .. (278.17,299.74) .. controls (267.2,309.82) and (249.37,309.76) .. (238.33,299.61) .. controls (225.01,287.36) and (218.31,281.22) .. (218.24,281.23) .. controls (218.31,281.22) and (224.97,275.15) .. (238.21,262.98) -- cycle ;
\draw [color={rgb, 255:red, 74; green, 144; blue, 226 }  ,draw opacity=1 ][line width=2.25]    (425.7,267.35) -- (434.38,246.36) ;
\draw   (373,280.25) .. controls (373,257.89) and (391.13,239.75) .. (413.5,239.75) .. controls (435.87,239.75) and (454,257.89) .. (454,280.25) .. controls (454,302.62) and (435.87,320.75) .. (413.5,320.75) .. controls (391.13,320.75) and (373,302.62) .. (373,280.25) -- cycle ;
\draw   (432.12,219.4) .. controls (432.12,219.4) and (432.12,219.4) .. (432.12,219.4) .. controls (432.12,219.4) and (432.12,219.4) .. (432.12,219.4) .. controls (430.84,204.56) and (441.78,190.47) .. (456.55,187.93) .. controls (471.33,185.39) and (484.35,195.36) .. (485.63,210.2) .. controls (486.92,225.04) and (475.98,239.13) .. (461.2,241.67) .. controls (443.36,244.74) and (434.42,246.29) .. (434.38,246.36) .. controls (434.42,246.29) and (433.66,237.31) .. (432.12,219.4) -- cycle ;
\draw [color={rgb, 255:red, 239; green, 36; blue, 36 }  ,draw opacity=1 ][line width=1.5]    (480.82,196.18) -- (501.16,215.15) ;
\draw   (493.33,172.2) .. controls (500.19,158.97) and (516.97,152.92) .. (530.81,158.67) .. controls (544.65,164.43) and (550.32,179.82) .. (543.47,193.05) .. controls (536.62,206.27) and (519.84,212.33) .. (505.99,206.57) .. controls (489.29,199.62) and (480.89,196.16) .. (480.82,196.18) .. controls (480.89,196.16) and (485.07,188.16) .. (493.33,172.2) -- cycle ;
\draw [color={rgb, 255:red, 74; green, 144; blue, 226 }  ,draw opacity=1 ][line width=2.25]    (188.7,444.2) -- (216,441.1) ;
\draw   (135,441.1) .. controls (135,418.73) and (153.13,400.6) .. (175.5,400.6) .. controls (197.87,400.6) and (216,418.73) .. (216,441.1) .. controls (216,463.47) and (197.87,481.6) .. (175.5,481.6) .. controls (153.13,481.6) and (135,463.47) .. (135,441.1) -- cycle ;
\draw   (194.12,380.24) .. controls (192.84,365.4) and (203.78,351.31) .. (218.55,348.77) .. controls (233.33,346.23) and (246.35,356.21) .. (247.63,371.05) .. controls (248.92,385.89) and (237.98,399.98) .. (223.2,402.52) .. controls (205.36,405.58) and (196.42,407.15) .. (196.38,407.2) .. controls (196.42,407.15) and (195.66,398.15) .. (194.12,380.24) -- cycle ;
\draw [color={rgb, 255:red, 239; green, 36; blue, 36 }  ,draw opacity=1 ][line width=1.5]    (215.24,440.23) -- (227.85,465.01) ;
\draw   (235.21,421.98) .. controls (246.18,411.9) and (264.02,411.96) .. (275.05,422.11) .. controls (286.09,432.26) and (286.14,448.66) .. (275.17,458.74) .. controls (264.2,468.82) and (246.37,468.76) .. (235.33,458.61) .. controls (222.01,446.36) and (215.31,440.22) .. (215.24,440.23) .. controls (215.31,440.22) and (221.97,434.15) .. (235.21,421.98) -- cycle ;
\draw [color={rgb, 255:red, 74; green, 144; blue, 226 }  ,draw opacity=1 ][line width=2.25]    (415.7,426.35) -- (424.38,405.36) ;
\draw   (363,439.25) .. controls (363,416.89) and (381.13,398.75) .. (403.5,398.75) .. controls (425.87,398.75) and (444,416.89) .. (444,439.25) .. controls (444,461.62) and (425.87,479.75) .. (403.5,479.75) .. controls (381.13,479.75) and (363,461.62) .. (363,439.25) -- cycle ;
\draw   (399.69,394.29) .. controls (386.08,388.24) and (379.03,371.85) .. (383.96,357.69) .. controls (388.88,343.53) and (403.91,336.95) .. (417.52,343.01) .. controls (431.13,349.06) and (438.17,365.45) .. (433.25,379.61) .. controls (427.3,396.71) and (424.34,405.29) .. (424.38,405.36) .. controls (424.34,405.29) and (416.12,401.6) .. (399.69,394.29) -- cycle ;
\draw [color={rgb, 255:red, 239; green, 36; blue, 36 }  ,draw opacity=1 ][line width=1.5]    (485.88,376.26) -- (514.7,367.6) ;
\draw   (433.43,379.59) .. controls (438.09,365.44) and (453.7,356.8) .. (468.28,360.29) .. controls (482.86,363.77) and (490.9,378.06) .. (486.24,392.21) .. controls (481.57,406.36) and (465.97,415) .. (451.39,411.52) .. controls (433.78,407.31) and (424.95,405.22) .. (424.88,405.26) .. controls (424.95,405.22) and (427.8,396.67) .. (433.43,379.59) -- cycle ;

\draw (12,70.4) node [anchor=north west][inner sep=0.75pt]  [font=\Large]  {$L\circ e=$};
\draw (302,75.15) node [anchor=north west][inner sep=0.75pt]  [font=\Large]  {$+$};
\draw (522,76.4) node [anchor=north west][inner sep=0.75pt]  [font=\Large]  {$\circ $};
\draw (120,72.4) node [anchor=north west][inner sep=0.75pt]  [font=\Large]  {$($};
\draw (492,72.4) node [anchor=north west][inner sep=0.75pt]  [font=\Large]  {$)$};
\draw (78,225.4) node [anchor=north west][inner sep=0.75pt]  [font=\Large]  {$=$};
\draw (77,394.4) node [anchor=north west][inner sep=0.75pt]  [font=\Large]  {$\sim $};
\draw (324,229.15) node [anchor=north west][inner sep=0.75pt]  [font=\Large]  {$+$};
\draw (321,388.15) node [anchor=north west][inner sep=0.75pt]  [font=\Large]  {$+$};

\end{tikzpicture}
\end{equation}

\tikzset{every picture/.style={line width=0.75pt}} 

\begin{equation}
\begin{tikzpicture}[x=0.75pt,y=0.75pt,yscale=-0.6,xscale=0.6]

\draw [color={rgb, 255:red, 239; green, 36; blue, 36 }  ,draw opacity=1 ][line width=1.5]    (205,109.35) -- (226,107.45) ;
\draw [color={rgb, 255:red, 74; green, 144; blue, 226 }  ,draw opacity=1 ][line width=2.25]    (165,92.45) -- (164.5,68.85) ;
\draw   (124,109.35) .. controls (124,86.98) and (142.13,68.85) .. (164.5,68.85) .. controls (186.87,68.85) and (205,86.98) .. (205,109.35) .. controls (205,131.71) and (186.87,149.85) .. (164.5,149.85) .. controls (142.13,149.85) and (124,131.71) .. (124,109.35) -- cycle ;
\draw   (183.12,48.49) .. controls (181.84,33.65) and (192.78,19.56) .. (207.55,17.02) .. controls (222.33,14.48) and (235.35,24.45) .. (236.63,39.29) .. controls (237.92,54.13) and (226.98,68.22) .. (212.2,70.76) .. controls (194.36,73.83) and (185.42,75.39) .. (185.38,75.45) .. controls (185.42,75.39) and (184.67,66.41) .. (183.12,48.49) -- cycle ;
\draw [color={rgb, 255:red, 239; green, 36; blue, 36 }  ,draw opacity=1 ][line width=1.5]    (439,108.1) -- (460,106.2) ;
\draw [color={rgb, 255:red, 74; green, 144; blue, 226 }  ,draw opacity=1 ][line width=2.25]    (411.7,111.2) -- (439,108.1) ;
\draw   (358,108.1) .. controls (358,85.73) and (376.13,67.6) .. (398.5,67.6) .. controls (420.87,67.6) and (439,85.73) .. (439,108.1) .. controls (439,130.47) and (420.87,148.6) .. (398.5,148.6) .. controls (376.13,148.6) and (358,130.47) .. (358,108.1) -- cycle ;
\draw   (417.12,47.24) .. controls (415.84,32.4) and (426.78,18.31) .. (441.55,15.77) .. controls (456.33,13.23) and (469.35,23.21) .. (470.63,38.05) .. controls (471.92,52.89) and (460.98,66.98) .. (446.2,69.52) .. controls (428.36,72.58) and (419.42,74.15) .. (419.38,74.2) .. controls (419.42,74.15) and (418.66,65.15) .. (417.12,47.24) -- cycle ;
\draw [color={rgb, 255:red, 74; green, 144; blue, 226 }  ,draw opacity=1 ][line width=2.25]    (600.7,94.35) -- (609.38,73.36) ;
\draw   (548,107.25) .. controls (548,84.89) and (566.13,66.75) .. (588.5,66.75) .. controls (610.87,66.75) and (629,84.89) .. (629,107.25) .. controls (629,129.62) and (610.87,147.75) .. (588.5,147.75) .. controls (566.13,147.75) and (548,129.62) .. (548,107.25) -- cycle ;
\draw   (607.12,46.4) .. controls (605.84,31.56) and (616.78,17.47) .. (631.55,14.93) .. controls (646.33,12.39) and (659.35,22.36) .. (660.63,37.2) .. controls (661.92,52.04) and (650.98,66.13) .. (636.2,68.67) .. controls (618.36,71.74) and (609.42,73.29) .. (609.38,73.36) .. controls (609.42,73.29) and (608.66,64.31) .. (607.12,46.4) -- cycle ;
\draw [color={rgb, 255:red, 239; green, 36; blue, 36 }  ,draw opacity=1 ][line width=1.5]    (653.38,20.36) -- (668.7,4.35) ;
\draw [color={rgb, 255:red, 239; green, 36; blue, 36 }  ,draw opacity=1 ][line width=1.5]    (226,266.35) -- (247,264.45) ;
\draw [color={rgb, 255:red, 74; green, 144; blue, 226 }  ,draw opacity=1 ][line width=2.25]    (186,249.45) -- (185.5,225.85) ;
\draw   (145,266.35) .. controls (145,243.98) and (163.13,225.85) .. (185.5,225.85) .. controls (207.87,225.85) and (226,243.98) .. (226,266.35) .. controls (226,288.71) and (207.87,306.85) .. (185.5,306.85) .. controls (163.13,306.85) and (145,288.71) .. (145,266.35) -- cycle ;
\draw   (204.12,205.49) .. controls (204.12,205.49) and (204.12,205.49) .. (204.12,205.49) .. controls (204.12,205.49) and (204.12,205.49) .. (204.12,205.49) .. controls (202.84,190.65) and (213.78,176.56) .. (228.55,174.02) .. controls (243.33,171.48) and (256.35,181.45) .. (257.63,196.29) .. controls (258.92,211.13) and (247.98,225.22) .. (233.2,227.76) .. controls (215.36,230.83) and (206.42,232.39) .. (206.38,232.45) .. controls (206.42,232.39) and (205.67,223.41) .. (204.12,205.49) -- cycle ;
\draw   (136.92,225.34) .. controls (136.92,225.34) and (136.92,225.34) .. (136.92,225.34) .. controls (136.92,225.34) and (136.92,225.34) .. (136.92,225.34) .. controls (122.55,221.4) and (113.13,206.26) .. (115.87,191.51) .. controls (118.6,176.77) and (132.47,168.01) .. (146.83,171.95) .. controls (161.2,175.89) and (170.63,191.03) .. (167.89,205.78) .. controls (164.58,223.57) and (162.96,232.5) .. (162.99,232.56) .. controls (162.96,232.5) and (154.26,230.09) .. (136.92,225.34) -- cycle ;
\draw [color={rgb, 255:red, 239; green, 36; blue, 36 }  ,draw opacity=1 ][line width=1.5]    (426,268.35) -- (447,266.45) ;
\draw [color={rgb, 255:red, 74; green, 144; blue, 226 }  ,draw opacity=1 ][line width=2.25]    (386,251.45) -- (385.5,227.85) ;
\draw   (345,268.35) .. controls (345,245.98) and (363.13,227.85) .. (385.5,227.85) .. controls (407.87,227.85) and (426,245.98) .. (426,268.35) .. controls (426,290.71) and (407.87,308.85) .. (385.5,308.85) .. controls (363.13,308.85) and (345,290.71) .. (345,268.35) -- cycle ;
\draw   (433.59,229.12) .. controls (433.59,229.12) and (433.59,229.12) .. (433.59,229.12) .. controls (433.59,229.12) and (433.59,229.12) .. (433.59,229.12) .. controls (438.5,215.06) and (454.25,206.69) .. (468.77,210.43) .. controls (483.29,214.17) and (491.08,228.6) .. (486.17,242.66) .. controls (481.26,256.72) and (465.5,265.09) .. (450.98,261.35) .. controls (433.45,256.84) and (424.66,254.6) .. (424.6,254.64) .. controls (424.66,254.6) and (427.65,246.1) .. (433.59,229.12) -- cycle ;
\draw   (389.35,209.28) .. controls (389.35,209.28) and (389.35,209.28) .. (389.35,209.28) .. controls (389.35,209.28) and (389.35,209.28) .. (389.35,209.28) .. controls (381.6,196.56) and (385.12,179.08) .. (397.22,170.22) .. controls (409.32,161.37) and (425.42,164.5) .. (433.17,177.22) .. controls (440.93,189.94) and (437.41,207.42) .. (425.31,216.28) .. controls (410.71,226.96) and (403.39,232.34) .. (403.38,232.42) .. controls (403.39,232.34) and (398.72,224.63) .. (389.35,209.28) -- cycle ;
\draw [color={rgb, 255:red, 239; green, 36; blue, 36 }  ,draw opacity=1 ][line width=1.5]    (629,266.35) -- (650,264.45) ;
\draw [color={rgb, 255:red, 74; green, 144; blue, 226 }  ,draw opacity=1 ][line width=2.25]    (589,249.45) -- (588.5,225.85) ;
\draw   (548,266.35) .. controls (548,243.98) and (566.13,225.85) .. (588.5,225.85) .. controls (610.87,225.85) and (629,243.98) .. (629,266.35) .. controls (629,288.71) and (610.87,306.85) .. (588.5,306.85) .. controls (566.13,306.85) and (548,288.71) .. (548,266.35) -- cycle ;
\draw   (607.12,205.49) .. controls (607.12,205.49) and (607.12,205.49) .. (607.12,205.49) .. controls (607.12,205.49) and (607.12,205.49) .. (607.12,205.49) .. controls (605.84,190.65) and (616.78,176.56) .. (631.55,174.02) .. controls (646.33,171.48) and (659.35,181.45) .. (660.63,196.29) .. controls (661.92,211.13) and (650.98,225.22) .. (636.2,227.76) .. controls (618.36,230.83) and (609.42,232.39) .. (609.38,232.45) .. controls (609.42,232.39) and (608.67,223.41) .. (607.12,205.49) -- cycle ;
\draw   (672.52,171.99) .. controls (672.52,171.99) and (672.52,171.99) .. (672.52,171.99) .. controls (672.52,171.99) and (672.52,171.99) .. (672.52,171.99) .. controls (679.03,158.59) and (695.65,152.11) .. (709.64,157.51) .. controls (723.62,162.9) and (729.69,178.14) .. (723.18,191.54) .. controls (716.67,204.94) and (700.05,211.42) .. (686.07,206.03) .. controls (669.18,199.51) and (660.7,196.27) .. (660.63,196.29) .. controls (660.7,196.27) and (664.67,188.16) .. (672.52,171.99) -- cycle ;
\draw [color={rgb, 255:red, 239; green, 36; blue, 36 }  ,draw opacity=1 ][line width=1.5]    (257,422.35) -- (278,420.45) ;
\draw [color={rgb, 255:red, 74; green, 144; blue, 226 }  ,draw opacity=1 ][line width=2.25]    (217,405.45) -- (216.5,381.85) ;
\draw   (176,422.35) .. controls (176,399.98) and (194.13,381.85) .. (216.5,381.85) .. controls (238.87,381.85) and (257,399.98) .. (257,422.35) .. controls (257,444.71) and (238.87,462.85) .. (216.5,462.85) .. controls (194.13,462.85) and (176,444.71) .. (176,422.35) -- cycle ;
\draw   (264.59,383.12) .. controls (269.5,369.06) and (285.25,360.69) .. (299.77,364.43) .. controls (314.29,368.17) and (322.08,382.6) .. (317.17,396.66) .. controls (312.26,410.72) and (296.5,419.09) .. (281.98,415.35) .. controls (264.45,410.84) and (255.66,408.6) .. (255.6,408.64) .. controls (255.66,408.6) and (258.65,400.1) .. (264.59,383.12) -- cycle ;
\draw   (220.35,363.28) .. controls (212.6,350.56) and (216.12,333.08) .. (228.22,324.22) .. controls (240.32,315.37) and (256.42,318.5) .. (264.17,331.22) .. controls (271.93,343.94) and (268.41,361.42) .. (256.31,370.28) .. controls (241.71,380.96) and (234.39,386.34) .. (234.38,386.42) .. controls (234.39,386.34) and (229.72,378.63) .. (220.35,363.28) -- cycle ;
\draw [color={rgb, 255:red, 74; green, 144; blue, 226 }  ,draw opacity=1 ][line width=2.25]    (418,406.45) -- (417.5,382.85) ;
\draw   (377,423.35) .. controls (377,400.98) and (395.13,382.85) .. (417.5,382.85) .. controls (439.87,382.85) and (458,400.98) .. (458,423.35) .. controls (458,445.71) and (439.87,463.85) .. (417.5,463.85) .. controls (395.13,463.85) and (377,445.71) .. (377,423.35) -- cycle ;
\draw   (436.12,362.49) .. controls (434.84,347.65) and (445.78,333.56) .. (460.55,331.02) .. controls (475.33,328.48) and (488.35,338.45) .. (489.63,353.29) .. controls (490.92,368.13) and (479.98,382.22) .. (465.2,384.76) .. controls (447.36,387.83) and (438.42,389.39) .. (438.38,389.45) .. controls (438.42,389.39) and (437.67,380.41) .. (436.12,362.49) -- cycle ;
\draw   (478.32,395.9) .. controls (490.2,386.91) and (507.95,388.66) .. (517.97,399.81) .. controls (527.99,410.97) and (526.49,427.3) .. (514.61,436.29) .. controls (502.74,445.28) and (484.99,443.53) .. (474.97,432.37) .. controls (462.87,418.91) and (456.78,412.18) .. (456.71,412.17) .. controls (456.78,412.18) and (463.99,406.75) .. (478.32,395.9) -- cycle ;
\draw [color={rgb, 255:red, 239; green, 36; blue, 36 }  ,draw opacity=1 ][line width=1.5]    (525,421.17) -- (546.71,425.62) ;
\draw [color={rgb, 255:red, 239; green, 36; blue, 36 }  ,draw opacity=1 ][line width=1.5]    (239,586.35) -- (260,584.45) ;
\draw [color={rgb, 255:red, 74; green, 144; blue, 226 }  ,draw opacity=1 ][line width=2.25]    (199,569.45) -- (198.5,545.85) ;
\draw   (158,586.35) .. controls (158,563.98) and (176.13,545.85) .. (198.5,545.85) .. controls (220.87,545.85) and (239,563.98) .. (239,586.35) .. controls (239,608.71) and (220.87,626.85) .. (198.5,626.85) .. controls (176.13,626.85) and (158,608.71) .. (158,586.35) -- cycle ;
\draw   (217.12,525.49) .. controls (217.12,525.49) and (217.12,525.49) .. (217.12,525.49) .. controls (217.12,525.49) and (217.12,525.49) .. (217.12,525.49) .. controls (215.84,510.65) and (226.78,496.56) .. (241.55,494.02) .. controls (256.33,491.48) and (269.35,501.45) .. (270.63,516.29) .. controls (271.92,531.13) and (260.98,545.22) .. (246.2,547.76) .. controls (228.36,550.83) and (219.42,552.39) .. (219.38,552.45) .. controls (219.42,552.39) and (218.67,543.41) .. (217.12,525.49) -- cycle ;
\draw   (149.92,545.34) .. controls (135.55,541.4) and (126.13,526.26) .. (128.87,511.51) .. controls (131.6,496.77) and (145.47,488.01) .. (159.83,491.95) .. controls (174.2,495.89) and (183.63,511.03) .. (180.89,525.78) .. controls (177.58,543.57) and (175.96,552.5) .. (175.99,552.56) .. controls (175.96,552.5) and (167.26,550.09) .. (149.92,545.34) -- cycle ;
\draw [color={rgb, 255:red, 74; green, 144; blue, 226 }  ,draw opacity=1 ][line width=2.25]    (369,564.45) -- (368.5,540.85) ;
\draw   (328,581.35) .. controls (328,558.98) and (346.13,540.85) .. (368.5,540.85) .. controls (390.87,540.85) and (409,558.98) .. (409,581.35) .. controls (409,603.71) and (390.87,621.85) .. (368.5,621.85) .. controls (346.13,621.85) and (328,603.71) .. (328,581.35) -- cycle ;
\draw   (388.12,519.49) .. controls (388.12,519.49) and (388.12,519.49) .. (388.12,519.49) .. controls (388.12,519.49) and (388.12,519.49) .. (388.12,519.49) .. controls (386.84,504.65) and (397.78,490.56) .. (412.55,488.02) .. controls (427.33,485.48) and (440.35,495.45) .. (441.63,510.29) .. controls (442.92,525.13) and (431.98,539.22) .. (417.2,541.76) .. controls (399.36,544.83) and (390.42,546.39) .. (390.38,546.45) .. controls (390.42,546.39) and (389.67,537.41) .. (388.12,519.49) -- cycle ;
\draw   (429.32,553.9) .. controls (429.32,553.9) and (429.32,553.9) .. (429.32,553.9) .. controls (441.2,544.91) and (458.95,546.66) .. (468.97,557.81) .. controls (478.99,568.97) and (477.49,585.3) .. (465.61,594.29) .. controls (453.74,603.28) and (435.99,601.53) .. (425.97,590.37) .. controls (413.87,576.91) and (407.78,570.18) .. (407.71,570.17) .. controls (407.78,570.18) and (414.99,564.75) .. (429.32,553.9) -- cycle ;
\draw [color={rgb, 255:red, 239; green, 36; blue, 36 }  ,draw opacity=1 ][line width=1.5]    (476,579.17) -- (497.71,583.62) ;

\draw (0,70.4) node [anchor=north west][inner sep=0.75pt]  [font=\Large]  {$E\circ L=$};
\draw (498,73.15) node [anchor=north west][inner sep=0.75pt]  [font=\Large]  {$+$};
\draw (258,73.4) node [anchor=north west][inner sep=0.75pt]  [font=\Large]  {$\circ $};
\draw (316,70.4) node [anchor=north west][inner sep=0.75pt]  [font=\Large]  {$($};
\draw (688,70.4) node [anchor=north west][inner sep=0.75pt]  [font=\Large]  {$)$};
\draw (78,240.4) node [anchor=north west][inner sep=0.75pt]  [font=\Large]  {$=$};
\draw (82,545.4) node [anchor=north west][inner sep=0.75pt]  [font=\Large]  {$\sim $};
\draw (294,242.15) node [anchor=north west][inner sep=0.75pt]  [font=\Large]  {$-$};
\draw (512,243.15) node [anchor=north west][inner sep=0.75pt]  [font=\Large]  {$+$};
\draw (135,398.15) node [anchor=north west][inner sep=0.75pt]  [font=\Large]  {$+$};
\draw (338,398.15) node [anchor=north west][inner sep=0.75pt]  [font=\Large]  {$+$};
\draw (289,556.15) node [anchor=north west][inner sep=0.75pt]  [font=\Large]  {$+$};

\end{tikzpicture}    
\end{equation}

\tikzset{every picture/.style={line width=0.75pt}} 
\begin{equation}
\begin{tikzpicture}[x=0.75pt,y=0.75pt,yscale=-0.6,xscale=0.6]

\draw [color={rgb, 255:red, 239; green, 36; blue, 36 }  ,draw opacity=1 ][line width=1.5]    (641,106.35) -- (665.7,104.6) ;
\draw [color={rgb, 255:red, 74; green, 144; blue, 226 }  ,draw opacity=1 ][line width=2.25]    (600.7,89.6) -- (600.5,65.85) ;
\draw   (560,106.35) .. controls (560,83.98) and (578.13,65.85) .. (600.5,65.85) .. controls (622.87,65.85) and (641,83.98) .. (641,106.35) .. controls (641,128.71) and (622.87,146.85) .. (600.5,146.85) .. controls (578.13,146.85) and (560,128.71) .. (560,106.35) -- cycle ;
\draw   (619.12,45.49) .. controls (617.84,30.65) and (628.78,16.56) .. (643.55,14.02) .. controls (658.33,11.48) and (671.35,21.45) .. (672.63,36.29) .. controls (673.92,51.13) and (662.98,65.22) .. (648.2,67.76) .. controls (630.36,70.83) and (621.42,72.39) .. (621.38,72.45) .. controls (621.42,72.39) and (620.67,63.41) .. (619.12,45.49) -- cycle ;
\draw [color={rgb, 255:red, 239; green, 36; blue, 36 }  ,draw opacity=1 ][line width=1.5]    (243,110.1) -- (264,108.2) ;
\draw [color={rgb, 255:red, 74; green, 144; blue, 226 }  ,draw opacity=1 ][line width=2.25]    (215.7,113.2) -- (243,110.1) ;
\draw   (162,110.1) .. controls (162,87.73) and (180.13,69.6) .. (202.5,69.6) .. controls (224.87,69.6) and (243,87.73) .. (243,110.1) .. controls (243,132.47) and (224.87,150.6) .. (202.5,150.6) .. controls (180.13,150.6) and (162,132.47) .. (162,110.1) -- cycle ;
\draw   (221.12,49.24) .. controls (219.84,34.4) and (230.78,20.31) .. (245.55,17.77) .. controls (260.33,15.23) and (273.35,25.21) .. (274.63,40.05) .. controls (275.92,54.89) and (264.98,68.98) .. (250.2,71.52) .. controls (232.36,74.58) and (223.42,76.15) .. (223.38,76.2) .. controls (223.42,76.15) and (222.66,67.15) .. (221.12,49.24) -- cycle ;
\draw [color={rgb, 255:red, 74; green, 144; blue, 226 }  ,draw opacity=1 ][line width=2.25]    (404.7,96.35) -- (413.38,75.36) ;
\draw   (352,109.25) .. controls (352,86.89) and (370.13,68.75) .. (392.5,68.75) .. controls (414.87,68.75) and (433,86.89) .. (433,109.25) .. controls (433,131.62) and (414.87,149.75) .. (392.5,149.75) .. controls (370.13,149.75) and (352,131.62) .. (352,109.25) -- cycle ;
\draw   (411.12,48.4) .. controls (409.84,33.56) and (420.78,19.47) .. (435.55,16.93) .. controls (450.33,14.39) and (463.35,24.36) .. (464.63,39.2) .. controls (465.92,54.04) and (454.98,68.13) .. (440.2,70.67) .. controls (422.36,73.74) and (413.42,75.29) .. (413.38,75.36) .. controls (413.42,75.29) and (412.66,66.31) .. (411.12,48.4) -- cycle ;
\draw [color={rgb, 255:red, 239; green, 36; blue, 36 }  ,draw opacity=1 ][line width=1.5]    (457.38,22.36) -- (472.7,6.35) ;
\draw [color={rgb, 255:red, 239; green, 36; blue, 36 }  ,draw opacity=1 ][line width=1.5]    (194.7,214.6) -- (202.5,234.6) ;
\draw [color={rgb, 255:red, 74; green, 144; blue, 226 }  ,draw opacity=1 ][line width=2.25]    (232.7,275.2) -- (260,272.1) ;
\draw   (179,272.1) .. controls (179,249.73) and (197.13,231.6) .. (219.5,231.6) .. controls (241.87,231.6) and (260,249.73) .. (260,272.1) .. controls (260,294.47) and (241.87,312.6) .. (219.5,312.6) .. controls (197.13,312.6) and (179,294.47) .. (179,272.1) -- cycle ;
\draw   (238.12,211.24) .. controls (238.12,211.24) and (238.12,211.24) .. (238.12,211.24) .. controls (238.12,211.24) and (238.12,211.24) .. (238.12,211.24) .. controls (236.84,196.4) and (247.78,182.31) .. (262.55,179.77) .. controls (277.33,177.23) and (290.35,187.21) .. (291.63,202.05) .. controls (292.92,216.89) and (281.98,230.98) .. (267.2,233.52) .. controls (249.36,236.58) and (240.42,238.15) .. (240.38,238.2) .. controls (240.42,238.15) and (239.66,229.15) .. (238.12,211.24) -- cycle ;
\draw   (176.6,316.42) .. controls (172.76,330.82) and (157.68,340.34) .. (142.92,337.7) .. controls (128.16,335.06) and (119.31,321.26) .. (123.15,306.87) .. controls (127,292.47) and (142.08,282.95) .. (156.84,285.59) .. controls (174.65,288.78) and (183.59,290.34) .. (183.65,290.31) .. controls (183.59,290.34) and (181.24,299.06) .. (176.6,316.42) -- cycle ;
\draw [color={rgb, 255:red, 239; green, 36; blue, 36 }  ,draw opacity=1 ][line width=1.5]    (518.63,191.05) -- (537.33,182.55) ;
\draw [color={rgb, 255:red, 74; green, 144; blue, 226 }  ,draw opacity=1 ][line width=2.25]    (462.7,274.2) -- (490,271.1) ;
\draw   (409,271.1) .. controls (409,248.73) and (427.13,230.6) .. (449.5,230.6) .. controls (471.87,230.6) and (490,248.73) .. (490,271.1) .. controls (490,293.47) and (471.87,311.6) .. (449.5,311.6) .. controls (427.13,311.6) and (409,293.47) .. (409,271.1) -- cycle ;
\draw   (468.12,210.24) .. controls (468.12,210.24) and (468.12,210.24) .. (468.12,210.24) .. controls (468.12,210.24) and (468.12,210.24) .. (468.12,210.24) .. controls (466.84,195.4) and (477.78,181.31) .. (492.55,178.77) .. controls (507.33,176.23) and (520.35,186.21) .. (521.63,201.05) .. controls (522.92,215.89) and (511.98,229.98) .. (497.2,232.52) .. controls (479.36,235.58) and (470.42,237.15) .. (470.38,237.2) .. controls (470.42,237.15) and (469.66,228.15) .. (468.12,210.24) -- cycle ;
\draw   (392.22,240.91) .. controls (392.22,240.91) and (392.22,240.91) .. (392.22,240.91) .. controls (392.22,240.91) and (392.22,240.91) .. (392.22,240.91) .. controls (377.57,238.21) and (366.88,223.93) .. (368.35,209.01) .. controls (369.82,194.09) and (382.88,184.18) .. (397.53,186.87) .. controls (412.18,189.57) and (422.87,203.85) .. (421.4,218.77) .. controls (419.64,236.79) and (418.77,245.82) .. (418.82,245.88) .. controls (418.77,245.82) and (409.91,244.17) .. (392.22,240.91) -- cycle ;
\draw [color={rgb, 255:red, 74; green, 144; blue, 226 }  ,draw opacity=1 ][line width=2.25]    (694.7,268.2) -- (722,265.1) ;
\draw   (641,265.1) .. controls (641,242.73) and (659.13,224.6) .. (681.5,224.6) .. controls (703.87,224.6) and (722,242.73) .. (722,265.1) .. controls (722,287.47) and (703.87,305.6) .. (681.5,305.6) .. controls (659.13,305.6) and (641,287.47) .. (641,265.1) -- cycle ;
\draw   (700.12,204.24) .. controls (700.12,204.24) and (700.12,204.24) .. (700.12,204.24) .. controls (700.12,204.24) and (700.12,204.24) .. (700.12,204.24) .. controls (698.84,189.4) and (709.78,175.31) .. (724.55,172.77) .. controls (739.33,170.23) and (752.35,180.21) .. (753.63,195.05) .. controls (754.92,209.89) and (743.98,223.98) .. (729.2,226.52) .. controls (711.36,229.58) and (702.42,231.15) .. (702.38,231.2) .. controls (702.42,231.15) and (701.66,222.15) .. (700.12,204.24) -- cycle ;
\draw   (624.22,234.91) .. controls (624.22,234.91) and (624.22,234.91) .. (624.22,234.91) .. controls (624.22,234.91) and (624.22,234.91) .. (624.22,234.91) .. controls (609.57,232.21) and (598.88,217.93) .. (600.35,203.01) .. controls (601.82,188.09) and (614.88,178.18) .. (629.53,180.87) .. controls (644.18,183.57) and (654.87,197.85) .. (653.4,212.77) .. controls (651.64,230.79) and (650.77,239.82) .. (650.82,239.88) .. controls (650.77,239.82) and (641.91,238.17) .. (624.22,234.91) -- cycle ;
\draw [color={rgb, 255:red, 239; green, 36; blue, 36 }  ,draw opacity=1 ][line width=1.5]    (719.63,250.05) -- (743.7,249.6) ;
\draw [color={rgb, 255:red, 74; green, 144; blue, 226 }  ,draw opacity=1 ][line width=2.25]    (199.7,454.2) -- (227,451.1) ;
\draw   (146,451.1) .. controls (146,428.73) and (164.13,410.6) .. (186.5,410.6) .. controls (208.87,410.6) and (227,428.73) .. (227,451.1) .. controls (227,473.47) and (208.87,491.6) .. (186.5,491.6) .. controls (164.13,491.6) and (146,473.47) .. (146,451.1) -- cycle ;
\draw   (236.34,406.12) .. controls (244.82,393.87) and (262.23,390) .. (275.23,397.47) .. controls (288.23,404.94) and (291.89,420.92) .. (283.42,433.17) .. controls (274.94,445.42) and (257.53,449.29) .. (244.53,441.82) .. controls (228.83,432.8) and (220.95,428.31) .. (220.88,428.32) .. controls (220.95,428.31) and (226.1,420.9) .. (236.34,406.12) -- cycle ;
\draw   (266.72,371.78) .. controls (262,357.66) and (269.34,341.4) .. (283.11,335.47) .. controls (296.88,329.54) and (311.87,336.19) .. (316.6,350.32) .. controls (321.32,364.44) and (313.98,380.7) .. (300.21,386.63) .. controls (283.59,393.79) and (275.26,397.4) .. (275.23,397.47) .. controls (275.26,397.4) and (272.42,388.84) .. (266.72,371.78) -- cycle ;
\draw [color={rgb, 255:red, 239; green, 36; blue, 36 }  ,draw opacity=1 ][line width=1.5]    (283.42,433.17) -- (298.7,443.6) ;
\draw [color={rgb, 255:red, 74; green, 144; blue, 226 }  ,draw opacity=1 ][line width=2.25]    (428.7,456.6) -- (455,463.1) ;
\draw   (376,449.1) .. controls (376,426.73) and (394.13,408.6) .. (416.5,408.6) .. controls (438.87,408.6) and (457,426.73) .. (457,449.1) .. controls (457,471.47) and (438.87,489.6) .. (416.5,489.6) .. controls (394.13,489.6) and (376,471.47) .. (376,449.1) -- cycle ;
\draw   (466.34,404.12) .. controls (474.82,391.87) and (492.23,388) .. (505.23,395.47) .. controls (518.23,402.94) and (521.89,418.92) .. (513.42,431.17) .. controls (504.94,443.42) and (487.53,447.29) .. (474.53,439.82) .. controls (458.83,430.8) and (450.95,426.31) .. (450.88,426.32) .. controls (450.95,426.31) and (456.1,418.9) .. (466.34,404.12) -- cycle ;
\draw   (496.72,369.78) .. controls (492,355.66) and (499.34,339.4) .. (513.11,333.47) .. controls (526.88,327.54) and (541.87,334.19) .. (546.6,348.32) .. controls (551.32,362.44) and (543.98,378.7) .. (530.21,384.63) .. controls (513.59,391.79) and (505.26,395.4) .. (505.23,395.47) .. controls (505.26,395.4) and (502.42,386.84) .. (496.72,369.78) -- cycle ;
\draw [color={rgb, 255:red, 239; green, 36; blue, 36 }  ,draw opacity=1 ][line width=1.5]    (456.42,439.17) -- (477.7,448.6) ;
\draw [color={rgb, 255:red, 239; green, 36; blue, 36 }  ,draw opacity=1 ][line width=1.5]    (718.63,436.05) -- (741.7,437.6) ;
\draw [color={rgb, 255:red, 74; green, 144; blue, 226 }  ,draw opacity=1 ][line width=2.25]    (692.7,450.2) -- (720,447.1) ;
\draw   (639,447.1) .. controls (639,424.73) and (657.13,406.6) .. (679.5,406.6) .. controls (701.87,406.6) and (720,424.73) .. (720,447.1) .. controls (720,469.47) and (701.87,487.6) .. (679.5,487.6) .. controls (657.13,487.6) and (639,469.47) .. (639,447.1) -- cycle ;
\draw   (698.12,386.24) .. controls (696.84,371.4) and (707.78,357.31) .. (722.55,354.77) .. controls (737.33,352.23) and (750.35,362.21) .. (751.63,377.05) .. controls (752.92,391.89) and (741.98,405.98) .. (727.2,408.52) .. controls (709.36,411.58) and (700.42,413.15) .. (700.38,413.2) .. controls (700.42,413.15) and (699.66,404.15) .. (698.12,386.24) -- cycle ;
\draw   (622.22,416.91) .. controls (607.57,414.21) and (596.88,399.93) .. (598.35,385.01) .. controls (599.82,370.09) and (612.88,360.18) .. (627.53,362.87) .. controls (642.18,365.57) and (652.87,379.85) .. (651.4,394.77) .. controls (649.64,412.79) and (648.77,421.82) .. (648.82,421.88) .. controls (648.77,421.82) and (639.91,420.17) .. (622.22,416.91) -- cycle ;
\draw [color={rgb, 255:red, 239; green, 36; blue, 36 }  ,draw opacity=1 ][line width=1.5]    (202.7,537.6) -- (210.5,557.6) ;
\draw [color={rgb, 255:red, 74; green, 144; blue, 226 }  ,draw opacity=1 ][line width=2.25]    (240.7,598.2) -- (268,595.1) ;
\draw   (187,595.1) .. controls (187,572.73) and (205.13,554.6) .. (227.5,554.6) .. controls (249.87,554.6) and (268,572.73) .. (268,595.1) .. controls (268,617.47) and (249.87,635.6) .. (227.5,635.6) .. controls (205.13,635.6) and (187,617.47) .. (187,595.1) -- cycle ;
\draw   (246.12,534.24) .. controls (244.84,519.4) and (255.78,505.31) .. (270.55,502.77) .. controls (285.33,500.23) and (298.35,510.21) .. (299.63,525.05) .. controls (300.92,539.89) and (289.98,553.98) .. (275.2,556.52) .. controls (257.36,559.58) and (248.42,561.15) .. (248.38,561.2) .. controls (248.42,561.15) and (247.66,552.15) .. (246.12,534.24) -- cycle ;
\draw   (184.6,639.42) .. controls (180.76,653.82) and (165.68,663.34) .. (150.92,660.7) .. controls (136.16,658.06) and (127.31,644.26) .. (131.15,629.87) .. controls (135,615.47) and (150.08,605.95) .. (164.84,608.59) .. controls (182.65,611.78) and (191.59,613.34) .. (191.65,613.31) .. controls (191.59,613.34) and (189.24,622.06) .. (184.6,639.42) -- cycle ;
\draw [color={rgb, 255:red, 239; green, 36; blue, 36 }  ,draw opacity=1 ][line width=1.5]    (526.63,514.05) -- (545.33,505.55) ;
\draw [color={rgb, 255:red, 74; green, 144; blue, 226 }  ,draw opacity=1 ][line width=2.25]    (470.7,597.2) -- (498,594.1) ;
\draw   (417,594.1) .. controls (417,571.73) and (435.13,553.6) .. (457.5,553.6) .. controls (479.87,553.6) and (498,571.73) .. (498,594.1) .. controls (498,616.47) and (479.87,634.6) .. (457.5,634.6) .. controls (435.13,634.6) and (417,616.47) .. (417,594.1) -- cycle ;
\draw   (476.12,533.24) .. controls (474.84,518.4) and (485.78,504.31) .. (500.55,501.77) .. controls (515.33,499.23) and (528.35,509.21) .. (529.63,524.05) .. controls (530.92,538.89) and (519.98,552.98) .. (505.2,555.52) .. controls (487.36,558.58) and (478.42,560.15) .. (478.38,560.2) .. controls (478.42,560.15) and (477.66,551.15) .. (476.12,533.24) -- cycle ;
\draw   (400.22,563.91) .. controls (385.57,561.21) and (374.88,546.93) .. (376.35,532.01) .. controls (377.82,517.09) and (390.88,507.18) .. (405.53,509.87) .. controls (420.18,512.57) and (430.87,526.85) .. (429.4,541.77) .. controls (427.64,559.79) and (426.77,568.82) .. (426.82,568.88) .. controls (426.77,568.82) and (417.91,567.17) .. (400.22,563.91) -- cycle ;

\draw (8,70.4) node [anchor=north west][inner sep=0.75pt]  [font=\Large]  {$L\circ E=$};
\draw (302,73.15) node [anchor=north west][inner sep=0.75pt]  [font=\Large]  {$+$};
\draw (522,73.4) node [anchor=north west][inner sep=0.75pt]  [font=\Large]  {$\circ $};
\draw (120,72.4) node [anchor=north west][inner sep=0.75pt]  [font=\Large]  {$($};
\draw (492,72.4) node [anchor=north west][inner sep=0.75pt]  [font=\Large]  {$)$};
\draw (78,225.4) node [anchor=north west][inner sep=0.75pt]  [font=\Large]  {$=$};
\draw (322,235.15) node [anchor=north west][inner sep=0.75pt]  [font=\Large]  {$+$};
\draw (556,233.15) node [anchor=north west][inner sep=0.75pt]  [font=\Large]  {$+$};
\draw (119,405.15) node [anchor=north west][inner sep=0.75pt]  [font=\Large]  {$-$};
\draw (328,405.15) node [anchor=north west][inner sep=0.75pt]  [font=\Large]  {$-$};
\draw (563,409.15) node [anchor=north west][inner sep=0.75pt]  [font=\Large]  {$-$};
\draw (81,556.4) node [anchor=north west][inner sep=0.75pt]  [font=\Large]  {$\sim $};
\draw (330,558.15) node [anchor=north west][inner sep=0.75pt]  [font=\Large]  {$+$};

\end{tikzpicture}    
\end{equation}
In particular, one has $[e,L]\sim 0$ and $[E,L]\sim 0$ and the claim follows.\blackqed\\
Finally, to complete the proof of Proposition \ref{thm: main proposition} we simply note that \textbf{Claim} implies that
\begin{align}
 (d_{cell}-t[B,-]) (e+tE)^p&=   \sum_{i=1}^p(e+tE)^{i-1}\big(d_{cell}(e+tE)-t[B,(e+tE)]\big)(e+tE)^{p-i}   \\
 &= t\sum_{i=1}^p(e+tE)^{i-1}\;L\;(e+tE)^{p-i}\\
 &\sim tL \sum_{i=1}^p(e+tE)^{p-1}\\
 &=0,
\end{align}
where the last equality follows since we are working over characteristic $p$.\qed\\
\begin{rmk}
Under \eqref{eq: action of cyclic cacti on HH_* and HH^*}, $L$ gives rise to the Lie action of $HH^*(\mathcal{A})$ on $HH_*(\mathcal{A})$, and \eqref{eq:graphical Cartan homotopy formula} is just a graphical representation of the non-commutative Cartan homotopy formula \cite{Get}. The \textbf{Claim} above is also a shadow of a classical identity in non-commutative calculus. Namely, writing $\iota=e+tE$, then for any Hochschild cochains $D,E$
\begin{equation}\label{eq:DGT homotopy}
[L_D,\iota_E]-(-1)^{|D|+1}\iota_{[D,E]}=[b+tB,T(D,E)]-T(\delta D,E)-(-1)^{|D|}T(D,\delta E),  
\end{equation}
where $T(-,-)$ is the Daletsky-Gelfand-Tsygan homotopy operator \cite{DGT} given by
\begin{equation}
T(D,E)(a_0\otimes\cdots\otimes a_n):=\sum \pm D(a_{j_1+1},\cdots,E(a_{j_2+1},\cdots),\cdots,a_0,\cdots)\otimes \cdots.
\end{equation}
In particular, $T(D,E)=0$ if $D$ is a Hochschild cochain of length one. Setting $D=E$ to be the differential $d$ of $\mathcal{A}$, \eqref{eq:DGT homotopy} implies that $[L_d,\iota_d]=0$.\\ 
\end{rmk}
\begin{rmk}
Note that Proposition \ref{thm: main proposition} also holds if one replaces $\mathbf{k}$ by any characteristic $p$ commutative ring $K$, and moreover implies that $(e+tE)^p$ is $(\mathcal{A},[\phi^{\otimes p}])$-equivalent to a degree $0$ cocycle of $(C^{cell}_{-*}(\mathrm{UCact}^p_{\circlearrowright};K)((t)),d_{cell}-t[B,-])$ for any even Hochschild cocycle $\phi$ such that
\begin{equation}\label{eq:condition for even HH cocycle}
\phi\in \mathrm{Hom}_{K}(\mathcal{A}[1],\mathcal{A})\subset \prod_{d\geq 0}\mathrm{Hom}_{K}(\mathcal{A}[1]^{\otimes d},\mathcal{A})\quad\textrm{and}\quad \phi\circ \phi=0;
\end{equation}
compare proof of Lemma \ref{thm:sim equivalence implies other equivalence}. This has the following application. \par\indent
Namely, let $\mathcal{A}_s$ be a flat family of d($\mathbb{Z}/2$)g algebras over the affine line $K=\mathbf{k}[s]$ representing a \emph{linear deformation} of $\mathcal{A}_0:=\mathcal{A}_s|_{s=0}$, i.e. the product of $\mathcal{A}_s$ is constant and the differential $d_s$ is of the form $d+sd_1$. Then the Kodaira-Spencer cocycle 
\begin{equation}
\frac{\partial d_s}{\partial s}=d_1   \in CC^2(\mathcal{A}_s,\mathcal{A}_s)
\end{equation}
satisfies the condition \eqref{eq:condition for even HH cocycle}, and thus $(e+tE)^p$ is $(\mathcal{A}_s,[d_1^{\otimes p}])$-equivalent to a degree $0$ cocycle of $(C^{cell}_{-*}(\mathrm{UCact}^p_{\circlearrowright};K)((t)),d_{cell}-t[B,-])$.
As a result, arguing verbatim as in Section 4.2 shows that 
\begin{equation}
F^{GGM}_{t\partial_s}=c\cdot\bigcap\nolimits^{C_p}([d_1],-)
\end{equation}
for some constant $c$. In a hypothetical situation where $\mathcal{A}_s/\mathbf{k}[s]$ is smooth\footnote{This rarely happens on the nose, but can sometimes be achieved after a suitable localization, cf. \cite[Proposition 3.1.12]{PS2}.}, then using the same argument as Corollary \ref{thm: nilpotence of p curvature} one may deduce that $F^{GGM}_{t\partial_s}$ is nilpotent.   
\end{rmk}

\subsection{Proof of Theorem \ref{thm: p curvature equals multiplicative operation}}
For the reader's convenience, we record the following immediate consequence of Corollary \ref{thm: simplified p curvature} and Proposition \ref{thm: main proposition}.\\
\begin{cor}\label{thm: p curvature as KS operation}
There exists a degree $0$ cohomology class 
\begin{equation}
[\alpha]\in H^0\big(C_{-*}^{cell}(\mathrm{UCact}^p_{\circlearrowright};\mathbf{k})((t)), d_{cell}-t[B,-]\big)\cong H^0\big((\mathrm{UCact}^p_{\circlearrowright,\mathbf{k}})^{tS^1}\big)     
\end{equation}
such that the $p$-curvature of the $t$-connection on $HH^{per}_*(\mathcal{A})$ is equal to 
\begin{equation}
 \Xi_{\mathcal{A}}([\alpha])([d^{\otimes p}],-),   
\end{equation}
 where $d$ is the differential of $\mathcal{A}$ and $\Xi_{\mathcal{A}}$ is the equivariant action map from \eqref{eq: equivariant action of cyclic cacti on HH_* and HH^*}. In particular, the $p$-curvature of the $t$-connection is a Kontsevich-Soibelman operation.  \qed
\end{cor}
Therefore, to prove Theorem \ref{thm: p curvature equals multiplicative operation}, it suffices to show that every Kontsevich-Soibelman operation associated to a degree $0$ class $[\alpha]\in H^0\big((\mathrm{UCact}^p_{\circlearrowright,\mathbf{k}})^{tS^1}\big)$ is a constant multiple of a multiplicative action of $HH^{even}(\mathcal{A})^{(1)}$ on $HH^{per}_*(\mathcal{A})$. As we will see, this reduces to certain basic topological properties of the Kontsevich-Soibelman operad.\par\indent
Recall from Theorem \ref{thm:comparison of Cact with Cyl} that there is an $\Sigma_p\times S^1\times S^1$-equivariant quasi-equivalence 
\begin{equation}
 \mathrm{Cact}^p_{\circlearrowright,\mathbf{k}}\simeq C_{-*}(\mathrm{Cyl}(p,1);\mathbf{k}). 
\end{equation}
Taking quotient of the $\Sigma_p$-action and restricting to the diagonal $S^1\subset S^1\times S^1$ on both sides, one obtains an $S^1$-equivariant quasi-equivalence 
\begin{equation}
 \mathrm{UCact}^p_{\circlearrowright,\mathbf{k}}\simeq C_{-*}(\mathrm{UCyl}(p,1);\mathbf{k}),
\end{equation}
where $\mathrm{UCyl}(p,1):=\mathrm{Cyl}(p,1)/\Sigma_p$ is the unordered configuration space of disks on a cylinder with one marked point on each boundary component modulo overall rotation (compare Section 3.2), and the $S^1$-action simultaneously rotates the marked points clockwise (or equivalently, rotates the lateral surface of the cylinder counterclockwise). \par\indent
There is an $S^1$-equivariant homotopy equivalence, by identifying the cylinder with $\mathbb{C}^*$, the disks with their centers and the difference between the two boundary marked points with $S^1$,
\begin{equation}\label{eq:UCyl as UConf}
\mathrm{UCyl}(p,1)\simeq \mathrm{UConf}_p(\mathbb{C}^*)\times S^1,    
\end{equation}
where $\mathrm{UConf}_p(\mathbb{C}^*)$ denotes the unordered configuration space of $p$ points in $\mathbb{C}^*$, and the $S^1$-action on $\mathrm{UConf}_p(\mathbb{C}^*)\times S^1$ is induced by the standard counterclockwise rotation $S^1\circlearrowright \mathbb{C}^*$ on the first factor and trivial on the second factor. \par\indent
By Kunneth formula,
\begin{equation}\label{eq:Kunneth decomp of equivariant cohomology of cyclic cacti}
H^0(C_{-*}(\mathrm{UCyl}(p,1);\mathbf{k})^{tS^1})\cong H^0(C_{-*}(\mathrm{UConf}_p(\mathbb{C}^*);\mathbf{k})^{tS^1})\otimes H_0(S^1;\mathbf{k})\oplus H^{1}(C_{-*}(\mathrm{UConf}_p(\mathbb{C}^*);\mathbf{k})^{tS^1})\otimes H_1(S^1;\mathbf{k}).
\end{equation}
\begin{lemma}\label{thm:action of degree 0 equivariant cohomology of UCyl}
The action map
\begin{equation}
H^0(C_{-*}(\mathrm{UCyl}(p,1);\mathbf{k})^{tS^1})\cong  H^0((\mathrm{UCact}^p_{\circlearrowright,\mathbf{k}})^{tS^1})\xrightarrow{\Xi_{\mathcal{A}}} \mathrm{Hom}_{\mathbf{k}((t))}\big(H^*((CC^*(\mathcal{A})^{\otimes p})^{\Sigma_p})\otimes HH_*^{per}(\mathcal{A}),HH_*^{per}(\mathcal{A})\big)  
\end{equation}
factors through the projection 
\begin{equation}
H^0(C_{-*}(\mathrm{UCyl}(p,1);\mathbf{k})^{tS^1})\rightarrow H^0(C_{-*}(\mathrm{UConf}_p(\mathbb{C}^*);\mathbf{k})^{tS^1})\otimes H_0(S^1;\mathbf{k}).     
\end{equation}
\end{lemma}
\emph{Proof}. Let $[S^1]\in H_1(S^1;\mathbf{k})$ be the fundamental class. The image of an element \begin{equation}
[\alpha]\otimes [S^1]\in H^{1}(C_{-*}(\mathrm{UConf}_p(\mathbb{C}^*);\mathbf{k})^{tS^1})\otimes H_1(S^1;\mathbf{k})    
\end{equation} in $H^0(C_{-*}(\mathrm{UCyl}(p,1);\mathbf{k})^{tS^1})\cong H^0\big((\mathrm{UCact}^p_{\circlearrowright,\mathbf{k}})^{tS^1}\big)$ via \eqref{eq:Kunneth decomp of equivariant cohomology of cyclic cacti} is equal to $B\circ [\alpha]$\footnote{Where we abused notation to denote by $[\alpha]$ its image under the inclusion $$
 H^0(C_{-*}(\mathrm{UConf}_p(\mathbb{C}^*);\mathbf{k})^{tS^1})\xrightarrow{\iota}H^0(C_{-*}(\mathrm{UCyl}(p,1);\mathbf{k})^{tS^1})    \cong H^0\big((\mathrm{UCact}^p_{\circlearrowright,\mathbf{k}})^{tS^1}\big),
$$
where $\iota$ is induced by identifying (up to homotopy) $\mathrm{UConf}_p(\mathbb{C}^*)$ with the subspace of $\mathrm{UCyl}(p,1)$ where the two boundary marked points are aligned.}, where $B$ is the $1$-cell representing the Connes operator (cf. Figure \ref{fig:cells_in_cyclic_cacti_and_their_actions}). The statement follows since the Connes operator acts trivially on the periodic cyclic homology. \qed\\

\begin{lemma}\label{thm:elementary observation regarding UConf_p^Cp}
There is an $S^1$-equivariant homotopy equivalence $\mathrm{UConf}_p(\mathbb{C}^*)^{C_p}\simeq S^1/C_p$, where $\mathrm{UConf}_p(\mathbb{C}^*)^{C_p}$ is equipped with the $S^1$-action induced from its inclusion into $\mathrm{UConf}_p(\mathbb{C}^*)$.    
\end{lemma}
\emph{Proof}. First, note that there is a homeomorphism
\begin{equation}\label{eq:Cp fixed point is also UConf}
\mathrm{UConf}_{p}(\mathbb{C}^*)^{C_p}\cong  \mathrm{UConf}_1(\mathbb{C}^*)   
\end{equation}    
Indeed, let $\zeta=e^{2\pi i/p}$, and define
\begin{equation}
H=\{z\in\mathbb{C}^*:\arg(z)\in[0,2\pi/p]\}/(x\sim\zeta x,x\in\mathbb{R}_{>0})\end{equation}
Then it is easy to check that the map $\mathrm{UConf}_1(H)\rightarrow \mathrm{UConf}_{p}(\mathbb{C}^*)^{C_p}$ given by $\{z_1\}\mapsto \{z_1,\zeta z_1,\cdots,\zeta^{p-1}z_1\}$ is a homeomorphism. Moreover, there is a homeomorphism $H\cong \mathbb{C}^*$ via the $p$-fold map $z\mapsto z^p$. The result follows since there is an $S^1$-equivariant homotopy equivalence $\mathrm{UConf}_1(\mathbb{C}^*)=\mathbb{C}^*\simeq S^1$. \qed\\
\begin{lemma}\label{thm:simplifying S^1 Tate fixed points of UConf_p}
The inclusion $\mathrm{UConf}_{p}(\mathbb{C}^*)^{C_p}\hookrightarrow \mathrm{UConf}_{p}(\mathbb{C}^*)$ induces a quasi-isomorphism
\begin{equation}
 C_{-*}(\mathrm{UConf}_{p}(\mathbb{C}^*)^{C_p};\mathbf{k})^{tS^1}\xrightarrow{\simeq} C_{-*}(\mathrm{UConf}_{p}(\mathbb{C}^*);\mathbf{k})^{tS^1}.   
\end{equation}
\end{lemma} 
\emph{Proof}. By $C_p$-localization (cf. for instance \cite[Theorem 5.1]{Che3}) the induced map
\begin{equation}
  C_{-*}(\mathrm{UConf}_{p}(\mathbb{C}^*)^{C_p};\mathbf{k})^{tC_p}\xrightarrow{\simeq} C_{-*}(\mathrm{UConf}_{p}(\mathbb{C}^*);\mathbf{k})^{tC_p}  
\end{equation}
is a quasi-isomorphism. The statement then follows from the fact that there is a functorial decomposition $X^{tC_p}=X^{tS^1}\oplus X^{tS^1}\theta$ for an  $S^1$-chain complex $X$ over characteristic $p$, cf. \cite[Appendix A.3]{Che3}.\qed\par\indent
By Lemma \ref{thm:elementary observation regarding UConf_p^Cp} and Lemma \ref{thm:simplifying S^1 Tate fixed points of UConf_p},
\begin{equation}\label{eq:C_p localization for UConf_p}
 H^*(C_{-*}(\mathrm{UConf}_{p}(\mathbb{C}^*);\mathbf{k})^{tC_p}) \cong  H_{-*}(\mathrm{UConf}_{p}(\mathbb{C}^*)^{C_p};\mathbf{k})\otimes \mathbf{k}((t,\theta))\cong H_{-*}(S^1/C_p;\mathbf{k})\otimes \mathbf{k}((t,\theta)),
\end{equation}
where we identify $\mathbf{k}((t,\theta))$ with $H^*(\mathbf{k}^{tC_p})$. Moreover, the degree $0$ and degree $1$ generators of $ H_{-*}(S^1/C_p;\mathbf{k})$ correspond, under the equivalences of \eqref{eq:C_p localization for UConf_p} and \eqref{eq:UCyl as UConf}, to the generators $[e_0],[e_1]\in H_{-*}(\mathrm{UCyl}(p,1)^{C_p};\mathbf{k})$ shown in Figure \ref{fig:generator_of_UCyl}.
\begin{figure}[H]
 \centering
 \includegraphics[width=0.7\textwidth]{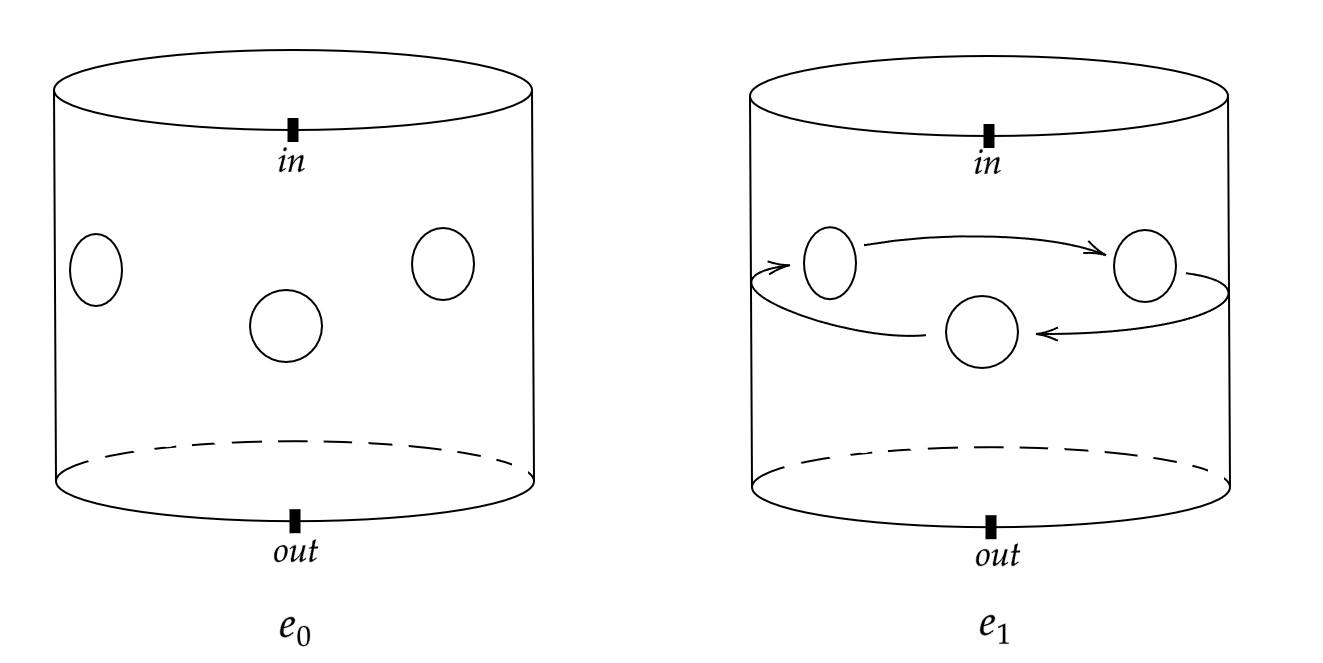}
 \caption{$e_0$ represents the $C_p$-invariant $0$-cycle where the centers of the $p$ disks lie at the $p$-th roots of unity of a chosen equator; $e_1$ represents the $C_p$-invariant $1$-cycle obtained from $e_0$ by moving the centers of the $p$ disks simultaneously through angle $[0,2\pi /p]$.}
\label{fig:generator_of_UCyl}
\end{figure}
As a preliminary step towards Theorem \ref{thm: p curvature equals multiplicative operation}, we prove a $C_p$-version of it as follows. We abuse notation and write  $[e_0]$ also for its image under
\begin{equation}
H_0(\mathrm{UCyl}(p,1)^{C_p};\mathbf{k})\rightarrow H^0(C_{-*}(\mathrm{UCyl}(p,1);\mathbf{k})^{tC_p})\cong H^0((\mathrm{UCact}^p_{\circlearrowright,\mathbf{k}})^{C_p}). 
\end{equation}
Applying \eqref{eq:restricting KS action} to $[e_0]$ gives a map 
\begin{equation}\label{eq:C_p equivariant action of e_0}
\Xi^p_{\mathcal{A}}([e_0])(-,-): H^*((CC^*(\mathcal{A})^{\otimes_R p})^{h\Sigma_p})\rightarrow \mathrm{End}_{\mathbf{k}((t,\theta))}(HH_*^{C_p,per}(\mathcal{A})).
\end{equation}
Restricting to the even part\footnote{The odd part of the $p$-th power map is $[\phi]\mapsto [\phi^{\otimes p}]: HH^{odd}(\mathcal{A})^{(1)}\rightarrow  H^{odd}((CC^*(\mathcal{A})^{\otimes p}\otimes \mathbf{k}(p))^{h\Sigma_p})$, where $\mathbf{k}(p)$ denotes the sign representation of $\Sigma_p$. This case can be dealt with separately using the approach of \cite[Section 5.2]{Che3}, which we will not detail here since it is not needed.} and pre-composing with the $p$-th power map
\begin{equation}\label{eq:pth power map}
[\phi]\mapsto [\phi^{\otimes p}]: HH^{even}(\mathcal{A})\rightarrow  H^{even}((CC^*(\mathcal{A})^{\otimes p})^{h\Sigma_p}),
\end{equation}
gives
\begin{equation}\label{eq:pth power action of e_0}
\Xi^p_{\mathcal{A}}([e_0])([(-)^{\otimes p}],-): HH^{even}(\mathcal{A})\rightarrow \mathrm{End}_{\mathbf{k}((t,\theta))}(HH_*^{C_p,per}(\mathcal{A})).    
\end{equation}\\
\begin{prop}\label{thm:pth power action of e_0 is multiplicative}
The map \eqref{eq:pth power action of e_0} satisfies the three properties specified in \eqref{eq:multiplicative plus Frobenius linear plus zero property}.    
\end{prop}
\emph{Proof.} The only nontrivial property to check is the multiplicativity. This follows from Figure \ref{fig:multiplicativity_of_pth_power_e_0_action}. Again, the pictures are drawn for the operad $\mathrm{Cyl}$, which is justified by Theorem \ref{thm:comparison of Cact with Cyl}.
\begin{figure}[H]
 \centering
 \includegraphics[width=1.15\textwidth]{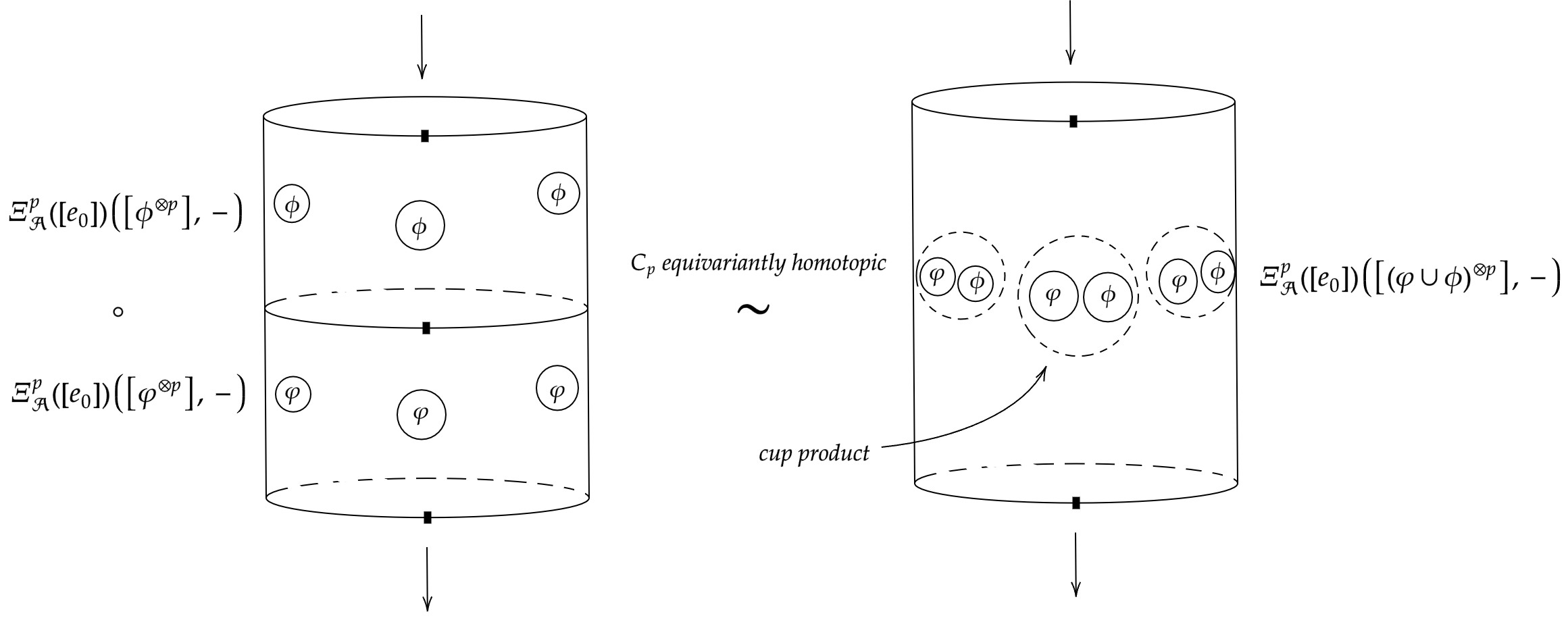}
 \caption{}
\label{fig:multiplicativity_of_pth_power_e_0_action}
\end{figure}\qed\\
\begin{rmk}\label{thm:e_0 gives C_p equivariant cap product}
In fact, it was proved in \cite[Corollary 6.21]{Che3} that using the explicit chain model for $CC^{C_p,per}_*(\mathcal{A})=(CC_*(\mathcal{A},\mathcal{A}^{\otimes^{\mathbb{L}}p})((t,\theta)),d_{eq})$ (cf. Section 2.3 loc.cit.), there is a chain-level formula (called the \emph{$C_p$-equivariant cap product}) representing $\Xi^p_{\mathcal{A}}([e_0])([\phi^{\otimes p}],-)$ of the form
$$\mathbf{a}^1\otimes a^1_1\otimes \cdots\otimes a^1_{k_1}\otimes\cdots\otimes \mathbf{a}^p\otimes a^p_1\otimes \cdots\otimes a^p_{k_p}\mapsto$$
\begin{equation}\label{eq:p fold cap product}
\sum_{0\leq j_1\leq k_1,\cdots,0\leq j_p\le k_p} \pm\phi(a^p_{j_p+1} \cdots,,a^p_{k_p})\mathbf{a}^1\otimes a^1_1\otimes\cdots\otimes a^1_{j_1}\otimes\phi(a^1_{j_1+1} \cdots,,a^1_{k_1})\mathbf{a}^2\otimes a^2_1\otimes\cdots \otimes\phi(a^{p-1}_{j_{p-1}+1} \cdots,,a^{p-1}_{k_{p-1}})\mathbf{a}^p\otimes a^p_1\otimes\cdots\otimes a^p_{j_p}.
\end{equation}
Proposition \ref{thm:pth power action of e_0 is multiplicative} also directly follows from this explicit description. In fact, using this interpretation one also obtains that the action \eqref{eq:pth power action of e_0} is additive, although we do not need it here. 
\end{rmk}
We are now ready to prove the main theorem of this section Theorem \ref{thm: p curvature equals multiplicative operation}, which we restate for the reader's convenience. \\
\begin{thm}(Theorem \ref{thm: p curvature equals multiplicative operation})\label{thm: p curvature equals multiplicative operation restated}
Let $\mathcal{A}$ be a d($\mathbb{Z}/2$)g algebra over $\mathbf{k}$, and denote by $F^{\mathcal{A}}_{2t^2\partial_t}$ the $p$-curvature of the $t$-connection (cf. Definition \ref{thm:t connection}) on the periodic cyclic homology along the vector field $2t^2\partial_t$. Then, there exists a map
\begin{equation}
\bigcap\nolimits^{C_p}: HH^{even}(\mathcal{A})\times HH_*^{per}(\mathcal{A})\rightarrow  HH_*^{per}(\mathcal{A})   
\end{equation}
such that
\begin{equation}\label{eq:multiplicative plus Frobenius linear plus zero property restated}
\bigcap\nolimits^{C_p}([\phi]\cup[\varphi],\alpha)=  \bigcap\nolimits^{C_p}([\phi],\bigcap\nolimits^{C_p}([\varphi],\alpha))\quad,\quad\bigcap\nolimits^{C_p}([a\phi],\alpha)=a^p\bigcap\nolimits^{C_p}([\phi],\alpha)\quad,\quad \bigcap\nolimits^{C_p}(0,\alpha)=0
\end{equation}
for all $[\phi],[\varphi]\in HH^{even}(\mathcal{A}), \alpha\in HH^{per}(\mathcal{A}), a\in \mathbf{k}$, together with a constant $c\in \mathbf{k}$ such that 
\begin{equation}
 F^{\mathcal{A}}_{2t^2\partial_t}=c\cdot\bigcap\nolimits^{C_p}([d],-).
\end{equation}
Here, $HH^{even}(\mathcal{A})$ denotes the even part of Hochschild cohomology, $\cup$ denotes the cup product and the differential $d$ is viewed as a Hochschild cohomology class of length one. 
\end{thm}
\emph{Proof}. By Corollary \ref{thm: p curvature as KS operation}, it suffices to show that for any $[\alpha]\in H^0(\mathrm{UCact}^p_{\circlearrowright,\mathbf{k}})^{tS^1})$, the induced map
\begin{equation}
HH^{even}(\mathcal{A})\times HH_*^{per}(\mathcal{A})\xrightarrow{\Xi_{\mathcal{A}}([\alpha])([(-)^{\otimes p}],-)} HH_*^{per}(\mathcal{A})    
\end{equation}
is a constant multiple of an action of $HH^{even}(\mathcal{A})$ on $HH_*^{per}(\mathcal{A})$ satisfying \eqref{eq:multiplicative plus Frobenius linear plus zero property restated}. Again, only multiplicativity is nontrivial. By Lemma \ref{thm:action of degree 0 equivariant cohomology of UCyl}, it suffices to prove this for $[\alpha]$ lying in the image of 
\begin{equation}
H^0(C_{-*}(\mathrm{UConf}_p(\mathbb{C}^*);\mathbf{k})^{tS^1})\xrightarrow{\iota}H^0(C_{-*}(\mathrm{UCyl}(p,1);\mathbf{k})^{tS^1})    \cong H^0\big((\mathrm{UCact}^p_{\circlearrowright,\mathbf{k}})^{tS^1}\big),
\end{equation}
where $\iota$ is induced by identifying (up to homotopy) $\mathrm{UConf}_p(\mathbb{C}^*)$ with the subspace of $\mathrm{UCyl}(p,1)$ where the two boundary marked points are aligned. \par\indent
Consider the commutative diagram (cf. \eqref{eq:restricting KS action})
\begin{equation}\label{eq:restricting degree 0 KS action}
\begin{tikzcd}[row sep=1.2cm, column sep=0.8cm]
H^0(\mathrm{UCact}^k_{\circlearrowright,\mathbf{k}})^{tS^1})\otimes_{\mathbf{k}((t))} HH_*^{per}(\mathcal{A})\times HH^{even}(\mathcal{A})\arrow[d,"{\mathrm{Res}_{C_p\subset S^1}}\;\otimes\; {\mathrm{Res}_{C_p\subset S^1}}\;\times \mathrm{id}"]\arrow[rrr,"{\Xi_{\mathcal{A}}(-)([(-)^{\otimes p}],-)}"]& & & HH_*^{per}(\mathcal{A})\arrow[d,"{\mathrm{Res}_{C_p\subset S^1}}"]    \\
H^0(\mathrm{UCact}^k_{\circlearrowright,\mathbf{k}})^{tC_p})\otimes_{\mathbf{k}((t,\theta))} HH_*^{C_p,per}(\mathcal{A})\times HH^{even}(\mathcal{A})\arrow[rrr,"{\Xi^p_{\mathcal{A}}(-)([(-)^{\otimes p}],-)}"]& & &HH_*^{C_p,per}(\mathcal{A})
\end{tikzcd}    
\end{equation}
induced from the restriction of $S^1$ to $C_p$ Tate fixed points. It is well known that the vertical maps in \eqref{eq:restricting degree 0 KS action} are (split) injective, cf. \cite[Appendix A.3]{Che3}. Hence, it suffices to show that any $[\alpha]\in H^0(\mathrm{UCact}^k_{\circlearrowright,\mathbf{k}})^{tC_p}) \cong H^0(C_{-*}(\mathrm{UCyl}(p,1);\mathbf{k})^{tC_p})$ in the image of
\begin{equation}\label{eq:S^1 to C_p fixed point restriction for UConf}
H^0(C_{-*}(\mathrm{UConf}_p(\mathbb{C}^*);\mathbf{k})^{tS^1})\xrightarrow{\mathrm{Res}_{C_p\subset S^1}}H^0(C_{-*}(\mathrm{UConf}_p(\mathbb{C}^*);\mathbf{k})^{tC_p})\xrightarrow{\iota}H^0(C_{-*}(\mathrm{UCyl}(p,1);\mathbf{k})^{tC_p})       
\end{equation}
induces a constant multiple of a multiplicative action 
\begin{equation}
HH^{even}(\mathcal{A})\times HH_*^{C_p,per}(\mathcal{A})\xrightarrow{\Xi_{\mathcal{A}}^p([\alpha])([(-)^{\otimes p}],-)} HH_*^{C_p,per}(\mathcal{A}). 
\end{equation}
Indeed, it is proved in \cite[Lemma 5.13 and Corollary 5.14]{Che3} $H^0(C_{-*}(\mathrm{UConf}_p(\mathbb{C}^*);\mathbf{k})^{tS^1})$ is one-dimensional and a standard generator gets mapped under $\mathrm{Res}_{C_p\subset S^1}$ to 
\begin{equation}
[e_0]-[e_1]\theta,    
\end{equation}
where $[e_0],[e_1]$ are the classes from Figure \ref{fig:generator_of_UCyl}. Since $\theta^2=0$, for $[\phi],[\varphi]\in HH^{even}(\mathcal{A})$
$$\Xi_{\mathcal{A}}([e_0]-[e_1]\theta)([\varphi^{\otimes p}],-)\circ \Xi_{\mathcal{A}}([e_0]-[e_1]\theta)([\phi^{\otimes p}],-)=\Xi_{\mathcal{A}}([e_0])([\varphi^{\otimes p}],-)\circ \Xi_{\mathcal{A}}([e_0])([\phi^{\otimes p}],-)$$
$$-\big(\Xi_{\mathcal{A}}([e_1])([\varphi^{\otimes p}],-)\circ \Xi_{\mathcal{A}}([e_0])([\phi^{\otimes p}],-)+\Xi_{\mathcal{A}}([e_0])([\varphi^{\otimes p}],-)\circ \Xi_{\mathcal{A}}([e_1])([\phi^{\otimes p}],-)\big)\theta$$
\begin{equation}
=\Xi_{\mathcal{A}}([e_0]-[e_1]\theta)([(\varphi\cup \phi)^{\otimes p}],-),
\end{equation}
where the last equality follows from Figure \ref{fig:multiplicativity_of_pth_power_e_0_action} and Figure \ref{fig:C_p_Lie_cup_product_formula}. This completes the proof.\qed
\begin{figure}[H]
 \centering
 \includegraphics[width=1.1\textwidth]{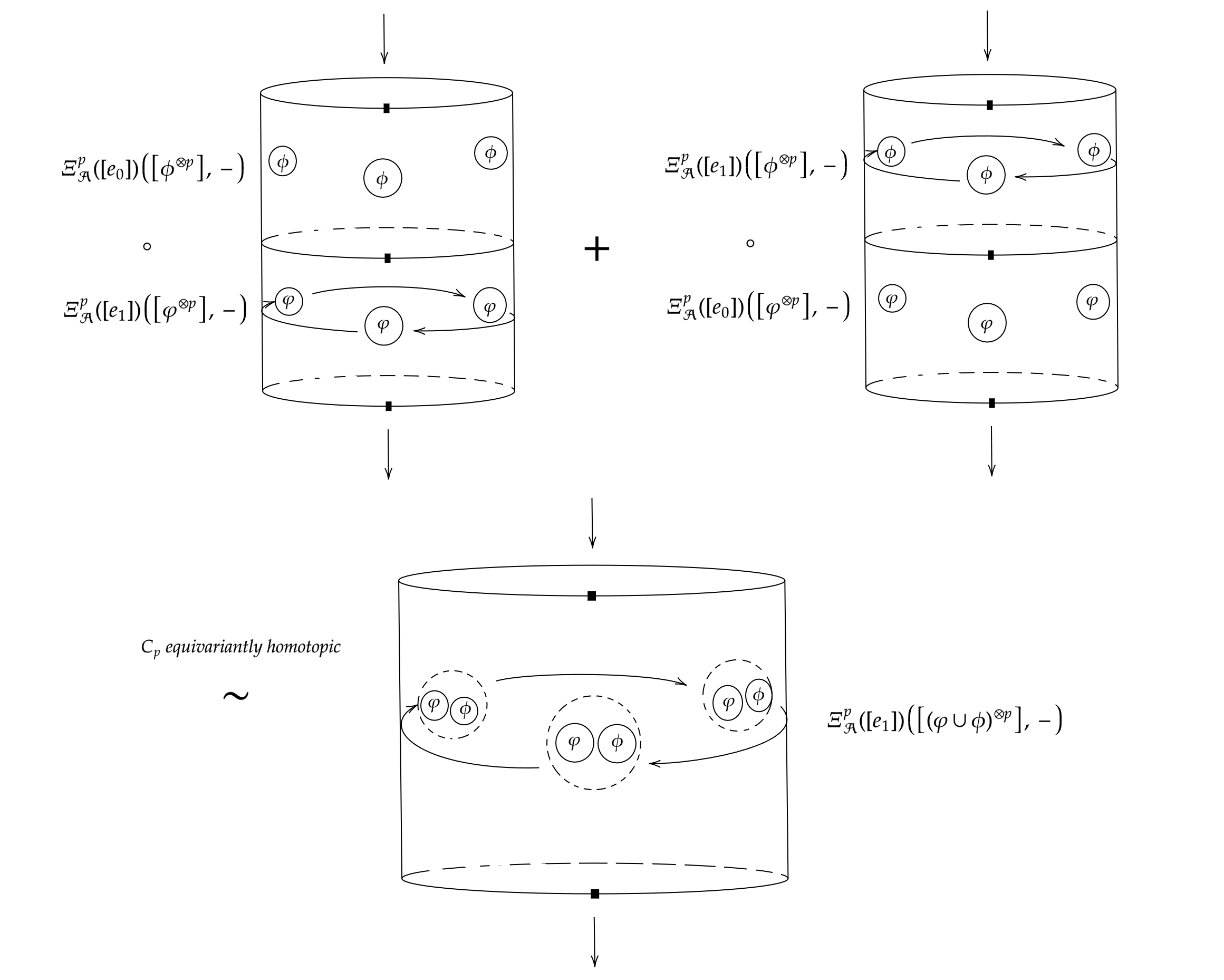}
 \caption{}
\label{fig:C_p_Lie_cup_product_formula}
\end{figure}
\begin{rmk}
We leave two remarks for the sake of completeness, even though they are not required for proving the main results. 
\begin{itemize} 
    \item It was further proved in \cite[Corollary 5.15]{Che3} that the endomorphism $\Xi_{\mathcal{A}}([e_0]-[e_1]\theta)$ can be expressed as the composition
    \begin{equation}\label{eq:action of e_0-e_1theta}
HH_*^{per}(\mathcal{A})\xrightarrow{\mathrm{Res}_{C_p\subset S^1}}    HH^{C_p,per}_*(\mathcal{A})\xrightarrow{\Xi^p([e_0])}  HH^{C_p,per}_*(\mathcal{A})=HH_*^{per}(\mathcal{A}) \oplus HH_*^{per}(\mathcal{A})\theta\xrightarrow{pr_1}HH_*^{per}(\mathcal{A}).
\end{equation}
Recall that $\Xi^p([e_0])$ agrees with the $C_p$-equivariant cap product (cf. Remark \ref{thm:e_0 gives C_p equivariant cap product}). 
    \item It is not difficult to see that the constant $c$ in Theorem \ref{thm: p curvature equals multiplicative operation restated} must equal $1$. Indeed, since the equality $F^{\mathcal{A}}_{2t^2\partial_t}=c\cdot\Xi_{\mathcal{A}}([e_0]-[e_1]\theta)(d^{\otimes p},-)$ is proved at the operadic level (which in particular implies that $c$ is independent of $\mathcal{A}$), it suffices to show that $c=1$ for one choice of d$(\mathbb{Z}/2)$g category $\mathcal{A}$. For instance, there is an abundance of supply of such $\mathcal{A}$ given by (summands of) the Fukaya category of a monotone symplectic manifold satisfying Abouzaid's generation criterion, cf. \cite{Che1}\cite{Che2}. 
\end{itemize}   
\end{rmk}

\section{Proof of Theorem \ref{thm:main theorem}}
In this section we assemble the various ingredients and complete the proof of the main theorem, which we recall for the reader's convenience.\\
\begin{thm}\label{thm:main theorem restated}
Let $\mathcal{A}$ be a smooth proper d($\mathbb{Z}/2$)g algebra over $\mathbb{C}$. Then the $t$-connection on $HH_*^{per}(\mathcal{A})$ (cf. \ref{thm:t connection}) has a regular singularity and quasi-unipotent monodromy at $t=0$.  
\end{thm}
As mentioned in the introduction, the proof of Theorem \ref{thm:main theorem restated} uses a spreading out argument, i.e. reduces the statement over $\mathbb{C}$ to properties of the $t$-connection over positive characteristic fields (that are proved in Section 4). By a \emph{global ring} we mean an integral domain $R$ which is finitely generated as an algebra over $\mathbb{Z}$ such that $\mathrm{Frac}(R)$ has characteristic $0$. \par\indent
The possibility of a spreading out argument builds on the following two well-known results. \\
\begin{thm}\label{thm:Toen theorem}    
Let $\mathcal{A}$ be a smooth proper d$(\mathbb{Z}/2)$g algebra over $\mathbb{C}$. Then there exists a global ring $R\subset \mathbb{C}$ together with a smooth proper  d$(\mathbb{Z}/2)$g algebra $\mathcal{A}_R$ over $R$ and a quasi-equivalence of d$(\mathbb{Z}/2)$g algebras $\mathcal{A}_R\otimes_R\mathbb{C}\simeq \mathcal{A}$.\qed
\end{thm}
\emph{Proof}. This statement is originally proved by To\"{e}n \cite{To} when $\mathcal{A}$ is $\mathbb{Z}$-graded. The adaptation to the $\mathbb{Z}/2$-graded setting follows from \cite[Theorem 3.2]{KKM}. Indeed, \cite[Theorem 3.2]{KKM} proves that the functor 
\begin{equation}
DAlg^{sat}:DComm(\mathbb{S})\rightarrow Cat     
\end{equation}
sending an $E_{\infty}$-ring spectrum $\mathcal{R}$ to the category of smooth proper $E_1$-algebras over $\mathcal{R}$ commutes with filtered homotopy colimits. Theorem \ref{thm:Toen theorem} then follows from the observation that d($\mathbb{Z}/2$)g algebras over an ordinary commutative ring $R$ are equivalent to $E_1$-algebras over (the Eilenberg-Maclane spectrum of) the graded algebra $R[u,u^{-1}]$ (where $u$ has degree $2$), and the fact that $\mathbb{C}[u,u^{-1}]$ can be written as a filtered homotopy colimit of $R[u,u^{-1}]$, where $R\subset \mathbb{C}$ is finite type over $\mathbb{Z}$. \qed \\
\begin{thm}(Katz)\label{thm:Katz theorem}
Let $R$ be a global ring. Let $(M,\nabla)$ be a finite rank free $R((t))$-module equipped with a connection. Suppose that for almost all maximal ideals $\mathfrak{m}\in\mathrm{mSpec}R$, the connection $(M,\nabla)\otimes_{R((t))}\kappa(\mathfrak{m})((t))$\footnote{The notation $(M,\nabla)\otimes_{R((t))}\kappa(\mathfrak{m})((t))$ stands for the pullback connection. Equivalently, it is obtained from fixing an $R((t))$-basis of $M$ and then reducing the coefficients of the corresponding connection matrix to $\kappa(\mathfrak{m})((t))$ (which is easily seen to be independent of the choice of basis).} has nilpotent $p$-curvature, where $p$ is the characteristic of the finite residue field $\kappa(\mathfrak{m})$, then the connection $(M,\nabla)\otimes_{R((t))}\mathrm{Frac}(R)((t))$ has regular singularity and quasi-unipotent monodromy at $t=0$. \par\indent
Moreover, if there exists a positive number $N$ such that for almost all $\mathfrak{m}\in\mathrm{mSpec}R$, the $p$-curvature of $(M,\nabla)\otimes_{R((t))}\kappa(\mathfrak{m})((t))$ is nilpotent of order $\leq N$, then the size of largest Jordan block of the monodromy of $(M,\nabla)\otimes_{R((t))}\mathrm{Frac}(R)((t))$ is $\leq N$.\qed
\end{thm}
We remark that Theorem \ref{thm:Katz theorem} was originally proved in \cite{Ka1} over the rational functions, and we refer the readers to \cite[Appendix A]{Che2} for a straightforward adaptation to the case over formal Laurent series.\par\indent
\emph{Proof of Theorem \ref{thm:main theorem restated}}. By Theorem \ref{thm:Toen theorem}, we fix a global ring $R\subset \mathbb{C}$ together with a smooth proper  d$(\mathbb{Z}/2)$g algebra $\mathcal{A}_R$ over $R$ and a quasi-equivalence of d$(\mathbb{Z}/2)$g algebras \begin{equation} \mathcal{A}_R\otimes_R\mathbb{C}\simeq \mathcal{A}.\end{equation}
By replacing $\mathcal{A}_R$ with a semi-free resolution, we may assume that $\mathcal{A}_R$ is term-wise free over $R$. \par\indent
To put ourselves in the situation of Theorem \ref{thm:Katz theorem} and apply the results of Section 4, we record some preliminary results concerning coefficient basechange.\par\indent
\textbf{Claim 1}. Let $R$ be a ring of global dimension $n<\infty$, and $(C_\bullet,d)$ an unbounded (e.g. $2$-periodic) acyclic complex of projective $R$-modules. Then $(C_\bullet,d)$ is contractible.\par\indent
\emph{Proof of Claim 1}. Let $Z_i=\ker(d:C_i\rightarrow C_{i+1})$. Since $C_{\bullet}$ is acyclic, there are short exact sequences 
\begin{equation}\label{eq:cocycle SES}
0\rightarrow Z_i\rightarrow C_i\rightarrow Z_{i+1}\rightarrow 0,
\end{equation}
which assembles to exact sequences
\begin{equation}
0\rightarrow Z_i\rightarrow C_i\rightarrow C_{i+1}\rightarrow \cdots\rightarrow C_{i+n-1}\rightarrow Z_{i+n}\rightarrow 0.
\end{equation}
Since $R$ has global dimension $n$, $Z_{i}$ is projective by the syzygy theorem. Since $i$ is arbitrary, this implies that each short exact sequence \eqref{eq:cocycle SES} splits, and hence $(C_{\bullet},d)$ is contractible.\blackqed \par\indent
Since $R$ has finite type over $\mathbb{Z}$, its regular locus is open, and is moreover nonempty since $R$ is a domain. After replacing $R$ by a localization $R[\frac{1}{f}]$, we will from now on assume that each local ring $R_{\mathfrak{p}}$ is regular and hence of finite global dimension. 
\par\indent
\textbf{Claim 2}. After replacing $R$ with a localization $R[\frac{1}{f}]$, one can ensure that $HH^{per}_*(\mathcal{A}_R/R)$ is a finite free $R((t))$-module and that the natural maps 
\begin{equation}\label{eq:basechange of periodic cyclic homology to C}
HH^{per}_*(\mathcal{A}_R/R)\otimes_{R((t))}\mathbb{C}((t))\rightarrow HH^{per}_*(\mathcal{A}/\mathbb{C})   
\end{equation}
and 
\begin{equation}\label{eq:basechange of periodic cyclic homology to residue field}
HH^{per}_*(\mathcal{A}_R/R)\otimes_{R((t))}\kappa(\mathfrak{m})((t))\rightarrow HH^{per}_*((\mathcal{A}_R\otimes_R\kappa(\mathfrak{m}))/\kappa(\mathfrak{m}))   
\end{equation} 
are isomorphisms for any maximal ideal $\mathfrak{m}\in\mathrm{mSpec}(R)$.
 \par\indent
\emph{Proof of Claim 2}. Since $\mathcal{A}_R$ is smooth and proper over $R$, $HH_*(\mathcal{A}_R/R)$ is a finitely generated $R$-module. Since $R$ is Noetherian (it is finite type over $\mathbb{Z}$), by generic freeness we may assume, after replacing $R$ with a localization $R[\frac{1}{f}]$, that  $HH_*(\mathcal{A}_R/R)$ is a finite free $R$-module. In particular, there is a quasi-isomorphism 
\begin{equation}\label{eq:minimal model of Hochschild complex}
 \iota:  (P,0):=(HH_*(\mathcal{A}_R/R),0)\xrightarrow{\simeq} (CC_*(\mathcal{A}_R/R),b). 
\end{equation}
By transfer, there exists a unital $A_{\infty}$-$R[\Lambda]/\Lambda^2$-module structure on $P$\footnote{Explicitly, odd maps $\delta_i:P\rightarrow P, i\geq 1$ such that $\sum_{i=0}^s\delta_i\delta_{s-i}=0$ for all $s\geq 0$.} and a quasi-equivalence of $A_{\infty}$-$R[\Lambda]/\Lambda^2$-modules
\begin{equation}\label{eq:minimal model of mixed complex}
(P,0,\{\delta_i\})\xrightarrow{\simeq}(CC_*(\mathcal{A}_R/R),b,B).    
\end{equation}
In particular, \eqref{eq:minimal model of mixed complex} induces a quasi-isomorphism 
\begin{equation}\label{eq:quasi iso of t complexes}
(P((t)), d_{eq}=\sum_{i\geq 1} \delta_it^i )\xrightarrow{\simeq} (CC_*(\mathcal{A}_R/R)((t)),b+tB).  
\end{equation}
Since $P$ is term-wise finite free over $R$, $P((t))$ is term-wise finite free over $R((t))$. Since $R((t))$ is Noetherian, by generic freeness (and the fact that a Laurent series is invertible if its leading coefficient is) we deduce that there is a localization $R'=R[\frac{1}{f}]$ such that 
\begin{equation}
 H^*(CC_*((\mathcal{A}_{R}\otimes_R R')/R')((t)),b+tB)  \cong H^*(P\otimes_R R'((t)),   d_{eq})\cong H^*(P((t)),d_{eq})\otimes_{R((t))}R'((t))
\end{equation}
is finite free over $R'((t))$. This proves the first part of Claim 2. \par\indent
Now set $R=R'$. To show that \eqref{eq:basechange of periodic cyclic homology to C} is an isomorphism, note that it can be identified with the composition
\begin{equation}
  HH^{per}_*(\mathcal{A}_R/R)\otimes_{R((t))}\mathbb{C}((t))\cong H^*(P((t)),d_{eq})\otimes_{R((t))}\mathbb{C}((t))\cong H^*(P\otimes_R \mathbb{C}((t)),d_{eq})\cong HH_*^{per}(\mathcal{A}/\mathbb{C}),  
\end{equation}
where the last two isomorphisms use the flatness of $\mathbb{C}$ over $R$ (resp. of $\mathbb{C}((t))$ over $R((t))$). \par\indent
To show that \eqref{eq:basechange of periodic cyclic homology to residue field} is an isomorphism, note that it can be decomposed as a composition 
$$HH^{per}_*(\mathcal{A}_R/R)\otimes_{R((t))}\kappa(\mathfrak{m})((t))  \xrightarrow[\cong]{\eqref{eq:quasi iso of t complexes}}  H^*(P((t)),d_{eq})\otimes_{R((t))}\kappa(\mathfrak{m})((t))\xrightarrow[\cong]{(i)} $$
\begin{equation}\label{eq:chain of iso over residue field}
H^*(P\otimes_R \kappa(\mathfrak{m})((t)),d_{eq})\xrightarrow[\cong]{(ii)}HH_*^{per}\big((\mathcal{A}_R\otimes_R\kappa(\mathfrak{m}))/\kappa(\mathfrak{m})\big)
\end{equation}
To see that (i) in \eqref{eq:chain of iso over residue field} is an isomorphism, note that by our previous choice of $R$ both $P((t))$ and $H^*(P((t)),d_{eq})$ are finite free over $R((t))$, and there is a quasi-isomorphism $i:(H^*(P((t)),d_{eq}),0)\rightarrow (P((t)),d_{eq})$. Since $R_{\mathfrak{m}}((t))$ is flat over $R((t))$, $P_{\mathfrak{m}}((t))$ and $H^*(P_{\mathfrak{m}}((t)),d_{eq})$ are finite free over $R_{\mathfrak{m}}((t))$, and there is a quasi-isomorphism $i_{\mathfrak{m}}:(H^*(P_{\mathfrak{m}}((t)),d_{eq}),0)\rightarrow (P_{\mathfrak{m}}((t)),d_{eq})$. Apply \textbf{Claim 1} to the cone of $i_{\mathfrak{m}}$ (note that since $R_{\mathfrak{m}}$ is regular of finite global dimension, so is $\mathrm{R}_{\mathfrak{m}}((t))$) we obtain that $i_{\mathfrak{m}}$ is a homotopy equivalence. Therefore, applying $-\otimes_{R_{\mathfrak{m}}((t))}\kappa(\mathfrak{m})((t))$ to $i_{\mathfrak{m}}$ gives a homotopy equivalence
\begin{equation}
 \big(H^*(P_{\mathfrak{m}}((t)),d_{eq})\otimes_{R_{\mathfrak{m}}((t))}\kappa(\mathfrak{m})((t)),0\big)\rightarrow (P_{\mathfrak{m}}((t))\otimes_{R_{\mathfrak{m}}((t))}\kappa(\mathfrak{m})((t)),d_{eq}),
\end{equation}
which exactly recovers (i) after passing to cohomology.\par\indent
To see that (ii) in \eqref{eq:chain of iso over residue field} is an isomorphism, apply \textbf{Claim 1} to the cone of the quasi-isomorphism $\iota_{\mathfrak{m}}:(P_{\mathfrak{m}},0)\rightarrow (CC_*((\mathcal{A}_{R}\otimes_RR_{\mathfrak{m}})/R_{\mathfrak{m}}),b)$ obtained from \eqref{eq:minimal model of Hochschild complex} by localizing at $\mathfrak{m}$ (recall that $\mathcal{A}_R$ and $P$ are term-wise free over $R$), which implies that $\iota_{\mathfrak{m}}$ is a homotopy equivalence. Applying $-\otimes_{R_{\mathfrak{m}}}\kappa(\mathfrak{m})$ to $\iota_{\mathfrak{m}}$ gives a homotopy equivalence
\begin{equation}
\iota_{\kappa(m)}:    (P\otimes_{R}\kappa(\mathfrak{m}),0)\rightarrow (CC_*((\mathcal{A}_R\otimes_R \kappa(\mathfrak{m}))/\kappa(\mathfrak{m}),b).
\end{equation}
This further implies, by a standard induction argument involving the $t$-filtration, that there is a quasi-isomorphism
\begin{equation}
 (P\otimes_{R}\kappa(\mathfrak{m})((t)),d_{eq})\rightarrow (CC_*\big(\mathcal{A}_R\otimes_R \kappa(\mathfrak{m})/\kappa(\mathfrak{m}\big)((t)),b+tB),
\end{equation}
which implies that (ii) is an isomorphism. \blackqed\par\indent
We now assume that the global ring $R$ is chosen so that \textbf{Claim 2} holds. By Corollary 
\ref{thm: nilpotence of p curvature}, the $p$-curvature of the $t$-connection on $HH^{per}_*((\mathcal{A}_R\otimes_R\kappa(\mathfrak{m}))/\kappa(\mathfrak{m}))$ is nilpotent. Since the $t$-connection is functorial under change of coefficients, \textbf{Claim 2} then identifies that the mod $\mathfrak{m}$ reduction of the $t$-connection on $HH^{per}_*(\mathcal{A}_R/R)$ with the $t$-connection on $HH^{per}_*((\mathcal{A}_R\otimes_R\kappa(\mathfrak{m}))/\kappa(\mathfrak{m}))$. The proof of Theorem \ref{thm:main theorem restated} then follows by applying Katz's criterion (Theorem \ref{thm:Katz theorem}).\qed

\section{Further applications}
\subsection{Non-commutative Hodge filtration}
Give a pure Hodge structure $(V,F^{\bullet}V,V_{\mathbb{Q}})$ of weight $w$, the Rees construction re-packages the Hodge filtration into a $t$-lattice in $V((t))$ given by
\begin{equation}
\mathcal{H}:=\sum_{q}t^{-q}F^qV[[t]]\subset V((t)).
\end{equation}
Moreover, there is a rank one $t$-connection 
\begin{equation}
\mathcal{T}_{\frac{w}{2}}:=(\mathbb{C}((t)),\nabla^{w/2}_{\partial_t}=\partial_t-\frac{w}{2t})    
\end{equation}
such that the lattice $\mathcal{H}\subset V((t))$ is logarithmic (i.e. preserved by $\nabla_{t\partial_t}$) with respect to the connection $\nabla=\nabla^{w/2}\otimes \mathrm{id}$, cf. \cite[Section 2.1.7]{KKP}.\par\indent
When $\mathcal{A}$ is a $\mathbb{Z}$-graded dg category, it is known \cite{KKP}\cite{Shk} that the canonical $t$-connection on $HH^{per}_*(\mathcal{A})$ has a simple pole and can be moreover fully recovered from the Hodge filtration on $HH^{per}_*(\mathcal{A})$ induced from the chain level filtration by the sub-complexes $(t^nCC_*(\mathcal{A})[[t]],b+tB)$ via the above Rees construction. \par\indent
In the $\mathbb{Z}/2$-graded case, the naive $t$-filtration on $HH^{per}_*(\mathcal{A})$ is too coarse to recover the $t$-connection. Instead, the $t$-connection should be viewed as part of the data of a `non-commutative Hodge filtration'. Given Theorem \ref{thm:main theorem restated}, one can explicitly describe how an actual filtration arise from the $t$-connection following the exposition of \cite[Section 3.5]{Shk}.\par\indent
For the setup, consider $\mathcal{A}$ a smooth proper d$(\mathbb{Z}/2)$g category and further assume that the Hodge-to-de-Rham spectral sequence degenerates\footnote{In the $\mathbb{Z}$-graded case this is proved by work of Kaledin \cite{Kal1}, and to the author's knowledge is still open in the $\mathbb{Z}/2$-graded case.}. Geometrically, this means that the negative cyclic homology $\mathcal{G}^0:=HH^-_*(\mathcal{A})$ is a vector bundle over the formal disk $\mathrm{Spf}\;\mathbb{C}[[t]]$ whose generic fiber is $\mathcal{G}:=HH^{per}_*(\mathcal{A})$ and special fiber is $HH_*(\mathcal{A})$. \par\indent
By Theorem \ref{thm:main theorem restated}, $(HH^{per}_*(\mathcal{A}),\nabla_{\partial_t})$ is regular singular at $t=0$, and hence there is a Kashiwara-Malgrange $V$-filtration \cite{Kas}\cite{Mal2} on $\mathcal{G}$ by free $\mathbb{C}[[t]]$-submodules $V^{\alpha}\mathcal{G}\subset \mathcal{G}, \alpha\in \mathbb{C}$. In the most general case, the definition of the $V$-filtration requires a choice of ordering of $\mathbb{C}$; however, as we showed that the monodromy of $\nabla_{t\partial_t}$ is quasi-unipotent, all the exponents $\alpha$ are rational numbers and hence $V^{\bullet}$ is canonically defined. In fact, it is the unique exhaustive, separated and left-continuous (decreasing) filtration characterized by the following properties:
\begin{itemize}
      \item $t\cdot V^{\alpha}\mathcal{G}\subset V^{\alpha+1}\mathcal{G}$ and $\partial_t\cdot V^{\alpha}\mathcal{G}\subset V^{\alpha-1}\mathcal{G}$.
    \item $t\partial_t-\alpha$ acts nilpotently on $\mathrm{Gr}^{\alpha}_V(\mathcal{G})$
\end{itemize}
\par\indent
By our assumption, $\mathcal{G}^0\subset\mathcal{G}$ is a lattice. The triple $(\mathcal{G},\nabla,\mathcal{G}^0)$ gives rise to a finite dimensional filtered vector space defined by
\begin{equation}\label{eq:non commutative Hodge filtration}
\mathcal{H}:=\bigoplus_{\alpha\in (-1,0]}\mathrm{Gr}^{\alpha}_V(\mathcal{G})\quad,\quad F^n\mathcal{H}:=\bigoplus_{\alpha\in(-1,0]}  \mathrm{Gr}^{\alpha}_V(t^n\mathcal{G}^0),  
\end{equation}
where
\begin{equation}
 \mathrm{Gr}^{\alpha}_V(t^n\mathcal{G}^0):= \frac{(V^{\alpha}\mathcal{G}\cap t^n\mathcal{G}^0)+V^{>\alpha}\mathcal{G}}{V^{>\alpha}\mathcal{G}};
\end{equation}
\eqref{eq:non commutative Hodge filtration} (or rather a suitable $2$-periodification, cf. \cite[Section 3.5]{Shk} for more details) can be viewed as the analogue of the Hodge filtration in non-commutative setting.

\subsection{Variants of the monodromy theorem}
The study of non-commutative Hodge structures, especially the $t$-connection, associated to d$(\mathbb{Z}/2)$g categories has largely been motivated by the classical Hodge theory of isolated hypersurface singularities, cf. \cite{vSt} for a nice survey and \cite{ST} for a more detailed treatment. \par\indent
We now briefly describe the (slightly simplified) setup. Let $n\geq 2$, and take a polynomial $W\in\mathbb{C}[x_1,\cdots,x_n]$ such that
\begin{equation}\label{eq:superpotential W}
W(0,\cdots,0)=\partial_{x_1}W(0,\cdots,0)=\cdots=\partial_{x_n}W(0,\cdots,0)=0    
\end{equation}
and that the origin is the only critical point of $W$. One can view $W$ as a (germ of a) smooth function $W:\mathbb{C}^n\rightarrow\mathbb{C}$. The \emph{Milnor fiber of $W$} is given by $M_t:=W^{-1}(t)\cap B_{\epsilon}$, where $t$ is a small regular value of $W$ and $B_{\epsilon}$ is a small ball centered at the origin with no other critical points, which is known (cf. \cite{Mil}) to have the homotopy type of a wedge of $S^{n-1}$'s. The local system obtained from the middle cohomologies of Milnor fibers lying over a small punctured disk 
\begin{equation}
\bigcup_{0<|t|\ll 1} H^{n-1}(M_t;\mathbb{C})    
\end{equation}
is called the \emph{Gauss-Manin system of $W$}. \par\indent
A fundamental insight of Brieskorn \cite{Br} is that the Gauss-Manin system admits an algebraic description using the de-Rham model. In modern formulation, the algebraic Gauss-Manin system can be defined as
\begin{equation}\label{eq:algebraic Gauss Manin system}
\mathcal{G}:=\Omega^n(\mathbb{C}^n)_{(0)}[t^{-1}]/(-t^{-1}dW+d)\Omega^{n-1}(\mathbb{C}^n)_{(0)}[t^{-1}],
\end{equation}
where $\Omega^{\bullet}(\mathbb{C}^n)_{(0)}$ denotes germs of holomorphic differential forms on $\mathbb{C}^n$. It is equipped with a $D$-module structure (i.e. a module over the Weyl algebra $\mathbb{C}\langle q,\partial_q\rangle$) given by
\begin{equation}
q(\eta t^{-i})=W\eta t^{-i}-i\eta t^{-i+1}\quad,\quad \partial_q(\eta t^{-i})=\eta t^{-i-1}.    
\end{equation}
A key property is that the action of $\partial_q$ is invertible (equivalently, the obvious $\mathbb{C}[t^{-1}]$-module structure can be extended to a $\mathbb{C}[t,t^{-1}]$-module structure): an inverse is given by
\begin{equation}
\partial^{-1}_q(\eta t^{-i})=dW\wedge \eta' t^{-i}\;\;\mathrm{for}\;\; d\eta'=\eta.   
\end{equation}
From the categorical perspective, it is more natural to consider the \emph{(inverted) Fourier-Laplace dual}\footnote{The convention here is to identify the localized Weyl algebras $\mathbb{C}\langle q,\partial_q,\partial_q^{-1}\rangle$ and $\mathbb{C}\langle t,t^{-1},\partial_t\rangle$
 via $q\leftrightarrow t^2\partial_t=-\partial_{t^{-1}}, \partial_q\leftrightarrow t^{-1}$.} of the above $D$-module structure, i.e. the action of the localized Weyl algebra $\mathbb{C}\langle t,t^{-1},\partial_t\rangle$ where $t$ acts by multiplication and $\partial_t$ acts by $t^{-2}q$. Equivalently, it is the $D$-module structure induced from the $t$-connection on $\mathcal{G}$ given by
\begin{equation}\label{eq:t connection on Gauss Manin system}
\frac{d}{dt}+\frac{W}{t^2}.
\end{equation}
Let $\hat{\mathcal{G}}:=\mathcal{G}\otimes_{\mathbb{C}[t,t^{-1}]}\mathbb{C}((t))$, then there is an isomorphism (cf. \cite[Section 2]{Schu})
\begin{equation}
\hat{\mathcal{G}}\cong H^n(\Omega^{\bullet}(\mathbb{C}^n)_{(0)}((t)), td-dW\wedge).   
\end{equation}
The following theorem of Shklyarov connects the above story to the categorical $t$-connection of the d$(\mathbb{Z}/2)$g category of matrix factorizations $\mathrm{MF}(W)$.\\
\begin{thm}(\cite[Theorem 1.1]{Shk})\label{thm:Shklyarov theorem}
There is an isomorphism 
\begin{equation}
H^*(CC_*(\mathrm{MF}(W))[[t]],b+tB)\cong H^*(\Omega^{\bullet}(\mathbb{C}^n)[[t]],td-dW\wedge)    
\end{equation}
which intertwines the categorical $t$-connection $\nabla^{\mathrm{MF}(W)}_{\partial_t}$ with the connection  
\begin{equation}\label{eq:t connection on twisted de Rham}
\nabla^{W}_{\partial_t}=\partial_t+\frac{W}{t^2}+\frac{\Gamma}{t},    
\end{equation}
where $\Gamma|_{\Omega^q}=-\frac{q}{2}$. \qed
\end{thm}\par\indent
\begin{rmk}
Note that when restricted to the top forms, the formula \eqref{eq:t connection on twisted de Rham} differs from the (Fourier-Laplace dual of the) Gauss-Manin $D$-module \eqref{eq:t connection on Gauss Manin system} by the term $-\frac{n}{2t}$. This is a feature, not a bug: the number of variables $n$ is \emph{not} a categorical invariant due to Kn\"{o}rrer periodicity.
\end{rmk}
One of the central questions in singularity theory is the study of the monodromy operator on $H^{n-1}(M_t;\mathbb{C})$ associated to the family of Milnor fibers over a small punctured disk. The following is a fundamental result in this direction. \\
\begin{thm}(The Monodromy Theorem, \cite{Lan})\label{thm:monodromy theorem}
The monodromy operator $T: H^{n-1}(M_t)\rightarrow H^{n-1}(M_t)$ has eigenvalues given by roots of unity. Moreover, the size of its largest Jordan block is $\leq n$. \qed
\end{thm}
More relevant to us is the following generalization of the monodromy theorem obtained by Scherk.\\
\begin{thm} (\cite{Sche})\label{thm:Scherk thm}
If $W^{r}\in (\partial_{x_1}W,\cdots,\partial_{x_n}W)\subset \mathbb{C}[x_1,\cdots,x_n]$, then the size of the largest Jordan block of $T$ is $\leq r$. \qed     
\end{thm}
The relation between Scherk's theorem and the classical monodromy theorem lies in the following result\footnote{Remarkably, the proof of Brian\c{c}on and Skoda on this purely algebraic statement uses deep results from analysis.} of Brian\c{c}on and Skoda \cite{BS}: for any $W\in \mathbb{C}[x_1,\cdots,x_n]$ with $W(0,\cdots,0)=0$, $W^n$ is contained in the Jacobian ideal of $W$.  \par\indent
Using reduction mod $p$ and the topological interpretation of $p$-curvature in Section 4, we obtain the following non-commutative generalization of Scherk's monodromy theorem.\\
\begin{thm}\label{thm:nc Scherk}
Let $\mathcal{C}$ be a smooth proper d$(\mathbb{Z}/2)$g category over $\mathbb{C}$. Suppose the Hochschild cohomology class of its differential $[d]\in HH^*(\mathcal{C})$ satisfies $[d]^{\cup r}=0$, then the size of the largest Jordan block of the monodromy of $\nabla^{\mathcal{C}}_{\partial_t}$ is $\leq r$.   
\end{thm}
\emph{Proof}. By Lemma \ref{thm:reduction from category to algebra}, it suffices to consider the case where $\mathcal{C}=\mathcal{A}$ is a d$(\mathbb{Z}/2)$g algebra. Analogous to the proof of Theorem \ref{thm:main theorem restated} \textbf{Claim 2}, we fix a global ring $R$ and a smooth proper $R$-linear d$(\mathbb{Z}/2)$g algebra $\mathcal{A}_R$ such that $\mathcal{A}_R\otimes_R\mathbb{C}\simeq \mathcal{A}$, with the further property that $HH^*(\mathcal{A}_R/R)$ is a finite free $R$-module and the natural map
\begin{equation}
HH^*(\mathcal{A}_R/R)\otimes_R\mathbb{C}\rightarrow HH^*(\mathcal{A}/\mathbb{C})   
\end{equation}
is an isomorphism. In particular, the equality $[d]^{\cup r}=0$ already holds in $HH^*(\mathcal{\mathcal{A}_R}/R)$ and hence also in $HH^*\big((\mathcal{A}_R\otimes_R\kappa(\mathfrak{m}))/\kappa(\mathfrak{m})\big)$ for any maximal ideal $\mathfrak{m}$ of $R$. By Theorem \ref{thm: p curvature equals multiplicative operation}, this implies that for any $\mathfrak{m}$, the $p$-curvature (recall $p$ is the characteristic of $\kappa(\mathfrak{m})$) of the $t$-connection on $HH^{per}_*\big((\mathcal{A}_R\otimes_R\kappa(\mathfrak{m}))/\kappa(\mathfrak{m})\big)$ is nilpotent of order $\leq r$. The statement then follows from the second part of Theorem \ref{thm:Katz theorem}.\qed\\
\begin{rmk}\label{thm:remark about Scherk thm and FL transform}
Theorem \ref{thm:Scherk thm} can be recovered by applying Theorem \ref{thm:nc Scherk} to $\mathcal{C}=\mathrm{MF}(W)$ together with Theorem \ref{thm:Shklyarov theorem}, and the fact that there is a ring isomorphism $HH^*(MF(W))\cong\mathrm{Jac}(W)$, where $\mathrm{Jac}(W)=\mathbb{C}[x_1,\cdots,x_n]/(\partial_{x_1}W,\cdots,\partial_{x_n}W)$ denotes the Jacobian algebra of $W$, which moreover identifies the classes $[d]$ and $2[W]$\footnote{This follows, for instance, by using the fact that the pair of twisted HKR isomorphisms (cf. \cite{CT}) $\big(\mathrm{Jac}(W),\Omega^n(\mathbb{C}^n)/dW\wedge\Omega^{n-1}(\mathbb{C}^n)\big)\cong \big(HH^*(\mathrm{MF}(W)),HH_*(\mathrm{MF}(W))\big)$ form an isomorphism of a ring acting on a module, and then restricting Theorem \ref{thm:Shklyarov theorem}'s identification of $\nabla^{\mathrm{MF}(W)}_{t^2\partial_t}$ with $\nabla^W_{t^2\partial_t}$ to $t=0$.}. One subtlety is that a priori, Theorem \ref{thm:Shklyarov theorem} identifies $\nabla^{\mathrm{MF}(W)}$ with the (inverted) Fourier-Laplace dual of the Gauss-Manin system (up to a half-integer twist, which is irrelevant for the size of Jordan blocks). On the other hand, on a regular holonomic $\mathbb{C}\langle q,\partial_q\rangle$-module $\mathcal{M}_q$, the Kashiwara-Malgrange-Saito comparison theorem identifies the $e^{2\pi i\alpha}$-generalized eigenspace, for $-1< \alpha\leq 0$, of the monodromy operator $T$ with (cf. \cite[Section 5]{vSt}) 
\begin{equation}
\mathrm{Gr}^{\alpha}_V\mathcal{M}_q
\end{equation}
and moreover the logarithm of its nilpotent part is identified, up to factors of $2\pi i$, with $q\partial_q-\alpha$. When $\mathcal{M}_q$ is the Gauss-Manin $D$-module of $W$, recall that $\partial_q$ is invertible. With the identification of variables $q\leftrightarrow t^2\partial_t, \partial_q\leftrightarrow t^{-1}$ and the relation $t\partial_t=t\cdot t^{-2}q=\partial_q q=q\partial_q+1$, one sees that the $V$-filtration on the (inverted) Fourier-Laplace dual module $\mathcal{M}_t$ satisfies (by uniqueness by the $V$-filtration)
\begin{equation}
\mathrm{Gr}^{\alpha}_V\mathcal{M}_q=\mathrm{Gr}^{\alpha+1}_V\mathcal{M}_t.
\end{equation}
Moreover, multiplication by $t=\partial_q^{-1}$ induces an isomorphism $t\cdot: \mathrm{Gr}^{\alpha}_V\mathcal{M}_t\xrightarrow{\cong} \mathrm{Gr}^{\alpha+1}_V\mathcal{M}_t$, and hence the Jordan block types of the nilpotent part of the monodromy is preserved under the (inverted) Fourier-Laplace transform\footnote{The invertibility of $\partial_q$ is crucial here, otherwise the sizes of Jordan blocks associated to the eigenvalue $1$ can shift by $\pm 1$ under Fourier-Laplace transform, cf. \cite[Corollary 2.1.28]{PS1}. Indeed, for a general $D$-module $\mathcal{M}_q$, the vanishing and nearby cycles agree for $-1<\alpha<0$, while at the endpoints of $[-1,0]$ they are related by maps $U=\partial_q: \mathrm{Gr}^0_V\mathcal{M}_q\rightarrow \mathrm{Gr}^{-1}_V\mathcal{M}_q, V=q: \mathrm{Gr}^{-1}_V\mathcal{M}_q\rightarrow \mathrm{Gr}^0_V\mathcal{M}_q$, and the logarithm of the monodromy operator on the nearby cycle associated to the endpoint $1$ is identified with $VU$. Under the ordinary Fourier-Laplace transform $q\leftrightarrow -\partial_t, \partial_q\leftrightarrow t$, the roles of the relevant vanishing and nearby cycles are reversed and the logarithm of the monodromy becomes $UV$, whose sizes of Jordan blocks can differ from $VU$ by shifts of $\pm 1$. This issue is absent if $U$ or $V$ is invertible.}.
\end{rmk}
Theorem \ref{thm:main theorem restated}, together with Theorem \ref{thm:nc Scherk} applied to the case $r=1$, immediately gives the following.\\
\begin{cor}\label{thm:nc quasi homogeneous}
Let $\mathcal{C}$ be a smooth proper d$(\mathbb{Z}/2)$g category over $\mathbb{C}$. Suppose $[d]=0\in HH^*(\mathcal{C})$, then the $t$-connection on $HH^{per}_*(\mathcal{C})$ is regular singular with finite monodromy. \qed   \\
\end{cor}
\begin{rmk}
When $\mathcal{C}=\mathrm{MF}(W)$, the condition $[d]=0\in HH^*(\mathcal{A})$ can be identified with the condition $W\in (\partial_{x_1}W,\cdots,\partial_{x_n}W)$ (cf. Remark \ref{thm:remark about Scherk thm and FL transform}), which by a theorem of Saito \cite{Sai} is equivalently to $W$ being \emph{quasi-homogeneous}. Therefore, Corollary \ref{thm:nc quasi homogeneous} is a non-commutative generalization of the classical fact that a quasi-homogeneous isolated hypersurface singularity has finite monodromy (cf. \cite[p 71]{Mil}). 
\end{rmk}
In the recent work of Pomerleano-Seidel \cite{PS2}, a refinement of Theorem \ref{thm:nc Scherk} concerning the conjugacy class of the monodromy operator is obtained in the context of quantum connection (recall this arises as one sets $\mathcal{C}$ to be the Fukaya category of a closed monotone symplectic manifold) using mod $p$ methods, cf. Theorem 1.4 in loc.cit. It also has a counterpart in classical singularity theory known as Scherk's conjecture, proved by Varchenko \cite{Var}. A non-commutative counterpart of these statements, at the level of generality of Theorem \ref{thm:main theorem restated}, seems to be lacking (one issue is, for instance, in general $HH^*(\mathcal{C})$ and $HH_*^{per}(\mathcal{C})$ don't even have the same rank over $\mathbb{C}$ and $\mathbb{C}((t))$). It poses an interesting question whether such a counterpart exists when restricted to a meaningful subclass of d$(\mathbb{Z}/2)$g categories $\mathcal{C}$. \\
\begin{conj}
Let $\mathcal{C}$ be a smooth proper Calabi-Yau d$(\mathbb{Z}/2)$g category such that the non-commutative Hodge-to-de-Rham spectral sequence degenerates\footnote{Such assumptions are satisfied for $\mathcal{A}=\mathrm{MF(W)}$ where $W$ is an isolated hypersurface singularity  and $\mathcal{A}=\mathrm{Fuk}(X)$ for a large class of closed monotone symplectic manifolds $X$ \cite{Gan}.}. Then, under the identification $HH^*(\mathcal{C})\cong HH_*(\mathcal{C})$ via the Calabi-Yau structure, the nilpotent part of the monodromy of $\nabla^{\mathcal{C}}_{\partial_t}$ lies in the closure of the conjugacy class of the operator $[d]\cup-\in \mathrm{End}(HH^*(\mathcal{C}))$.
\end{conj}
Another question which we won't pursue in this paper (but think is worth further investigation) is the formulation and proof of a non-commutative version of Theorem \ref{thm:monodromy theorem}. At first sight, this is slightly mysterious as the dimension of the ambient space of an isolated hypersurface singularity $W$ is not an invariant of $\mathcal{C}=\mathrm{MF}(W)$ due to Kn\"{o}rrer periodicity. We remark that it is possible that a correct formulation requires working with a `$\mathbb{Z}/2$-graded smooth and proper deformation of a smooth $\mathbb{Z}$-graded dg category' instead of just a single smooth proper d$(\mathbb{Z}/2)$g category (for instance, this perspective is central to the work of \cite{PS1} on the quantum connection). \\
\begin{question}
Formulate and prove a non-commutative version of Theorem \ref{thm:monodromy theorem}. 
\end{question}

\subsubsection{Monodromy and categorical dimensions}
The ensuing discussion can be motivated by the following question: given a smooth proper d$(\mathbb{Z}/2)$g category $\mathcal{C}$, is there a meaningful upper bound for the order of nilpotence of $[d]\in HH^*(\mathcal{C})$ (the proof of Corollary \ref{thm: nilpotence of p curvature} shows it is finite)? In light of Theorem \ref{thm:nc Scherk}, that constant will also give an upper bound for the sizes of Jordan blocks of the monodromy of $\nabla^{\mathcal{C}}_{\partial_t}$. \par\indent
Our answer depends on various notions of categorical dimensions, which are measures of the complexity of a dg category. We now briefly recall the relevant definitions, see e.g. \cite{BF} \cite{BFL}\cite{EL} for more details. \par\indent
Let $\mathcal{T}$ be a ($\mathbb{Z}/2$-graded) triangulated category, and $\mathcal{I}\subset \mathcal{T}$ a full subcategory. Let $\langle\mathcal{I}\rangle$ denote the full subcategory of $\mathcal{T}$ whose objects are isomorphic to summands of finite coproducts of shifts of objects in $\mathcal{I}$. For two full subcategories $\mathcal{I}_1,\mathcal{I}_2$, denote by $\mathcal{I}_1*\mathcal{I}_2$ the full subcategory of objects $B$ such that there is a
distinguished triangle 
\begin{equation}
B_1\rightarrow B\rightarrow B_2 \rightarrow B_1[1]\quad, \quad B_i\in \mathcal{I}_i.   
\end{equation}
Then, inductively set
\begin{equation}
\langle\mathcal{I}\rangle_0:=\langle\mathcal{I}\rangle\quad,\quad   \langle\mathcal{I}\rangle_n:=\langle\langle\mathcal{I}\rangle_{n-1}*\langle\mathcal{I}\rangle \rangle.
\end{equation}

\begin{mydef}\label{thm:Hochschild and diagonal dimension}
Let $\mathcal{A}$ be a d$(\mathbb{Z}/2)$g algebra. The \emph{Hochschild dimension} of $\mathcal{A}$, denoted $\mathrm{Hdim}(\mathcal{A})$, is the minimum $n$ for which $\mathcal{A}_{\Delta}\in \langle \mathcal{A}^{e}:=\mathcal{A}^{op}\otimes \mathcal{A}\rangle_n\subset \mathrm{Perf}(\mathcal{A}^e)$. The \emph{diagonal dimension} of $\mathcal{A}$, denoted $\mathrm{Ddim}(\mathcal{A})$, is the minimum $n$ for which there exist $G\in \mathrm{Perf}(\mathcal{A}^{op}),F\in \mathrm{Perf}(\mathcal{A})$ such that $\mathcal{A}_{\Delta}\in \langle G\boxtimes F\rangle_n\subset \mathrm{Perf}(\mathcal{A}^e)$. \par\indent
If $\mathcal{C}$ is more generally a d$(\mathbb{Z}/2)$g category, one can define $\mathrm{Ddim}(\mathcal{C})$ to be the diagonal dimension of $\mathcal{A}=\mathrm{End}_{\mathrm{Perf}(\mathcal{C})}(G)$, for any choice of compact generator $G\in \mathrm{Perf}(\mathcal{C})$\footnote{This is well-defined by the Morita invariance of diagonal dimension, cf. \cite[Proposition 4.3]{EL}.}. 
\end{mydef}

Let $\mathcal{T}$ be a triangulated category and $G\in \mathcal{T}$ an object. Then a morphism $f: X\rightarrow Y$ is called a \emph{$G$-ghost} if $\mathrm{Hom}_{\mathcal{T}}^i(G,f)$ is zero for all $i$. The following lemma (often called the \emph{ghost lemma}) is a standard tool in obtaining lower bounds for generation time; see e.g. \cite[Lemma 2.12]{BFL} for a proof. \\
\begin{lemma}\label{thm:ghost lemma}
Let $\mathcal{T}$ be a triangulated category and $G\in \mathcal{T}$ an object. If there exists a sequence of $G$-ghosts $f_i:X_{i-1}\rightarrow X_i, 1\leq i\leq t$ such that $f_t\circ\cdots\circ f_1\neq 0$, then $X_0\notin \langle G\rangle_{t-1}$. \qed 
\end{lemma}
Recall that there is a decreasing length filtration $F^iHH^*(\mathcal{A})$ on the Hochschild cohomology of a d$(\mathbb{Z}/2)$g algebra $\mathcal{A}$, cf. \eqref{eq:length filtration on HH^*}. The following lemma follows essentially from \cite[Corollary 3.10]{KK} (which applies to more general settings), and we include a proof for completeness.\\
\begin{lemma}\label{thm:F^1HH^* is C^e ghost}
Fix $[\varphi]\in F^1HH^*(\mathcal{A})$, viewed as the class of a closed morphism $\varphi\in \mathrm{Hom}_{D(\mathcal{A}^e)}(\mathcal{A}_{\Delta},\mathcal{A}_{\Delta})$. Then $[\varphi]$ is an $\mathcal{A}^e$-ghost.     
\end{lemma}
\emph{Proof}. Let 
\begin{equation}
 B(\mathcal{A})=\bigoplus_{n\geq 0} B_n(\mathcal{A})\rightarrow \mathcal{A}_{\Delta},   
\end{equation} 
where $B_n(\mathcal{A})=\mathcal{A}\otimes \mathcal{A}[1]^{\otimes n}\otimes\mathcal{A}$,
denote the bar resolution of $\mathcal{A}_{\Delta}$ (cf. \eqref{eq:bar resolution}). Represent $\varphi$ by a closed morphism $f_{\varphi}:B(\mathcal{A})\rightarrow \mathcal{A}_{\Delta}$. By definition of the length filtration, 
\begin{equation}\label{eq:varphi vanishes on length 0 part}
f_{\varphi}|_{B_0(\mathcal{A})}=0.     
\end{equation}
Let $[f]\in \mathrm{Hom}_{D(\mathcal{A}^e)}(\mathcal{A}^e, \mathcal{A}_{\Delta})$ be the class of a closed morphism corresponding to a cohomology class $[x]\in H^*(\mathcal{A}_{\Delta})$, and consider its lift $[\tilde{f}]$ to $B_0(\mathcal{A})=\mathcal{A}^e=\mathcal{A}^{op}\otimes \mathcal{A}$ corresponding to the class $[1\otimes x]$. Then, $\varphi\circ f$ is represented by the composition
\begin{equation}
\mathcal{A}^e\xrightarrow{\tilde{f}}  B_0(\mathcal{A})\xrightarrow{\subset} B(\mathcal{A})\xrightarrow{f_{\varphi}}\mathcal{A}_{\Delta},
\end{equation}
which must be null by \eqref{eq:varphi vanishes on length 0 part}. Since $f$ is arbitrary, this proves the claim. \qed\par\indent
As an immediate corollary, we obtain the following quantitative version of Corollary \ref{thm: nilpotence of p curvature}.\\
\begin{cor}\label{thm:bounding nilpotence order of d by Hdim}
Let $\mathcal{A}$ be a smooth (so that $\mathrm{Hdim}(\mathcal{A})<\infty$) d$(\mathbb{Z}/2)$g algebra. Then
\begin{equation}
[d]^{\cup (\mathrm{Hdim}(\mathcal{A})+1)}=0 \in HH^*(\mathcal{A}).     
\end{equation}
\end{cor}
\emph{Proof}. Since $[d]\in F^1HH^*(\mathcal{A})$, Lemma \ref{thm:F^1HH^* is C^e ghost} implies that it is an $\mathcal{A}^e$-ghost when viewed as an element $[d]\in\mathrm{Hom}_{D(\mathcal{A}^e)}(\mathcal{A}_{\Delta},\mathcal{A}_{\Delta})$. The statement then follows by applying Lemma \ref{thm:ghost lemma} to the sequence $f_i=[d], 1\leq i\leq \mathrm{Hdim}(\mathcal{A})+1$. \qed\\
\begin{thm}\label{thm:bounding Jordan blocks by Ddim}
Let $\mathcal{C}$ be a smooth proper d$(\mathbb{Z}/2)$g category over $\mathbb{C}$. Then the size of the largest Jordan block of the monodromy of $\nabla^{\mathcal{C}}_{\partial_t}$ is $\leq \mathrm{Ddim}(\mathcal{C})+1$.     
\end{thm}
\emph{Proof}. By the Morita invariance of categorical $t$-connection (cf. Lemma \ref{thm:reduction from category to algebra}) and diagonal dimension, it suffices to prove the theorem for $\mathcal{C}=\mathrm{Perf}(\mathcal{A})$ for a smooth proper d$(\mathbb{Z}/2)$g algebra $\mathcal{A}$. Furthermore, by Theorem \ref{thm:nc Scherk} and Corollary \ref{thm:bounding nilpotence order of d by Hdim}, it suffices to find one compact generator $X\in \mathrm{Perf}(\mathcal{A})$ such that $\mathcal{B}=\mathrm{End}_{\mathrm{Perf}(\mathcal{A
})}(X)$ satisfies $\mathrm{Ddim}(\mathcal{B})\geq \mathrm{Hdim}(\mathcal{B})$\footnote{We remark that the Hochschild dimension is in general not derived Morita invariant.}.\par\indent
To see this, let $n=\mathrm{Ddim}(\mathcal{A})$, and fix $G\in \mathrm{Perf}(\mathcal{A}^{op}), F\in \mathrm{Perf}(\mathcal{A})$ such that $\mathcal{A}_{\Delta}\in \langle G\boxtimes F\rangle_n$. Let $G^{\vee}:=\mathrm{Hom}_{\mathrm{Perf}(\mathcal{A}^{op})}(G,\mathcal{A})\in \mathrm{Perf}(\mathcal{A})$, and consider the compact generator
\begin{equation}
X:=\mathcal{A}\oplus G^{\vee}\oplus F\in \mathrm{Perf}(\mathcal{A}).     
\end{equation}
Since $G$ (resp. $F$) is a direct summand of $X^{\vee}$ (resp. $X$), $G\boxtimes F\in \langle X^{\vee}\boxtimes X\rangle_0$ and hence $\mathcal{A}_{\Delta}\in \langle X^{\vee}\boxtimes X\rangle_n$. Let $\mathcal{B}$ be the endomorphism d$(\mathbb{Z}/2)$g algebra $\mathrm{End}_{\mathrm{Perf}(\mathcal{A})}(X)$ and consider the quasi-equivalence
\begin{equation}
 \mathrm{Perf}(\mathcal{A}^e)\rightarrow     \mathrm{Perf}(\mathcal{B}^e) 
\end{equation}
given by
\begin{equation}
N\mapsto   X \otimes_{\mathcal{A}}N\otimes_{\mathcal{A}} X^{\vee},
\end{equation}
which sends $\mathcal{A}_{\Delta}$ to $\mathcal{B}_{\Delta}$ and $X^{\vee}\boxtimes X$ to $(X\otimes_{\mathcal{A}} X^{\vee})\boxtimes (X\otimes_{\mathcal{A}} X^{\vee})\simeq \mathcal{B}^e$. Hence
\begin{equation}
\mathcal{B}_{\Delta}\in \langle \mathcal{B}^e\rangle_n,    
\end{equation}
which completes the proof. \qed\par\indent
Applying Theorem \ref{thm:bounding Jordan blocks by Ddim} to $\mathcal{C}=\mathrm{MF}(W)$ (combined with Theorem \ref{thm:Shklyarov theorem}) gives the following result.\\
\begin{cor}\label{thm:bounding Jordan blocks of Milnor fiber by Ddim}
Let $W$ be as in \eqref{eq:superpotential W}, and $M_t$ its Milnor fiber. Then the size of the largest Jordan block of the monodromy operator $T:H^{n-1}(M_t)\rightarrow H^{n-1}(M_t)$ is $\leq\mathrm{Ddim}(\mathrm{MF}(W))+1$.\qed    
\end{cor}
One may interpret Corollary \ref{thm:bounding Jordan blocks of Milnor fiber by Ddim} as giving a topological lower bound of the diagonal dimension of $\mathrm{MF}(W)$, which to the author's knowledge is new (though it can only be nontrivial when $W$ is non-quasi-homogeneous). \par\indent
To conclude this section, we consider the special case where $\mathcal{C}$ is \emph{weak smooth Calabi-Yau category}, i.e. there exists an element $\Omega\in HH_\nu(\mathcal{C})=\mathrm{Hom}_{D(\mathcal{C}^e)}(\mathcal{C}^!,\mathcal{C}_{\Delta}[\nu])$ (called the \emph{volume form}) such that the corresponding morphism \footnote{Here $\mathcal{C}^!:=\mathrm{Hom}_{D(\mathcal{C}^e)}(\mathcal{C}_{\Delta},\mathcal{C}^e)\in \mathcal{D}(\mathcal{C}^e)$ denotes the bimodule dual of $\mathcal{C}_{\Delta}$. }
\begin{equation}
\Theta_{\Omega}:\mathcal{C}^!\rightarrow \mathcal{C}_{\Delta}[\nu]    
\end{equation}
is a quasi-isomorphism (for some $\nu\in\mathbb{Z}/2$); this in particular implies that $-\cap\Omega: HH^*(\mathcal{C})\rightarrow HH_{*+\nu}(\mathcal{C})$ is an isomorphism. Recall that there is a filtration 
\begin{equation}
CC^{\leq n}_*(\mathcal{C}):=B_{\leq n}(\mathcal{C})\otimes_{\mathcal{C}^e}\mathcal{C}_{\Delta}\subset CC_*(\mathcal{C})    
\end{equation} 
on Hochschild chains by length of tensor factors, which induces a filtration of $HH_*(\mathcal{C})$. We define the \emph{length} of a volume form $\Omega$, denoted $\mathrm{length}(\Omega)$, to be the smallest non-negative integer $n$ such that $\Omega\in HH^{\leq n}_*(\mathcal{C})$. The next lemma shows the length of volume form is in fact an invariant of $\mathcal{C}$ itself.\\
\begin{lemma}\label{thm:invariance of length}
Let $\Omega,\Omega'$ be two choices of volume forms for a weak smooth Calabi-Yau $\mathcal{C}$, then $\mathrm{length}(\Omega)=\mathrm{length}(\Omega')$.    
\end{lemma}
\emph{Proof}. Since $\Omega$ is a volume form, there exists a Hochschild cohomology class $\phi$ such that $\phi\cap \Omega=\Omega'$. Since cap product with a Hochschild cochain cannot increase length, $\mathrm{length}(\Omega)\geq \mathrm{length}(\Omega')$. Thus by symmetry, $\mathrm{length}(\Omega)=\mathrm{length}(\Omega')$.\qed\par\indent
Let $\mathcal{C}$ be a weak smooth Calabi-Yau d$(\mathbb{Z}/2)$g category. Define the \emph{length} of $\mathcal{C}$, denoted $\mathrm{length}(\mathcal{C})$, to be the length of any choice of volume form (which is well-defined by Lemma \ref{thm:invariance of length}).   \\
\begin{prop}\label{thm:bounding nilpotence of d by CY length}
Let $\mathcal{A}$ be a weak smooth and proper Calabi-Yau d$(\mathbb{Z}/2)$g algebra. Then
\begin{equation}
\mathrm{Ddim}(\mathcal{A})\leq \mathrm{length}(\mathcal{A}). 
\end{equation}
\end{prop}
\emph{Proof}. Let $N:=\mathrm{length}(\mathcal{A})$, and $z\in CC_\nu^{\leq N}(\mathcal{A})=\mathrm{Hom}_{D(\mathcal{A}^e)}(B_{\leq N}(\mathcal{A})^!,\mathcal{A}[\nu])$ be a cycle representing a volume form $\Omega$. Let 
\begin{equation}
\Theta_z: B_{\leq N}(\mathcal{A})^!\rightarrow \mathcal{A}_{\Delta}[\nu]    
\end{equation}
be the corresponding closed morphism. Then $\Theta_\Omega: \mathcal{A}^!\rightarrow \mathcal{A}_{\Delta}[\nu]$ factors through $\Theta_z$, which (as $\Theta_{\Omega}$ is a quasi-isomorphism) establishes $\mathcal{A}_{\Delta}$ as a shifted idempotent summand of $B_{\leq N}(\mathcal{A})^!$. By construction, $B_{\leq N}(\mathcal{A})$ and hence $B_{\leq N}(\mathcal{A})^!$ can be generated from $\mathcal{A}^e$ by at most $N$ cones (properness guarantees that, up to replacing $\mathcal{A}$ with a quasi-equivalent model, each term in its bar complex is a finite free $\mathcal{A}^e$ module). Therefore, $\mathcal{A}_{\Delta}\in \langle \mathcal{A}^e=\mathcal{A}^{op}\boxtimes\mathcal{A}\rangle_N$, which completes the proof.\qed \par\indent
Given a weak smooth Calabi-Yau d$(\mathbb{Z}/2)$g category $\mathcal{C}$, there exists a compact generator $G$ of $\mathrm{Perf}(\mathcal{C})$ such that $\mathcal{A}:=\mathrm{End}_{\mathrm{Perf}(\mathcal{C})}(G)$ satisfies $\mathrm{length}(\mathcal{A})\leq \mathrm{length}(\mathcal{C})$. This observation combined with Theorem \ref{thm:bounding Jordan blocks by Ddim} and Proposition \ref{thm:bounding nilpotence of d by CY length} implies that:\\
\begin{cor}
Let $\mathcal{C}$ be a weak smooth proper Calabi-Yau d$(\mathbb{Z}/2)$g category. Then the size of the largest Jordan block of the monodromy of $\nabla^{\mathcal{C}}_{\partial_t}$ is $\leq \mathrm{length}(\mathcal{C})+1$. \qed    
\end{cor}

\end{document}